%% file: main_arxiv.tex
\documentclass[]{article}

\input{preamble}

\usepackage{cite}
\usepackage{graphicx}
\usepackage[usenames,dvipsnames]{xcolor}
\usepackage{caption}
\usepackage{subcaption}
\numberwithin{equation}{section}
\usepackage[a4paper, text={14cm,23cm}]{geometry}

\usepackage[]{microtype}
\usepackage{paralist} 

\usepackage[compact]{titlesec}
\titlespacing{\section}{0pt}{2ex}{1ex}
\titlespacing{\subsection}{0pt}{1ex}{0.5ex}
\titlespacing{\subsubsection}{0pt}{1ex}{0ex}
\titlespacing{\paragraph}{0pt}{0.5ex}{1ex}

\usepackage{hyperref}
\hypersetup{
  colorlinks,
  citecolor=Red,
  linkcolor=Blue,
  urlcolor=Violet}

\newcommand\blfootnote[1]{%
  \begingroup
  \renewcommand\thefootnote{}\footnote{#1}%
  \addtocounter{footnote}{-1}%
  \endgroup
}

\usepackage{tikz} 
\usetikzlibrary{shapes,arrows,calc,matrix,positioning,decorations.text,decorations.pathreplacing} 
\usepackage{xcolor}

\usepackage[final]{changes}
\definechangesauthor[name={r1}, color=orange!180]{r1}
\definechangesauthor[name={r2}, color=violet]{r2}
\setremarkmarkup{$^{(#2)}$}

\ifpdf
  \DeclareGraphicsExtensions{.eps,.pdf,.png,.jpg}
\else
  \DeclareGraphicsExtensions{.eps}
\fi

\newcommand{\f}{f} 
\newcommand{\leg}{Q} 
\newcommand{\ens}[1]{\mathscr{#1}} 
\newcommand{\bbg}{\bar{\mbf{g}}} 
\newcommand{\id}{\mbf{Id}} 

\newcommand{\eind}{j} 
\newcommand{\Eind}{J} 

\newcommand{\match}{\mcl{P}} 
\newcommand{\res}{\mcl{R}} 
\newcommand{\ext}{\mcl{E}} 
\newcommand{\Exp}{\mbb{E}} 
\newcommand{\I}{\mcl{I}_{\f}} 

\newcommand{\opnorm}[1]{{\vert\kern-0.25ex\vert\kern-0.25ex\vert #1
    \vert\kern-0.25ex\vert\kern-0.25ex\vert}}

\newcommand{\pds}{\msr{P}} 
\newcommand{\pms}{\msr{M}_{1}} 

\newcommand{\wei}{\mrm{We}} 

\newcommand{\p}{\rh} 
\newcommand{\prior}{\pi} 
\newcommand{\targ}{\Ps} 

\renewcommand{\Pr}{\mbb{P}} 

\title{A micro-macro acceleration method for the Monte Carlo simulation of stochastic differential equations\footnote{
Manuscript submitted to \textit{SIAM Journal on Numerical Analysis}. The research of the second and third authors was partially supported by the Research Council of the University of Leuven through grant OT/13/66, by the Interuniversity Attraction Poles Programme of the Belgian Science Policy Office under grant IUAP/V/22, and by the Research Foundation -- Flanders (FWO -- Vlaanderen) under the grant G.A003.13.}}

\usepackage{authblk}
\author[1]{Kristian Debrabant}
\author[2]{Giovanni Samaey}
\author[2\,$\dagger$]{Przemys{\l}aw Zieli\'{n}ski}
\affil[1]{{\small Department of Mathematics and Computer Science, University of Southern Denmark}}
\affil[2]{{\small Department of Computer Science, University of Leuven}}

\date{\today}

\begin{document}
\maketitle

\begin{abstract}
We present and analyse a micro-macro acceleration method for the Monte Carlo simulation of stochastic differential equations with separation between the (fast) time-scale of individual trajectories and the (slow) time-scale of the macroscopic function of interest. The algorithm combines short bursts of path simulations with extrapolation of a number of macroscopic state variables forward in time. The new microscopic state, consistent with the extrapolated variables, is obtained by a matching operator that minimises the perturbation caused by the extrapolation. We provide a proof of the convergence of this method, in the absence of statistical error, and we analyse various strategies for matching, as an operator on probability measures. Finally, we present numerical experiments that illustrate the effects of the different approximations on the resulting error in macroscopic predictions.
\blfootnote{$^\dagger$Email: \href{mailto:przemyslaw.zielinski@kuleuven.be}{przemyslaw.zielinski@kuleuven.be}}
\end{abstract}

\section{Introduction}

In many applications, one considers a process modelled with a stochastic differential equation (SDE), while the ultimate concern is the evolution of the expectation of a certain function of interest.
For this type of problem, one often resorts to Monte Carlo simulation~\cite{Caflisch1998}, i.e., the simulation of a large ensemble of realisations of the SDE, combined with ensemble averaging to obtain an approximation of the quantity of interest at the desired moments in time.
In this manuscript we present and analyse a \emph{micro-macro acceleration technique} for the Monte Carlo simulation of SDEs, motivated by the development of generic multi-scale techniques, such as \emph{equation-free} \cite{KevGeaHymKevRunThe2003,KevSam2009} and \emph{heterogeneous multi-scale} methods~\cite{EEnq2003,EEnqLiRenVan2007}.

We consider an equation
\begin{equation}\label{eq:sde}
\der{\mbf{X}}(t) = \mbf{a}(t,\mbf{X}(t))\der{t} + \mbf{b}(t,\mbf{X}(t))\star\der{\mbf{W}}(t),\quad t\in[0,T],
\end{equation}
in which $\mbf{a}\from[0,T]\times G\to\mbb{R}^d$ is the drift vector, $\mbf{b}\from[0,T]\times G\to\mbb{R}^{d\times m}$ is the dispersion matrix, $G\subseteq\mbb{R}^d$ is open, and $\mbf{W}$ is an $m$-dimensional Wiener process. As usual, \eqref{eq:sde} is an abbreviation of the integral form
\begin{equation}\label{eq:sde_int}
\mbf{X}(t) = \mbf{X}(0) + \int_0^t\mbf{a}(s,\mbf{X}(s))\,\der{s} + \int_0^t\mbf{b}(s,\mbf{X}(s))\star\der{\mbf{W}}(s),\quad t\in[0,T].
\end{equation}
The integral with respect to $\mbf{W}$ can be interpreted either as an It\^{o} integral with $\star\der{\mbf{W}} = \der{\mbf{W}}$ or as a Stratonovich integral with $\star\der{\mbf{W}} = \circ\der{\mbf{W}}$. Equations~\eqref{eq:sde} rsp.~\eqref{eq:sde_int} are solved for given $\mbf{X}(0)$ independent of $\mbf{W}$. The function of interest for the Monte Carlo simulation is defined as the expectation $\Exp$ of a continuous function $\mbf{g}\from G\to\mbb{R}^{d'}$ at time $t\in [0,T]$ via
\begin{equation}\label{eq:avg}
t\mapsto \bbg(t)=\Exp \mbf{g}(\mbf{X}(t)).
\end{equation}

Numerous methods exist to increase the efficiency of Monte Carlo simulation of SDEs. Let us mention only weak explicit \cite{Milstein1995,KloPla1999,komori07wso,roessler07sor,debrabant09foe,debrabant10rkm,abdulle13msa} and implicit \cite{KloPla1999,komori08wfo,debrabant09ddi,debrabant08bao,debrabant11bao,amiri15aco} higher order schemes, which can increase the time step but might suffer from instability; various extrapolation methods~\cite{TalTub1990} to obtain the precision of higher order from low-order schemes; and variance reduction techniques \cite{newton94vrf,debrabant15ota}, including the multilevel Monte Carlo method~\cite{giles08imm,giles08mmc}.

Our main interest lies in systems with a separation between a (fast) time-scale, on which individual trajectories of the SDE~\eqref{eq:sde} need to be simulated, and a~(slow) time-scale, on which the function of interest~$\bar{\mbf{g}}$ evolves. The technique we introduce serves specifically to increase the time step for such stiff systems beyond step-size for which direct time discretisation becomes unstable. Our approach can be augmented with a~variance reduction method or a~higher order scheme to yield gains in computational efficiency and precision.

Due to possibly high computational cost, a large class of methods bypass Monte Carlo simulation, replacing it by an analytically derived approximate macroscopic model, which consists of a number of evolution equations for the macroscopic state variables, complemented with a constitutive equation for the observable of interest (as a function of these variables).
This approach is particularly popular in the micro-macro simulation of dilute solutions of polymers~\cite{LasOtt1993}, the motivating example in this paper, where an SDE models the evolution of the configuration of each individual polymer driven by the flow field, and the observable is a non-Newtonian stress tensor. See \cite{Keunings1997,LieHalJauKeuLeg1998} for derivations of macroscopic closures for FENE dumbbell models in polymeric flow.
In~\cite{IlgKarOtt2002} the authors propose a \emph{quasi-equilibrium} approach (based on thermodynamical considerations) that is, in principle, applicable to general SDEs with additive noise. Several algorithms exist for simulating the evolution of this model numerically~\cite{SamLelLeg2011,Wang:2008}.
In contrast with the numerical closures relying on the assumption that a closed model exists in terms of the macroscopic state variables, the micro-macro acceleration method only uses these variables for computational purposes, and maintains weak convergence to the full microscopic dynamics.

To describe shortly the algorithm let us define the \emph{microscopic level} via an ensemble $\ens{X} = (\mbf{X}_\eind)_{\eind=1}^\Eind$ of $\Eind$ realisations evolving according to \eqref{eq:sde} and the \emph{macroscopic level} via a vector of $L$ macroscopic state variables $\mbf{m}= (m_1,\ldots,m_L)$, 
corresponding to expected values of some appropriately chosen functions $R_1,\ldots,R_L$. The method exploits a separation in time scales by combining short bursts of microscopic simulation using SDE \eqref{eq:sde} with a macroscopic extrapolation step, in which only the macroscopic state $\mbf{m}$ is extrapolated forward in time. \added[id=r2, remark={1}]{To connect two levels of description, we introduce a \emph{restriction operator} that computes the macroscopic state variables by averaging,}
\begin{equation*}
\res_L\ens{X} = \Big(\sum_{\eind=1}^{J}R_1(\mbf{X}_\eind),\ldots,\sum_{\eind=1}^{J}R_L(\mbf{X}_\eind)\Big),
\end{equation*}
\added[id=r2, remark={1}]{and a \emph{matching operator} $\match_L(\mbf{m},\ens{X}^{\pi})$ that alters a \emph{prior} ensemble $\ens{X}^{\pi}$ to make it consistent with a given set of macroscopic state variables $\mbf{m}$, that is $\res_L\match_L(\mbf{m},\ens{X}^{\pi})=\mbf{m}$. A more precise definition of these operators will be given in Section~\ref{sec:mM_acc}. With these operators,} one time step of the algorithm includes four stages  \added[id=r2, remark={1}]{(see also Figure~\ref{fig:mic-mac_method})}: (i)~microscopic \emph{simulation} of~the ensemble $\ens{X}$ using SDE~\eqref{eq:sde}; (ii)~\emph{restriction}, i.\,e., extraction of an estimate of the macroscopic states (or macroscopic time derivative) based on simulation in the first stage; (iii)~forward in time \emph{extrapolation} of the macroscopic state; and (iv)~\emph{matching} of the ensemble that was available at the end of the microscopic simulation with the extrapolated macroscopic state.  Here, the focus is precisely this matching step.

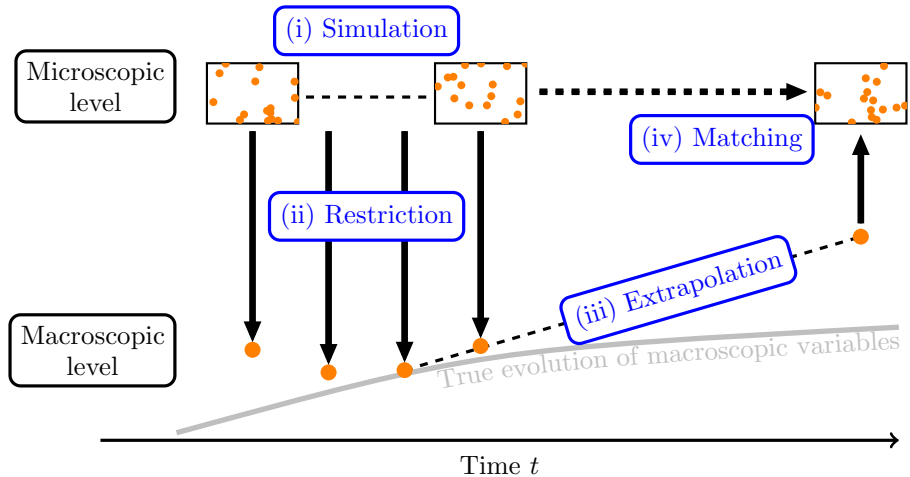
\begin{figure}[t]
\centering
    \textbf{One step of micro-macro acceleration algorithm}\par\medskip
\input{diagram}
\caption{One step of the micro-macro method starts with a given initial ensemble of realisations (top left). First, we propagate the ensemble for a few microscopic steps~(i).  Then, we compute the macroscopic variables corresponding to the obtained ensembles by averaging~(ii), after which we extrapolate the macroscopic states forward in time~(iii). Finally, to reinitialise the microscopic simulation for the next step, we match the last available ensemble with the extrapolated state~(iv).}
\label{fig:mic-mac_method}
\end{figure}

\begin{samepage}
The main contributions of the present paper are the following:
\begin{itemize}
	\item From a theoretical viewpoint, we propose and analyse a general framework for performing the matching step based on a \emph{matching operator} that transforms a given prior distribution into a distribution consistent with a number of prescribed (extrapolated) macroscopic state variables, while introducing a ``minimal perturbation" with respect to the prior. We make precise the notion of a ``minimal perturbation'' using different (semi-)distances between probability measures.
	\item From a numerical analysis viewpoint, we discuss how the resulting error depends on: the number $L$ of macroscopic state variables, the microscopic time discretisation error, and the extrapolation time step. We prove the convergence of the micro-macro acceleration method in the absence of statistical error, i.e.,~without discretisation in probability space, requiring only some general properties of the matching operator.  We additionally show that these consistency properties are fulfilled for a specific matching operator based on the $L_2$-norm distance.
	\item From a practical viewpoint, we provide numerical results for a nontrivial test case originating from the micro-macro simulation of dilute polymers. We discuss how to implement various matching operators for ensembles of finite size, and illustrate the interplay between the different sources of numerical~error.
\end{itemize}
\end{samepage}

Section~\ref{sec:model} gives the precise mathematical setting. Section~\ref{sec:mM_acc} describes the micro-macro acceleration method. Section~\ref{sec:match_oper} details some examples and analysis of matching operators. In Section~\ref{sec:convergence-proof} we provide the proof of convergence of the micro-macro method. Numerical implementation of matching and experiments illustrating accuracy and performance are given in Sections~\ref{sec:num_match} and~\ref{sec:numerics} while Section~\ref{sec:concl} concludes with an~outlook to future work.

\section{Mathematical setting}\label{sec:model}

\subsection{Preliminaries}
Throughout the paper we denote by: $\mbb{N}=\{1,2,\ldots\}$ the set of natural numbers; $\mbb{N}_0=\{0\}\cup\mbb{N}$ the set of non-negative integers; $\mbb{R}^d$, for $d\in\mbb{N}$, the $d$-dimensional Euclidean space with $2$-norm $\|\cdot\|$; $\mbb{R}=\mbb{R}^1$. For any multi-index $\al=(\al_1,\ldots,\al_d)\in\mbb{N}_0^d$ and $\mbf{x}=(x_1,\ldots,x_d)\in\R^d$, $\pder[\mbf{x}]^{\al}$ stands for the partial derivative $\pder[x_1]^{\al_1}\ldots\pder[x_d]^{\al_d}$ of order $|\al|=\al_1+\ldots+\al_d$.

\subsubsection{Notations and assumptions on the SDE}\label{sec:not}
We consider SDE~\eqref{eq:sde} on the time interval $I=[0,T]$, with $T>0$. We will need to compare solutions starting from different initial conditions at different moments in time.  To keep track of these solutions, we denote by $\mbf{X}(\,\cdot\,;t,\mbf{Z})$ the solution of the auxiliary problem
\begin{equation}\label{eq:sde_aux}
	\mbf{X}(s;t,\mbf{Z}) = \mbf{Z} + \int_{t}^{s}\mbf{a}(u,\mbf{X}(u;t,\mbf{Z}))\der{u} + \int_{t}^{s}\mbf{b}(u,\mbf{X}(u;t,\mbf{Z}))\star\der{\mbf{W}}(u)
\end{equation}
on the interval $[t,T]$, with $0\leq t\leq T$ and $\mbf{Z}$ a given random variable independent of~$\{\mbf{W}(s)\}_{s\geq t}$. \added[id=r1,remark={2}]{We will always assume, without explicitly mentioning, that the initial random variable~$\mbf{Z}$ is \emph{viable} in~$G$, that is $\mbf{Z}\in G$ almost surely, and that the algebraic moments $\Exp\|\mbf{Z}\|^r=\Exp[\|\mbf{Z}\|^r]$ exist for all $r\in\mbb{N}$.}
\begin{assumption}\label{ass:sde_exist}
\added[id=r1,remark={2}]{For every viable initial condition $\mbf{Z}$, the set $G$ is \emph{invariant} under~\eqref{eq:sde_aux} \cite[Def.~2.2]{PraFra2004}, that is, equation~\eqref{eq:sde_aux} has a unique strong solution $\mbf{X}(\,\cdot\,;t,\mbf{Z})$ such that for every $s\in[t,T]$, $\mbf{X}(s;t,\mbf{Z})\in G$ almost surely.}
\end{assumption}
\added[id=r1,remark={2}]{
Assumption~\ref{ass:sde_exist} guarantees that the solutions are confined in the domain $G$, which is not obvious whenever $G$ is a proper subset of $\R^d$. In practice, such behaviour is related to: (i)~degeneration of the diffusion $\mbf{b}$ on the boundary of $G$, or (ii)~the repulsive character of the drift vector $\mbf{a}$ close to the boundary of $G$. For (i), see for example~\cite{PraFra2001,PraFra2007}, and for (ii), the FENE model from Section~\ref{sec:numerics}.}

We are interested in the evolution of functionals of the form~\eqref{eq:avg}, for specific functions of interest $g\from G\to\mbb{R}$. We first define the appropriate function class suitable for our analysis. Let $C_P^r(G, \mbb{R})$ denote the space of all functions\footnote{When writing $C^r(G,\mbb{R})$, we mean the space of functions from  $G$ to $\mbb{R}$ which have continuous partial derivatives up to order $r\geq0$.} $g\in C^r(G,\R)$ that can be (together with all their partial derivatives) polynomially bounded, i.e., for which there exist constants $C>0$ and $\ka>0$ such that $|\pder[\mbf{x}]^{\al} g(\mbf{x})|\leq C(1+\|\mbf{x}\|^{2\ka})$ for all $|\al|\leq r$ and $\mbf{x}\in G$. Further, we write $g\in C_P^{q,r}(I\times G, \mbb{R})$ if $g(\cdot,\mbf{x}) \in C^{q}(I,\mbb{R})$, $g(t,\cdot) \in C^r(G, \mbb{R})$ for all $t \in I$, $\mbf{x} \in G$, and $|\pder[t]^i\pder[\mbf{x}]^{\al}g(t,\mbf{x})|\leq C(1+\|\mbf{x}\|^{2\ka})$ holds for all $0\leq i\leq q$, $|\al|\leq r$ and $\mbf{x}\in G$, uniformly with respect to $t\in I$ (cf.\ \cite[p. 153]{KloPla1999},\cite[Def. 8.1., p. 102]{Milstein1995}).

We are now ready to state the main assumptions on the class of SDEs that we consider.
\begin{assumption}\label{ass:sde}
Let $g\in C_{P}^{2(p+1)}(G,\mbb{R})$ with some $p\in\mbb{N}_0$.
\begin{enumerate}[a.]
\item For every $t\in[0,T)$ the function $(s,\mbf{z})\mapsto\Exp g\big(\mbf{X}(s;t,\mbf{z})\big)$ belongs to the space $C_{P}^{p+1,0}([t,T]\times G,\mbb{R})$ with constants $C$ and $\ka$ uniform also with respect to $t$.\label{ass:sde_forw}
\item For every $s\in(0,T]$ the function $(t,\mbf{z})\mapsto\Exp g\big(\mbf{X}(s;t,\mbf{z})\big)$ belongs to the space $C_{P}^{0,2(p+1)}([0,s]\times G,\mbb{R})$ with constants $C$ and $\ka$ uniform also with respect to $s$.\label{ass:sde_back}
\item There is a constant $C_{g,T}>0$ such that for every $0\leq t\leq s\leq T$ and any initial random variables $\mbf{Z}_1,\mbf{Z}_2$:\label{ass:sde_init}
\[
\big|\Exp g\big(\mbf{X}(s;t,\mbf{Z}_1)\big) - \Exp g\big(\mbf{X}(s;t,\mbf{Z}_2)\big)\big|\leq C_{g,T}|\Exp g(\mbf{Z}_1) - \Exp g(\mbf{Z}_2)|.
\]
\end{enumerate}
\end{assumption}
Assumption~\ref{ass:sde}.\ref{ass:sde_forw} corresponds to the existence of stochastic Taylor expansion (w.r.t.~increments of the time) of the expectation of the functions of interest, which is essential for constructing extrapolation with appropriate order of consistency (cf.~Lemma~\ref{lem:extrap}). For the proof of convergence of such expansions and the analysis of truncation error see~\cite{Rossler2010}. Assumption~\ref{ass:sde}.\ref{ass:sde_back} ensures that the functions of interest remain polynomially bounded under the evolution of the SDE~\eqref{eq:sde} and is satisfied, for instance, when $\mbf{a} \in C_P^{2(p+1)}(\mbb{R}^d, \mbb{R}^d)$  and $\mbf{b} \in C_P^{2(p+1)}(\mbb{R}^d,\mbb{R}^{d\times m})$ \cite[Thm.~4.8.6, p.~153]{KloPla1999}. Finally, we can look at Assumption~\ref{ass:sde}.\ref{ass:sde_init} as a particular kind of weak continuous dependence on the initial condition.

We finish this section with one immediate consequence of Assumption~\ref{ass:sde}.\ref{ass:sde_forw}.
\begin{lemma}
For every $g\in C_{P}^{2}(G,\mbb{R})$ there are constants $C_{g,T}$ and $\ka_{g,T}$ such that for all $0\leq t\leq s\leq T$ and any initial random variable $\mbf{Z}$ it holds
\[
\big|\Exp g\big(\mbf{X}(s;t,\mbf{Z})\big)\big|\leq C_{g,T}(1+\Exp\|\mbf{Z}\|^{2\ka_{g,T}}).
\]
In particular, all moments $\Exp\|\mbf{X}(s;t,\mbf{Z})\|^{2r}$ remain uniformly bounded with respect to $t$ and $s$.
\end{lemma}
\begin{proof}
Let us fix $t\in[0,T)$ and denote $v(s,\mbf{z})=\Exp g\big(\mbf{X}(s;t,\mbf{z})\big)$, $(s,\mbf{z})\in[t,T]\times G$. If $C$ and $\ka$ are the constants from Assumption~\ref{ass:sde}.\ref{ass:sde_forw}, we get employing mean value theorem
\[
|v(s,\mbf{z}) - v(t,\mbf{z})|\leq \max_{u\in[t,s]}|\pder[u]v(u,\mbf{z})|\cdot(s-t)\leq C(1+\|\mbf{z}\|^{2\ka})\cdot(s-t).
\]
Hence, $|\Exp v(s,\mbf{Z}) - \Exp v(t,\mbf{Z})|\leq C(1+\Exp\|\mbf{Z}\|^{2\ka})\cdot(s-t)$, by monotonicity of expected value, and from the definition of $v$ we obtain
\begin{equation}\label{eq:sde_wcont}
\big|\Exp g\big(\mbf{X}(s;t,\mbf{Z}) - \Exp g\big(\mbf{Z})\big|\leq C(1+\Exp\|\mbf{Z}\|^{2\ka})\cdot(s-t),
\end{equation}
where $C$ and $\ka$ depend only on $g$ and $T$. Consequently we have
\begin{align*}
\big|\Exp g\big(\mbf{X}(s;t,\mbf{Z})\big)\big|
&\leq \big|\Exp g\big(\mbf{X}(s;t,\mbf{Z}) - \Exp g\big(\mbf{Z})\big| + \big|\Exp g\big(\mbf{Z})\big|\\
&\leq CT(1+\Exp\|\mbf{Z}\|^{2\ka}) + \otilde{C}(1+\Exp\|\mbf{Z}\|^{2\otilde{\ka}})\leq C_{g,T}(1+\Exp\|\mbf{Z}\|^{2\ka_{g,T}}),
\end{align*}
where $\otilde{C}$, $\otilde{\ka}$ come from the definition of the space $C_{P}$ and we put $C_{g,T}=2(CT+\otilde{C})$, $\ka_{g,T}=\ka+\otilde{\ka}$.
\end{proof}

\subsubsection{Notations and assumptions on time discretisation}
Let $\de t>0$ and let $\bsm{\xi}\in\mbb{R}^{m}$ be a vector of $m$ independent standard normal random variables, representing the increment of~the $m$-dimensional Wiener process $\mbf{W}$ over a time interval of~length one. We denote a generic one-step time discretisation method approximating the solution $\mbf{X}(t+\de t;t,\mbf{Z})$  of \eqref{eq:sde_aux} as
\begin{equation}
S(t,\mbf{Z}; \de t,\bsm{\xi}).
\end{equation}
Since, in the Monte Carlo setting, we are interested in the weak approximation, we require $S$ to satisfy the following definition (compare \cite[p. 113]{Milstein1995}).
\begin{definition}[Weak consistency of SDE discretisation]\label{def:tdiscr_cons}
If for all functions $g\in C_P^{2(p_{S}+1)}(G, \R)$ there exists a $C_{g,T}\in C_P^0(G,\mbb{R})$ such that
\begin{equation*}
\big|\Exp g\big(S(t,\mbf{z};\de t,\bsm{\xi})\big)-\Exp g\big(\mbf{X}(t+\de t;t,\mbf{z})\big)\big| \leq C_{g,T}(\mbf{z})\cdot(\de t)^{p_{S}+1}
\end{equation*}
is valid for all $\mbf{z} \in G$ and $t$, ${t+\de t \in I}$, we call the one-step method $S$ \emph{weakly consistent of order $p_{S}$}.
\end{definition}

To discretize SDE \eqref{eq:sde_aux} over the interval $[t^0,t^K]\subseteq I$ with $K\in\mbb{N}$ uniform steps, we let $\de t = (t^K-t^0)/K$  and put $t^k = t + k\de t$ for $k=0,\ldots,K$. Moreover, we fix the sequence $(\bsm{\xi}^{k})_{k=1,\ldots,K}$ of independent normally distributed random vectors in $\mbb{R}^{m}$. We begin by taking $\mbf{X}^0=\mbf{Z}$ and, assuming that $\mbf{X}^k$ is given for $k<K$, we put
\begin{equation}\label{eq:sde_tdiscr}
	\mbf{X}^{k+1}=S(t^k,\mbf{X}^k;\de t, \bsm{\xi}^k,)\equiv S^k(\mbf{X}^k;\delta t).
\end{equation}
We require the sequences of random variables generated by $S$ to have moments uniformly bounded with respect to $K$. More precisely:
\begin{assumption}\label{ass:tdiscr_moments}
For every sufficiently large $r$ there is a constant $C_{r,T}$ such that, if $\mbf{X}^{k}$ is given by~\eqref{eq:sde_tdiscr}, we have
\[
\Exp\|\mbf{X}^{k}\|^{2r}\leq C_{r,T}\Exp\|\mbf{X}^{0}\|^{2r}\quad k=1,\ldots,K.
\]
\end{assumption}
See for example \cite[Lem.~9.1]{Milstein1995} for a sufficient condition on general one-step methods for this requirement, and \cite[Proposition 6.2]{roessler06rta} in case of Runge-Kutta methods.

For the discretisation on the whole interval $[t^0,t^K]$ we have, as a consequence of~assumptions already made, the following property.
\begin{lemma}\label{lem:tdiscr_cons_inter}
For all $g\in C_{P}^{2(p_{S}+1)}$ there are constants $C_{g,T}$ and $\ka_{g,T}$ such that
\[
\big|\Exp g\big(\mbf{X}(s;t^K,\mbf{X}^{K})\big) - \Exp g\big(\mbf{X}(s;t^0,\mbf{X}^{0})\big)\big| \leq C_{g,T}( 1+\Exp\|\mbf{X}^{0}\|^{2\ka_{g,T}})\cdot K(\de t)^{p_{S}+1},
\]
uniformly with respect to $s\in[t^K,T]$.
\end{lemma}
\begin{proof}
Assumption~\ref{ass:sde}.\ref{ass:sde_init} provides us with a constant $\otilde{C}_{g,T}>0$ such that
\begin{align*}
\big|\Exp g\big(\mbf{X}(s;t^K,\mbf{X}^{K})\big) - &\Exp g\big(\mbf{X}(s;t^0,\mbf{X}^{0})\big)\big|\\
&\leq\sum_{k=1}^{K}\big|\Exp g\big(\mbf{X}(s;t^k,\mbf{X}^{k})\big) - \Exp g\big(\mbf{X}(s;t^{k-1},\mbf{X}^{k-1})\big)\big|\\
&\leq\otilde{C}_{g,T}\sum_{k=1}^{K}\big|\Exp g\big(\mbf{X}^{k}\big) - \Exp g\big(\mbf{X}(t^{k};t^{k-1},\mbf{X}^{k-1})\big)\big|,
\end{align*}
where we used the identity $\mbf{X}(s;t^{k-1},\mbf{X}^{k-1})=\mbf{X}(s;t^{k},\mbf{X}(t^{k};t^{k-1},\mbf{X}^{k-1}))$, valid due to the uniqueness of solutions. Moreover by Definition~\ref{def:tdiscr_cons}
\begin{align*}
\big|\Exp g\big(\mbf{X}^{k}\big) - \Exp g\big(\mbf{X}(t^{k};t^{k-1},\mbf{X}^{k-1})\big)\big|
&\leq\Exp\oline{C}_{g,T}(\mbf{X}^{k-1})(\de t)^{p_{S}+1}\\
&\leq\oline{C}(1+\Exp\|\mbf{X}^{k-1}\|^{2\oline{\ka}})(\de t)^{p_{S}+1},
\end{align*}
where $\oline{C}$ and $\oline{\ka}$ are constants corresponding to function $\oline{C}_{g,T}\in C_{P}^{0}(G,\mbb{R})$. Finally,  for sufficiently large constants $r>\oline{\ka}$ and $\otilde{C}_{r,T}>1$, independent of $K$, we can employ Assumption~\ref{ass:tdiscr_moments} to get $\Exp\|\mbf{X}^{k-1}\|^{2\oline{\ka}}\leq 1+ \Exp\|\mbf{X}^{k-1}\|^{2r}\leq \otilde{C}_{r,T}(1+\Exp\|\mbf{X}^{0}\|^{2r})$ for every $k=1,\ldots,K$.
\end{proof}

In the It\^{o} case, when $G=\mbb{R}^d$ one can use the  Euler-Maruyama scheme,
\begin{equation}\label{eq:em_step}
\mbf{X}^{k+1} = \mbf{X}^{k} + \mbf{a}(t^k,\mbf{X}^k)\de t + \mbf{b}(t^k,\mbf{X}^k)\sqrt{\de t}\bsm{\xi}^k
\end{equation}
which, for Lipschitz continuous coefficients, has weak order $1$, see~\cite{Milstein1995,KloPla1999}. In case that $G$ is bounded, we supplement this scheme with a truncation step (see Section~\ref{sec:numerics}).

\subsubsection{Monte Carlo simulation}\label{sec:MC_sim}
To discretise SDE~\eqref{eq:sde} in probability space we employ Monte Carlo method \cite{Caflisch1998}, simulating a finite number $J$ of SDE realisations. Given a random variable $\mbf{X}(t)$, the solution of~\eqref{eq:sde} at time $t\in I$, we denote the realisation corresponding to the event $\omega_j$ (that defines the specific Brownian path $t\mapsto\mbf{W}(t;\omega_j)$) as $\mbf{X}_j(t)\equiv\mbf{X}(t;\omega_j)$.

For a given function of interest $\mbf{g}\from\mbb{R}^d\to\mbb{R}^{d'}$ and the ensemble $\msr{X} = \{\mbf{X}_j\}_{j=1}^J$, the Monte Carlo estimate of the expectation of~$\mbf{g}$ is
\begin{equation*}
\hat{\mbf{g}}(\msr{X}) = \frac{1}{J}\sum_{j=1}^J\mbf{g}(\mbf{X}_j).
\end{equation*}
We approximate the evolution of $\bar{\mbf{g}}$, defined in~\eqref{eq:avg}, on $[t^{0},t^{K}]\subseteq I$ with the sequence $\hat{\mbf{g}}^k = \hat{\mbf{g}}(\msr{X}^k)$, $k = 1,\ldots,K$, where $\{\msr{X}^k\}_{k=1}^K$ is produced by the time discretization scheme \eqref{eq:sde_tdiscr} with uniform mesh $\{t^0,t^{1},\ldots t^{K}\}$.
The total error of simulation consists of a deterministic or systematic error -- due to the time discretisation (quantified by the weak error of the scheme), and a statistical error -- due to the finite number of samples. See \cite{KloPla1999} and references therein for more details on Monte Carlo simulation.

\subsubsection{Probability density functions}\label{sec:adv_diff}
For the analysis of the matching operator we will assume that, for every $t\in I=[0,T]$, the solution~$\mbf{X}(t)$ of~\eqref{eq:sde} has a~probability density function (PDF) $\p(t,\cdot)$. In the It\^{o} case this density evolves according to an advection-diffusion equation, also known as Fokker-Planck equation,
\begin{equation}\label{eq:FP}
\pder[t]\p = \grad[\mbf{x}]\cdot\left(-\p\mbf{a} + \frac{1}{2}\grad[\mbf{x}]\cdot(\p\,\bsm{\Si})\right),\quad \text{on}\ (0,T)\times G,
\end{equation}
where $\bsm{\Si}=\mbf{b}\mbf{b}^{T}$ is the diffusion matrix. We supplement \eqref{eq:FP} with the initial-boundary condition
\begin{equation}\label{eq:FP_IC}
\p(0,\cdot)=\p_0\ \text{on}\ G,\quad \p(\cdot,\mbf{x})=0\ \text{for every}\ \mbf{x}\in\bdy{G}
\end{equation}
where $\p_0$ is the law of $\mbf{X}(0)$.
It can be proved (under certain assumptions, see \cite{Pavliotis2014}) that for all $t\in I$ the solution $\p(t,\cdot)$ of \eqref{eq:FP}--\eqref{eq:FP_IC} is also a probability density. Hence, we can equivalently compute the averages in~\eqref{eq:avg} as
\begin{equation}\label{eq:avg_int}
\oline{\mbf{g}}(t) = \int_{G}\mbf{g}(\mbf{x})\,\p(t,\mbf{x})\der{\mbf{x}},\quad t\in I.
\end{equation}
Note that, for $G=\mbb{R}^d$, the integral is finite for any $g\in C_{P}^{0}(G,\mbb{R})$ when $\p(t,\cdot)$ is decaying rapidly (i.e. faster than any polynomial) at infinity. Moreover, we can guarantee the higher regularity of the solutions to~\eqref{eq:FP}, this is related to Assumption~\ref{ass:sde}, with appropriate smoothness of the initial condition $\p_{0}$ and the coefficients $\mbf{a}$ and $\mbf{b}$ (see the discussion in Section~\ref{sec:numerics} and \cite{Pavliotis2014}).

\section{Micro-macro acceleration method}\label{sec:mM_acc}

Our method aims at being faster than a full microscopic simulation, while converging to it when the extrapolation time step vanishes. The underlying assumption is that the macroscopic state variables can be simulated on a much slower time-scale than the microscopic dynamics, thus allowing the choice of a large extrapolation time step compared with the time step for microscopic simulation. We introduce the restriction and matching operators (\ref{sec:res_match}) to connect the two levels of description, then we discuss the extrapolation step (\ref{sec:extrapolation}). We present specific matching operators in Section~\ref{sec:match_oper}.

\subsection{Matching and restriction operators\label{sec:res_match}}

\subsubsection{Moments of probability measures}
Let $\pms(G)$, where $G\subseteq\mbb{R}^d$ is measurable, denote the set of all probability measures on $G$, i.e., all non-negative measures $\mu$ on $\mbb{R}^d$ with support in G and such that $\mu(G)=1$. For our analysis, we consider a~sequence of moments $(m_l(\mu))_{l=1}^\infty$, obtained from $\mu$ by taking the expectations of given functions~$R_l$, i.e.,
\begin{equation}\label{eq:moment_repres}
	m_l(\mu) = \int_{\mbb{R}^{d}}R_l(\mbf{x})\,\mu(\der{\mbf{x}}), \qquad l\ge 1.
\end{equation}
Henceforward we will call $R_l$ the \emph{moment function} and we will say that the value~$m_l$ is the $l$-th moment of measure~$\mu$.

The choice of the functions~$R_l$ is problem-dependent. Clearly, they should be selected such that the integrals in~\eqref{eq:restr_with_functions} exist, at least for a subset of $\pms(G)$ containing the laws of all trajectories of the SDE~\eqref{eq:sde} in consideration. In our case it suffices that $R_l$ are in the class of continuous and polynomially bounded functions defined in Section~\ref{sec:not}. We also want each moment to hold new information about the measures, so that the functionals $\mu\mapsto m_l(\mu)$, $l\geq1$ are linearly independent. Thus, we need an appropriate notion of independence for moment functions. Finally, we require a~one-to-one correspondence between the measure $\mu$ and its full moment representation $(m_l(\mu))_{l=1}^\infty$. To collect all required properties, we introduce the following assumption (for more on the pseudo-Haar property, see \cite{BorLew1991b} and Appendix~\ref{app:optimality}):
\begin{samepage}
\begin{assumption}\label{ass:hierarchy}
	The functions $R_l$,  $l \ge 1$, satisfy the following conditions:
	\begin{enumerate}[a.]
		\item they are linearly independent on every non-null subset of G (pseudo-Haar);
		\item for each $l$, $R_l\in C^{0}_{P}(\mbb{R}^d,\mbb{R})$ is non-constant;
		\item if the law of the random variable~$\mbf{Z}$ is uniquely determined by all its moments, the same is true for 			$\mbf{X}(s;t,\mbf{Z})$, for any $0\leq t<s\leq T$.\label{ass:moments}
	\end{enumerate}
\end{assumption}
\end{samepage}
\begin{example}[Algebraic moments]\label{ex:alg_moms}
In particular, we can consider the functions $R_{\al}(\mbf{x})=\mbf{x}^{\al}=x_1^{\al_1}\cdot\ldots\cdot x_d^{\al_d}$, $\mbf{x}\in G$, where $\al\in\mbb{N}_0^d$ is any multi-index, so that each value $m_{\al}(\mu)$ is the mixed raw moment of probability measure $\mu$. If $G=\mbb{R}^d$ these moments exist for measures having densities decaying at infinity faster than any polynomial, and for all elements of $\pms(G)$ if $G$ is bounded. The uniqueness of~the moment representation is guaranteed by the existence of the moment generating function.  
\end{example}

\begin{samepage}
\subsubsection{Definition and basic properties of restriction and matching}\label{sec:res_match_prop}
\paragraph{Restriction}
To reduce the distributions to a finite collection of $L$ macroscopic state variables, we introduce a \emph{restriction} operator
\begin{equation}\label{eq:res}
\res_{L}: \pms(G) \to \mbb{R}^L,\quad \mu \mapsto \res_{L}\mu.
\end{equation}
The values of~$\res_{L}$ represent the coarse-grained description of a microscopic law. We want to analyze the convergence of micro-macro acceleration method with increasing $L$, thus we will consider the hierarchy of restriction operators in which
\begin{equation}\label{eq:restr_with_functions}
	\res_L\mu = \big(m_1(\mu),\ldots, m_L(\mu)\big),\qquad L\geq1,
\end{equation}
i.e., we truncate the moment representation~\eqref{eq:moment_repres} of the measure $\mu$ to its first $L$ terms.
Note that if one is only interested in approximating dynamics up to some tolerance, one may clearly consider the restriction operator in terms of a limited number of~macroscopic state variables of interest, leaving the rest of the hierarchy unspecified.\smallskip
\end{samepage}

\paragraph{Matching}
Conversely, to obtain a probability measure (a law of a random variable) that is consistent with a finite set of $L$ macroscopic state variables, we consider a \emph{matching} operator $\match_{L}$. Since we have to deal with distributions that are in general not uniquely determined by a finite set of macroscopic state variables, the matching represents the inference procedure with which we associate a law to this macroscopic state. This is the case for the laws arising in the FENE dumbbell model presented in Section~\ref{sec:numerics}.

The idea of the matching operator is to use a \emph{prior} measure $\mu\in\pms(G)$ that we alter to make it consistent with the vector of (extrapolated) macroscopic states $\mbf{m}\in\mbb{R}^{L}$. In general, this is an underdetermined problem -- infinitely many solutions are possible. One therefore has to choose a particular strategy that, at least under certain assumptions, will pick a unique distribution. In Section~\ref{sec:match_oper}, we analyse a~strategy based on a~generalisation of the \emph{entropy principle}, which selects the consistent probability measure that is ``closest'' to the prior. Hence, the matching operator with $L$ moments will be defined, at least formally, as
\begin{equation}\label{eq:match}
\match_L(\mbf{m},\mu)=\argmin_{\nu\in\res_L^{-1}(\mbf{m})}d(\mu,\nu),
\end{equation}
where $d(\cdot,\cdot)$ quantifies the (quasi-)distance between probability measures. We will consider $d$ to be an $L^2$-norm distance or an $f$-divergence \cite{Csiszar1967}.

\begin{remark}[Restriction and matching with random variables]\label{rem:res_match_rv}
We defined the restriction and matching operators as acting on probability measures. However, with a slight abuse of notation, we will also write $\res_L(\mbf{X})$ and $\match_L(\mbf{m},\mbf{X})$ for a random variable $\mbf{X}$, meaning we consider its distribution as the argument in the operators. As a value of matching we take any random variable with the law given by $\match_L(\mbf{m},\mbf{X})$.
\end{remark}

\paragraph{Restriction-matching pair}
The matching and restriction operators are related and need to be studied simultaneously. Thus, we introduce the following notion:
\begin{definition}[Restriction-matching pair]\label{df:res-match}
Let $L\in\mbb{N}$ and $G\subseteq\mbb{R}^{d}$ be measurable. Assume that $\res_L\from\pms(G)\to\mbb{R}^L$ and $\match_L\from\mbb{R}^{L}\times\pms(G)\supseteq\dom\match_L\to\pms(G)$, satisfy
\begin{enumerate}
\item $\res_L\left(\match_L(\mbf{m},\mu)\right)=\mbf{m}$ for all $(\mbf{m},\mu)\in\dom\match_L$,
\item $\match_L(\res_L(\mu),\mu)=\mu$ for all $(\res_L(\mu),\mu)\in\dom\match_L$ (projection property).
\end{enumerate}
Then we call $(\res_L$, $\match_L)$ the \emph{restriction-matching pair} with $L$ macroscopic state variables.
\end{definition}

A few remarks on the restriction-matching pair are in order.
\begin{remark}[Domain of the matching operator]\label{rem:match_domain}
The first assumption in Definition~\ref{df:res-match} implies that $\dom\match_{L}\subseteq\im\res_{L}\times\pms(G)$, meaning that the moment vector ~$\mbf{m}$ we consider must always correspond to at least one probability distribution. This is, of course, a necessary condition for the existence of the matched measure but it may not be sufficient, as we point out in the case of matching based on $f$-divergences. In Section~\ref{sec:match_fail}, we discuss the effect on the numerical behaviour of trying to match a~prior~$\mu$ with a set of moments that are not realisable.
We also restrict the domain to a subset of $\pms(G)$, since sometimes not all measures can be obtained through the matching operator, for instance, if higher integrability or absolute continuity is required (see Section~\ref{sec:match_oper}).
\end{remark}

\begin{remark}[The term projection]
Let $\mbf{m}\in\img\res_L$ and consider the mapping $\match_L(\mbf{m},\cdot)$. The second condition in Definition~\ref{df:res-match} implies that $\match^2_{L}(\mbf{m},\cdot)=\match_L(\mbf{m},\cdot)$. This justifies the use of the term projection.
\end{remark}

We are now ready to formulate the requirements on the sequences of restriction-matching pairs that allow us to show convergence of the micro-macro acceleration method to the full microscopic dynamics.
\begin{property}[Continuity of matching]\label{prop:cont-match}
Fix $L\in\mbb{N}$ and let $(\res_{L},\match_{L})$ be a re\-stric\-tion-matching pair with $L$ macroscopic state variables. We say that $\match_L$ is \emph{(weakly) continuous} if for all $g\in C_P^{0}(G,\mbb{R})$ there exists a constant $C_L=C_L(g)>0$ such, that
\begin{equation}\label{eq:cont-match}
\left|\Exp g(\match_L(\mbf{m}^1,\mbf{Z})) - \Exp g(\match_L(\mbf{m}^2,\mbf{Z}))\right|\leq C_L \|\mbf{m}^1-\mbf{m}^2\|,
\end{equation}
for every random variable~$\mbf{Z}$ and all vectors $\mbf{m}^i\in\im\res_{L}$ with $(\mbf{m}^{i},\mbf{Z})\in\dom\match_{L}$, $i=1,2$.
\end{property}
\begin{property}[Consistency of matching]\label{prop:consist-match}
Consider a sequence $\left\{(\res_{L},\match_{L})\right\}_{L=1}^{\infty}$\\ of restriction-matching pairs with $\res_L$ given by \eqref{eq:restr_with_functions} and satisfying Assumption~\ref{ass:hierarchy}. We say that this sequence is \emph{consistent with equation~\eqref{eq:sde}} if for any solution $\mbf{X}$:
\begin{enumerate}[a.]
\item for all $L$ and $0\leq t'\leq t\in I$ we have $\big(\res_{L}\mbf{X}(t),\mbf{X}(t')\big)\in\dom\match_{L}$;
\item for every $g\in C_{P}^{2(p+1)}(G,\mbb{R})$ and $L\in\mbb{N}$, there exists a constant $C_L=C_L(g)>0$ such that
with $\mbf{m}_{L}=\res_L\mbf{X}(t;t',\mbf{Z})$
\begin{equation}\label{eq:match_cons_SDE}
\big|\Exp g\big(\match_L(\mbf{m}_L,\mbf{Z})\big) - \Exp g\big(\mbf{X}(t;t',\mbf{Z})\big)\big|\leq C_L(t-t'),
\end{equation}
for all $0\leq t'\leq t\in I$ and all initial random variables~$\mbf{Z}$;
\item for fixed $g\in C_{P}^{2(p+1)}(G,\mbb{R})$, $C_L\to 0$ as $L\to+\infty$.
\end{enumerate}
\end{property}
We study specific matching operators, and discuss Properties~\ref{prop:cont-match} and~\ref{prop:consist-match}, in Section~\ref{sec:match_oper}.

\subsection{Extrapolation operator\label{sec:extrapolation}}
In this manuscript we only consider first order extrapolation, which is reminiscent of forward Euler integration of the macroscopic state variables.  This idea was first proposed in \cite{GeaKevThe2002}, see also \cite{KevGeaHymKevRunThe2003,KevSam2009,VanRoo2008}.

We introduce two indices, $k=0,\ldots,K$ and $n=0,\ldots,N$, to emphasise the fact that there are two time steps involved: the microscopic time step $\de t>0$, over which we will evolve the full microscopic dynamics; and the macroscopic time step $\De t>0$, over which we will perform extrapolation of the macroscopic state variables. The idea is to take a small number $K$ of microscopic steps of size $\de t$, such that $K\de t\ll\De t$, and from a microscopic simulation, starting at time $t^n=n\De t$, evaluate $\mbf{m}^{n,k}$ at time $t^{n,k}=t^n+k\de t$, for $k=1,\ldots,K$, using the restriction operator~\eqref{eq:res}. We then extrapolate as follows:
\begin{equation}\label{eq:CFE}
	\mbf{m}^{n+1}=\ext\left((\mbf{m}^{n,k})_{k=0}^{K};\de t, \De t\right) = \mbf{m}^{n,0} + \frac{\De t}{K\de t}(\mbf{m}^{n,K} - \mbf{m}^{n,0}),
\end{equation}
in which $\mbf{m}^{n,0}\equiv \mbf{m}^n$. Clearly, the forward Euler extrapolation~\eqref{eq:CFE} satisfies:
\begin{equation}\label{eq:cont-cfe}
\begin{aligned}
\|\mbf{m}_1^{n+1}-\mbf{m}_2^{n+1}\|
&\leq\dfrac{\De t}{K\de t}\|\mbf{m}_{1}^{n,K}-\mbf{m}_{2}^{n,K}\|+
\Big(\dfrac{\De t}{K\de t}-1\Big)\|\mbf{m}_{1}^{n,0}-\mbf{m}_{2}^{n,0}\|\\
&\leq\dfrac{\De t}{K\de t}\big(\|\mbf{m}_{1}^{n,K}-\mbf{m}_{2}^{n,K}\|+\|\mbf{m}_{1}^{n,0}-\mbf{m}_{2}^{n,0}\|\big),
\end{aligned}
\end{equation}
with any two sequences $\{\mbf{m}_{i}^{n,k}\}_{k=0,\ldots,K}$, $i=1,2$, where in the first estimate we used the fact that $K\de t\leq\De t$. Moreover, we have the following Lemma (the proof follows from a simple first-order Taylor expansion):
\begin{lemma}[Consistency of extrapolation]\label{lem:extrap}
Let $\otilde{\mbf{m}}\in C^2([t^n,t^{n+1}])$ and put
\[
\mbf{m}^{n+1}=\ext\left((\otilde{\mbf{m}}(t^{n,k}))_{k=0}^{K};\de t, \De t\right)
\]
with the extrapolation operator $\ext$ defined in~\eqref{eq:CFE}.  Then, we have
\[
\|\otilde{\mbf{m}}(t^{n+1})-\mbf{m}^{n+1}\|\leq\max_{[t^n,t^{n+1}]}\|\otilde{\mbf{m}}''\|\cdot(\De t)^2.
\]
\end{lemma}
%
Higher order versions of \eqref{eq:CFE} can be constructed in several ways, e.g.~using the polynomial extrapolation \cite{GeaKevThe2002}. Adams-Bashforth or Runge-Kutta implementations of \eqref{eq:CFE} are
also possible \cite{RicoGearKevr04,Lee:2007p2355,Lafitte:EprintArxiv14046104:2014,Lafitte:EprintArxiv14064305:2014}, as are implicit versions, partially discussed in \cite{GeaKev2003}.  Another idea, based on \cite{SOMMEIJER:1990p2657}, trades accuracy for stability by designing a multistep state extrapolation method based on macroscopic states at multiple macroscopic time steps~\cite{Vandekerckhove:2009p4623}.

\subsection{Description of the method and convergence result}\label{sec:alg}
Now that we have introduced all building blocks, we describe the method as a whole in Algorithm~\ref{alg:accel}. Let us also discuss shortly some of the issues related to the method.
\begin{algorithm}
\label{alg:accel}
Given a microscopic state $\mbf{X}^n$ at time~$t^n$, macroscopic step size $\De t>0$, microscopic step size $\de t>0$, and a number $K\in\mbb{N}$ of microscopic steps, with $K\de t\leq\De t$, compute the microscopic state $\mbf{X}^{n+1}$ at time $t^{n+1}=t^n+\De t$ via a four-step procedure:
\begin{enumerate}[(i)]
\item \emph{Simulate} the microscopic system over $K$ time steps of size $\de t$ using a microscopic discretisation scheme
\begin{equation}\label{eq:alg-micro-step}
\mbf{X}^{n,k+1} = S^{n,k}(\mbf{X}^{n,k}; \de t),\quad k=0,\ldots,K-1,
\end{equation}
defined in~\eqref{eq:sde_tdiscr}, with $\mbf{X}^{n,0}=\mbf{X}^n$.
\item \emph{Record} the macroscopic states $\mbf{m}^{n,k}=\res_L\mbf{X}^{n,k}$ for $k=0,\ldots,K$ using the restriction operator~\eqref{eq:res}.
\item \emph{Extrapolate} the macroscopic states $\mbf{m}^{n,0},\ldots,\mbf{m}^{n,K}$ over a step of size $\De t$ to a new macroscopic state $\mbf{m}^{n+1}$ at time $t^{n+1}$ using the extrapolation operator~\eqref{eq:CFE},
\begin{equation}\label{eq:alg-extrap-step}
\mbf{m}^{n+1} = \ext\left((\mbf{m}^{n,k})_{k=0}^K;\de t,\De t\right).
\end{equation}
\item \emph{Match} the microscopic state $\mbf{X}^{n,K}$ at time $t^{n,K}$ with the extrapolated macroscopic state $\mbf{m}^{n+1}$ using the matching operator~\eqref{eq:match},
\begin{equation}\label{eq:alg-matching-step}
\mbf{X}^{n+1} = \match_L(\mbf{m}^{n+1},\mbf{X}^{n,K}),
\end{equation}
to obtain a new microscopic state $\mbf{X}^{n+1}$ at time $t^{n+1}$.
\end{enumerate}
\end{algorithm}

\paragraph{Stability of the algorithm}
\added[id=r1,remark={3}]{To investigate the efficiency of the method, we would be interested in the maximal value of the ratio $\De t / \de t$ that is affordable. Clearly, this efficiency will depend on the time-scale separation that is present in the problem. A common approach is to ask when a numerical method preserves the asymptotic stability of an equilibrium in a particular test equation, i.e., a linear stability analysis. In the stochastic context, the choice of the test equation and its connection with nonlinear dynamics is more involved than in the deterministic case, and numerous approaches exist~\cite{Saito2008,Buckwar2010,BucRieKlo2011}. In this manuscript, we concentrate on the convergence to the microscopic dynamics, and leave the stability analysis for future work.
}
\paragraph{Time-scale separation}
\added[id=r2,remark={maj5}]{In numerical closure methods, such as the equation-free and heterogenous multi-scale methods \cite{KevGeaHymKevRunThe2003,EEng03}, one obtains an algorithm that approximates an unavailable closed system of ODEs for the moments. We do not need such a closed model to exist in our approach, but when it does, the spectral properties of this system relate to the stability of extrapolation, the third stage in Algorithm~\ref{alg:accel}. In particular, the time-scale separation will then manifest itself as a large gap in the spectrum, and the combination of microscopic simulation with extrapolation dumps the fast components, allowing for large $\De t$. We refer to~\cite{GeaKev2003,VanRoo2008} for the study of the extrapolation procedure~\eqref{eq:CFE} in this context.
}
\paragraph{Numerical implementation}
\added[id=r2,remark={min1}]{We present the Algorithm as it operates on the random variables, see also Remark~\ref{rem:res_match_rv} on how we understand the matching step in this framework.} In the numerical implementation, we will need to consider a version of the algorithm that deals with an ensemble of realisations $\ens{X}=\left\{\mbf{X}_j\right\}_{j=1}^J$ instead of the random variable $\mbf{X}$.  This will be discussed in Section~\ref{sec:num_match}, \added[id=r2,remark={min1}]{where we demonstrate how the particular matching strategies lead to a natural re-weighting procedure.
}

Convergence of this method is guaranteed by the following theorem:
\begin{theorem}\label{thm:convergence}
Consider the SDE~\eqref{eq:sde} satisfying Assumptions~\ref{ass:sde_exist} and~\ref{ass:sde}, and its solution $\mbf{X}$ with initial condition $\mbf{X}(0)=\mbf{X}^0$. Let $R_l\in C^4(G,\mbb{R})$, $l\geq1$, be a sequence of moment functions, fulfilling Assumption~\ref{ass:hierarchy}, that generate restriction operators $\res_L$, $L\geq1$, by~\eqref{eq:restr_with_functions} and let $\{(\res_L,\match_L)\}_{L=1}^\infty$ be a sequence of restriction-matching pairs having Properties~\ref{prop:cont-match} and~\ref{prop:consist-match}. Furthermore, consider a time discretisation scheme~\eqref{eq:sde_tdiscr} of order $p_S\ge 1$ with time step $\de t$ and satisfying Assumption~\ref{ass:tdiscr_moments}. Finally, let $\ext$ be the extrapolation operator~\eqref{eq:CFE} with extrapolation step $\De t$ and let $K\in\mbb{N}$ be a number of microscopic steps with $K\de t\leq\De t$.

If we denote the solution of Algorithm~\ref{alg:accel} with $L$ macroscopic state variables at~time~$T$ as $\mbf{X}^{N}_L$, for any function $g\in C_P^{2(p_S+1)}(G, \mbb{R})$ we have
\begin{equation}\label{eq:error_mic-mac}
\left|\Exp\big[g(\mbf{X}^N_L) - g(\mbf{X}(T))\big]\right|\leq C_L + \otilde{C}_{L}\big((\de t)^{p_S}+\De t\big),
\end{equation}
in which  $C_L$ and $\otilde{C}_{L}$ are constants that depend also on $T, g$ and $\mbf{X}^0$, with $C_L\to 0$ as $L\to+\infty$.
\end{theorem}

\begin{remark}
\added[id=r2,remark={maj4}]{Note that in Theorem~\ref{thm:convergence},  we only require a consistency relation $K\de t\leq\De t$. When $\De t = K\de t$, due to the projection property in Definition~\ref{df:res-match}, Algorithm~\ref{alg:accel} reduces to the first stage -- the microscopic simulation. Thus, our convergence analysis is relevant for $\De t> K\de t$, when the extrapolation-matching stage is turned on.
}
\end{remark}

The proof of Theorem~\ref{thm:convergence} will be given in Section~\ref{sec:convergence-proof}.Theorem~\ref{thm:convergence} shows that the error of the micro-macro acceleration method is composed of three terms: (i)~matching error that depends only on the number $L$ of macroscopic state variables and that can be made arbitrarily small by choosing $L$ sufficiently large; (ii)~microscopic discretisation error; and (iii)~extrapolation error.  The last two errors can be made arbitrarily small by a suitable choice of $\de t$ and $\De t$. Moreover, convergence does not rely on the precise definition of the restriction-matching pair, but only on its generic continuity and consistency properties. Specific restriction-matching pairs are discussed and analysed in Section~\ref{sec:match_oper}.

\section{Matching operators}\label{sec:match_oper}
This Section is devoted to the investigation of specific matching operators that can be used in the micro-macro acceleration method. All operators are based on the minimisation of a distance (\ref{sec:l2n}) or an $f$-divergence (\ref{sec:fdiv}) between the probability density functions with the constraints given by the restriction operator.

\subsection{Notations and function spaces}\label{sec:match_oper_not}
In this Section, we will work in the Le\-bes\-gue spaces $L^p(G,\mu)$ with norm $\|\cdot\|_{p}$, where $G\subset\mbb{R}^d$ is open and bounded, $\mu$ is a finite Borel measure on $\mbb{R}^d$ with full support on $G$ and $p\in[1,+\infty)$. We will also consider the convex set $\pds^p(G)\subset L^p(G)$ of all probability densities integrable with  $p$-th power, and the cone $L_+^p(G,\mu)$ of all non-negative functions in $L^p(G,\mu)$. Finally, recall that the dual space to  $L^p(G,\mu)$ is isomorphic (congruent) to $L^q(G,\mu)$ with $q\in(1,+\infty]$ satisfying $1/p+1/q=1$ (see \cite[p. 128]{Holmes1975}). The \emph{dual pairing} between $\varphi\in L^q(G,\mu)$ and $\ph\in L^p(G,\mu)$ is given by
\begin{equation*}
\langle\varphi,\,\ph\rangle = \int_{G}\varphi(\mbf{x})\ph(\mbf{x})\,\mu(\der{\mbf{x}}).
\end{equation*}
Later on, we will specifically use $p=1$ and $p=2$.

By definition, for an appropriate vector $\mbf{m}\in\mbb{R}^L$ of moments and a prior density $\prior\in\pds(G)$ the result of matching is a PDF, i.e., $\match_L(\mbf{m},\prior)\in\pds(G)$. Hence, we need to ensure that the solutions of minimisation problem~\eqref{eq:match} integrate to one and are non-negative.
We deal with the first requirement by adding the unity of the zeroth moment (mass) as an additional linear constraint. To avoid confusion with the constraints imposed by the moments $\mbf{m}$, we introduce the (extended) restriction operator $\otilde{\res}_L\from L^p(G,\mu)\to\mbb{R}^{L+1}$, in which, besides the finite number of moments, also the conservation of mass is included:
\begin{equation}\label{eq:res_Lp}
(\otilde{\res}_L\ph)_{l} = \langle R_l,\,\ph\rangle\quad l=0,\ldots,L,
\end{equation}
for some \emph{moment functions} $R_l\in L^q(G,\mu)$, $l=1,\ldots,L$, and $R_0\equiv1\in L^\infty(G,\mu)$, see also equation~\eqref{eq:restr_with_functions}. The norm of this operator satisfies $\opnorm{\otilde{\res}_L}\leq\|\mbf{R}\|_{q}$, where $\mbf{R} = (R_{0},\ldots, R_{L})$  is the vector of moment functions.
To ensure non-negativity of the solution we proceed in two ways. In Section~\ref{sec:l2n}, we do not include this property in the problem itself. Instead, we distinguish later the set of prior densities and moments for which positivity is preserved. This approach facilitates the analysis of the matching operator. In Section~\ref{sec:fdiv}, we include this restriction directly in the minimisation problem as a convex constraint using the cone $L^p_+(G)$. In this case all solutions are guaranteed to be PDFs, but the analysis becomes more difficult.

\subsection{Matching with \texorpdfstring{$L^2$}{L2} norm}\label{sec:l2n}
First, \added[id=r1, remark={4}]{in Section~\ref{sec:l2n_def}}, we construct a matching using the $L^2$ norm (L2N) $\|\cdot\|_2$ as a distance criterion between probability measures. \added[id=r1, remark={4}]{Here, our approach does not guarantee from the outset the positivity of the resulting density, but we indicate a sufficient condition to preserve this property. We also establish the continuity of this matching (Property~\ref{prop:cont-match}). Second, in Section~\ref{sec:l2n_cons}, we restrict our analysis to the one-dimensional case with the algebraic moments from Example~\ref{ex:alg_moms}, and we demonstrate the consistency (Property~\ref{prop:consist-match}). Hence, the L2N based matching fulfils the two basic requirements from Section~\ref{sec:res_match_prop}.}

\subsubsection{\texorpdfstring{\added[id=r1, remark={4}]}{}{Definition and continuity}}\label{sec:l2n_def}
Let $(R_l)_{l=1}^{+\infty}\subset L^2(G)$ be a sequence of non-constant, linearly independent moment functions. Fix $L\in\mbb{N}$ and assume that we are given a positive prior PDF $\prior\in\pds^2(G)$, uniquely determined by all its moments, and a vector of target moments $\mbf{m}\in\mbb{R}^{L}$ such, that $\mbf{m} = \res_L(\targ)$ for some density $\targ\in\pds^2(G)$. Consider the optimal solution $\vph_2(\mbf{m},\pi)$ to the problem
\begin{equation}\label{eq:primal_L2}
\left\{\begin{array}{ll}\medskip
\text{inf} & \displaystyle\frac{1}{2}\|\vph-\prior\|_2^2\\\medskip
\text{s.t.} & \otilde{\res}_L\vph = (1,\mbf{m})\\
               & \vph\in L^2(G),
\end{array}\right.
\end{equation}
where $\otilde{\res}_L$ is the restriction operator \eqref{eq:res_Lp}. Here we take $\mu$ to be the Lebesgue measure on $G$ and without loss of generality we can assume that the system $\{R_l\}_{l=0,1,\ldots}$ is a~basis for $L^2(G)$. The unique solution to problem \eqref{eq:primal_L2} always exists, as can be seen by putting $\ph=\vph-\targ$ and considering equivalently the following least squares problem in $L^2(G)$:
\begin{equation}\label{eq:least_sq}
\left\{\begin{array}{ll}\medskip
\text{inf} & \displaystyle\frac{1}{2}\|\ph - (\prior - \targ)\|_2^2\\\medskip
\text{s.t.} & \ph \in \otilde{\res}_{L}^{-1}(0),
\end{array}\right.
\end{equation}
in which the constraints imply that $ \otilde{\res}_{L}(\vph)= \otilde{\res}_{L}(\targ)$. If we denote by $\{Q_l\}_{l=0,1,\ldots}$ the orthonormal basis obtained from $\{R_l\}_{l=0,1,\ldots}$ by the Gram-Schmidt orthogonalisation procedure, we have $\otilde{\res}_{L}^{-1}(0)=\cl{\linspan}\{Q_l:\ l\geq L+1\}$. This equality is an easy consequence of the fact that the change of basis in Gram-Schmidt orthonormalisation procedure is described by the infinite lower triangular matrix $\ohat{Q}$  such, that
\begin{equation}\label{eq:gram_schmidt}
Q_l=\sum_{j=0}^{l}\ohat{Q}_{l,j}R_j,\quad l = 0,1,\ldots\,.
\end{equation}
The solution to the least squares problem \eqref{eq:least_sq} is $\sum_{l=L+1}^{+\infty}\langle\prior-\targ,Q_l\rangle\, Q_l$ (see, e.g., \cite[Sec.~3.4 \& 3.6]{AtkHan2009}), so going back to $\vph$ we obtain
\begin{equation}\label{eq:optimal_L2N_orth}
\vph_2(\mbf{m},\prior) = \sum_{l=0}^{L}\langle\targ,Q_l\rangle\, Q_l + \sum_{l=L+1}^{+\infty}\langle\prior,Q_l\rangle\, Q_l.
\end{equation}
Let us also mention that for fixed $\mbf{m}$, $\vph_2(\mbf{m},\cdot\,)$ is a non-expansive projection operator in $L^2(G)$ onto the hyperplane $\otilde{\res}_{L}^{-1}(\mbf{m})$ \cite[Prop.~3.4.4]{AtkHan2009}. Here non-expansiveness means $\|\vph_2(\mbf{m},\prior_1) - \vph_2(\mbf{m},\prior_2)\|_{2}\leq\|\prior_{1}-\prior_{2}\|_{2}$ for any $\prior_{1},\prior_{2}\in L^2(G)$.

To derive the formula in the original basis, let $\ohat{Q}_L$ be the $(L+1)\times(L+1)$ left upper submatrix of $\ohat{Q}$. Adding and subtracting $\prior$ on the right hand side of \eqref{eq:optimal_L2N_orth}, and using the fact that for $l\leq L$ we can write $L$ in the upper limit of the sum in \eqref{eq:gram_schmidt}, we obtain
\begin{multline*}
\vph_2(\mbf{m},\prior) = \sum_{l=0}^{L}\langle\targ-\prior, Q_l\rangle\,Q_l + \prior
	= \sum_{l=0}^{L}\sum_{i,j=0}^{L}\langle\targ-\prior, R_j\rangle\,\ohat{Q}_{l,i}\ohat{Q}_{l,j}\,R_i + \prior\\
	= \sum_{i=0}^{L}\left(\sum_{j=0}^{L}(\ohat{Q}^T_L\ohat{Q}_L)_{i,j}\langle\targ-\prior, R_j\rangle\right)R_i + \prior
	=\left[H_{L}^{-1}\langle\targ-\prior, \mbf{R}\rangle\right]^T\mbf{R} + \prior,
\end{multline*}
where $H_L$ is the $(L+1)\times(L+1)$ matrix such, that $(H_L)^{-1} = \ohat{Q}^T_L\ohat{Q}_L$. Thus, we finally get
\begin{equation}\label{eq:optimal_L2N}
\vph_2(\mbf{m},\prior) 
= \sum_{l=0}^{L}\la_lR_l + \prior = \left(\frac{\sum_{l=0}^{L}\la_lR_l}{\prior}+1\right)\cdot\prior,
\end{equation}
where the coordinates of the vector $H_{L}^{-1}\langle\targ-\prior, \mbf{R}\rangle=H_{L}^{-1}(0,\mbf{m}-\res_L\prior)$, denoted by~$\la_l$, can be viewed as  Lagrange multipliers. The entries of $H_L$ are of the form
\begin{equation}\label{eq:hess_L2N}
(H_L)_{k,l}=\int_{G}R_k(\mbf{x})R_l(\mbf{x})\,\der{\mbf{x}},
\end{equation}
for $k,l=0,\ldots,L$, as can be seen from \eqref{eq:gram_schmidt} using the identity $\langle Q_k, Q_l\rangle = \de_{k,l}$.

Of course, $\vph_2(\mbf{m},\pi)$ has  mass one but can in general be negative on the set of positive Lebesgue measures, thus it corresponds to a signed measure. However, we can ensure positivity when the prior distribution is bounded away from zero on $G$ and the target moments $\mbf{m}$ are close enough to the moments of the prior $\pi$. The following Lemma is an easy consequence of representation~\eqref{eq:optimal_L2N}.
\begin{lemma}\label{lem:L2N_pos}
Assume that $\pi\geq c$ a.e.~on $G$ for some constant $c>0$. There exists a constant $\de=\de(L,G,\mbf{R})>0$ such that if $|\mbf{m}-\res_L\pi|<\de$, we have $\vph_2(\mbf{m},\pi)\geq0$ a.e.~on $G$.
\end{lemma}

Let us now analyse operator $\vph_2$. The continuity of matching (Property~\ref{prop:cont-match}) follows from:
\begin{lemma}
Let $\prior, \targ, \otilde{\targ}\in\pds^2(G)$, and let $\mbf{m}_{L}=\res_L\targ$, $\otilde{\mbf{m}}_{L}=\res_L\otilde{\targ}$ with $L\in\mbb{N}$. Then
\begin{equation}\label{eq:cont_L2N}
\|\vph_2(\mbf{m}_L,\prior) - \vph_2(\otilde{\mbf{m}}_L,\prior)\|_2\leq\|\ohat{Q}_L\|\cdot\|\mbf{m}_L - \otilde{\mbf{m}}_L\|.
\end{equation}
\end{lemma}
\begin{proof}
Let us denote by $\prior_l,\targ_l,\otilde{\targ}_l$ the $l$-th Fourier coefficient of $\prior, \targ, \otilde{\targ}$ respectively in the basis $\{Q_l\}_{l=0,1,\ldots}$. Note that for $\vph_2(\mbf{m}_L,\prior)$ and $\vph_2(\otilde{\mbf{m}}_L,\prior)$ only the first sum in expansion~\eqref{eq:optimal_L2N_orth} differs. Thus, Parseval's identity implies
\[
\|\vph_2(\mbf{m}_L,\prior) - \vph_2(\otilde{\mbf{m}}_L,\prior)\|^2_2 = \left\|\sum_{l=0}^{L}(\targ_l-\otilde{\targ}_l)\leg_l\right\|^2_2 = \sum_{l=0}^{L}(\targ_l-\otilde{\targ}_l)^2.
\]
Now, according to \eqref{eq:gram_schmidt}, we have $(\targ_0,\ldots,\targ_L)=\ohat{Q}_L\mbf{m}_L$ and $(\otilde{\targ}_0,\ldots,\otilde{\targ}_L)=\ohat{Q}_L\otilde{\mbf{m}}_L$. Hence, we finally obtain
\[
\|\vph_2(\mbf{m}_L,\prior) - \vph_2(\otilde{\mbf{m}}_L,\prior)\|^2_2 = \|\ohat{Q}_L(\mbf{m}_L-\otilde{\mbf{m}}_L)\|^2,
\]
from which \eqref{eq:cont_L2N} follows.
\end{proof}

\subsubsection{\texorpdfstring{\added[id=r1,remark={4}]}{}{Consistency}}\label{sec:l2n_cons}
\added[id=r1,remark={4}]{In this Section, we} prove consistency (Property~\ref{prop:consist-match}) \added[id=r1,remark={4}]{of L2N based matching, assuming} $G=(-\ga,\ga)\subset\mbb{R}$ and that the moment functions  are given as
\begin{equation}\label{eq:alg_basis}
R_l(x) = x^{l},\quad l= 0, 1,\ldots\ ,
\end{equation}
so that the restriction operator $\res_L$ extracts the first $L$ algebraic moments.

For a given prior probability density $\prior\in L^2(-\ga,\ga)$ the solution to \eqref{eq:primal_L2} is
\begin{equation}\label{eq:opt_L2N_mon}
\vph_2(\mbf{m},\prior)(x) = \bsm{\la}^T\mbf{R}(x)+\prior(x) = \sum_{l=0}^{L}\la_{l}x^{l} + \prior(x),
\end{equation}
where $\bsm{\la}= H_{L}^{-1}\otilde{\res}_L(\targ-\prior)$ and according to \eqref{eq:hess_L2N}
\begin{equation}
H_{i,j} = \frac{1+(-1)^{i+j}}{i+j+1}\ga^{i+j+1}
\end{equation}
for $i,j=0,\ldots,L$. The set $\{R_l:\ l=0,1,\ldots\}\subset L^{\infty}(-\ga,\ga)$ is the monomial basis of  $L^2(-\ga,\ga)$ and its Gram-Schmidt orthonormalisation $\leg_l\in L^{\infty}(-\ga,\ga)$, $l=0,1,\ldots\ $, satisfies $\leg_l(x) = P_l(x/\ga) / \sqrt{\ga}$ with ${P}_l$ the $l$-th (normalised) Legendre polynomial  on $(-1,1)$.

Before stating the main result of this Section, Theorem~\ref{thm:cons_L2N}, let us establish two supporting Lemmas.
\begin{lemma}\label{lem:fourier_int}
Let $\ph\in C^2([-\ga,\ga])$ and let $\ph_l=\int_{-\ga}^{\ga}\leg_l(x)\ph(x)\,\der{x}$, $l=0,1,\ldots,$ denote its Fourier coefficients with respect to $\{\leg_l\}$. Then
\begin{equation}\label{eq:fourier_int}
\ga\sum_{l=0}^{+\infty}l(l+1)|\ph_l|^2 = \int_{-\ga}^{\ga}\left(\ga^2-x^2\right)|\ph'(x)|^2\,\der{x}.
\end{equation}
\end{lemma}
The proof is a straightforward generalisation of the proof in \cite{Talenti1987} that uses the relation between $Q_l$ and the Legendre polynomials.
\begin{remark}(\added[id=r1,remark={4}]{On generalisation to multidimensional $G$})
\added[id=r1,remark={4}]{The subsequent analysis hinges upon Equation~\eqref{eq:fourier_int}, a consequence of Legendre's differential equation. The extension to the $d$-dimensional case would require the study of similar relations for the orthogonalisation of multivariate algebraic polynomials from Example~\ref{ex:alg_moms}. 
}
\end{remark}
For the second Lemma let us denote by $W^{1,2}([-\ga,\ga])$ the Sobolev space of all functions $\ph\in L^2([-\ga,\ga])$ such that $\ph'\in L^2([-\ga,\ga])$ with norm $\|\ph\|_{1,2}=\|\ph\|_{2}+\|\ph'\|_{2}$.\begin{lemma}\label{lem:L2N_sob_cont}
Fix $L\in\mbb{N}$, $\targ\in\pds^2([-\ga,\ga])$ and let $\mbf{m}_L=\res_L\targ$. The projection operator $\vph_2(\mbf{m}_L,\cdot\,)$ is continuous in $W^{1,2}([-\ga,\ga])$.
\end{lemma}
\begin{proof}
Let us take $(\ph_n)_{n=0}^{+\infty}\subset W^{1,2}([-\ga,\ga])$ such that $\lim_{n\to\infty}\|\ph_n-\ph_0\|_{1,2}=0$. According to~\eqref{eq:opt_L2N_mon}, $\vph_2(\mbf{m}_L,\ph_n)$ is the sum of $\ph_n$ and a polynomial, so it belongs to $W^{1,2}([-\ga,\ga])$ and we obtain
\begin{align*}
\|\vph_2(\mbf{m}_L,\ph_n)'-\vph_2(\mbf{m}_L,\ph_0)'\|_{2}
&\leq\|\ph_n'-\ph_0'\|_{2} + \bigg\|\sum_{l=0}^{L-1}(l+1)(\la^n_{l+1}-\la^0_{l+1})R_l\bigg\|_{2}\\
&\leq \|\ph_n'-\ph_0'\|_{2} + L\sum_{l=0}^{L-1}|\la^n_{l+1}-\la^0_{l+1}|\cdot\|R_l\|_{2}\\
&\leq\|\ph_n'-\ph_0'\|_{2} + L\|\bsm{\la}^n-\bsm{\la}^0\|\cdot\|\mbf{R}\|_{2}
\end{align*}
where $\bsm{\la}^n=H_{L}^{-1}(0,\mbf{m}_{L}-\res_L\ph_n)$ and the last estimate is based on the H\"{o}lder inequality. Moreover, using the estimate on the norm of $\otilde{\res}_{L}$ we have
\begin{equation*}
\|\bsm{\la}^n-\bsm{\la}^0\|\leq\|H_{L}^{-1}\|\cdot\|\otilde{\res}_{L}(\ph_n-\ph_0)\|\leq\|H_{L}^{-1}\|\cdot\|\mbf{R}\|_{2}\cdot\|\ph_n-\ph_0\|_{2}
\end{equation*}
with $\|H_{L}^{-1}\|$ the matrix norm. Combining these two estimates together with the non-expansiveness of $\vph_2(\mbf{m}_{L},\cdot)$ gives
\begin{equation*}
\|\vph_2(\mbf{m}_L,\ph_n)-\vph_2(\mbf{m}_L,\ph_0)\|_{1,2}\leq
\big(1+L\|H_{L}^{-1}\|\cdot\|\mbf{R}\|_{2}^{2}\big)\|\ph_n-\ph_0\|_{2} +
\|\ph_n'-\ph_0'\|_{2},
\end{equation*}
from which the continuity follows.
\end{proof}

We are now ready to prove the following Theorem on the rate of convergence for $L^2$-norm based matching.
\begin{theorem}\label{thm:cons_L2N}
Let $L\in\mbb{N}$ and let $\prior, \targ\in W^{1,2}([-\ga,\ga])$ be two probability densities. If $\mbf{m}_L=\res_L\targ$ we have
\begin{equation}\label{eq:cons_L2N}
\|\vph_2(\mbf{m}_L,\prior) - \targ\|_2\leq\frac{\sqrt{\ga}}{L+1}\|\prior' - \targ'\|_2.
\end{equation}
\end{theorem}
\begin{proof}
First assume that $\prior,\targ\in C^2([-\ga,\ga])$ and let $\prior_l,\targ_l$ denote the $l$-th Fourier coefficients of $\prior$ and $\targ$ respectively in the basis $\{Q_l\}_{l=0,1,\ldots}$ (which is the orthonormalisation of \eqref{eq:alg_basis}). Then, using expansion \eqref{eq:optimal_L2N_orth} and Parseval's identity, we get
\begin{align*}
\|\vph_2(\mbf{m}_L,\prior) - \targ\|^2_2 = \sum_{l=L+1}^{+\infty}(\targ_l-\prior_l)^2
&\leq\sum_{l=L+1}^{+\infty}\frac{l(l+1)}{(L+1)^2}(\targ_l-\prior_l)^2\\
&\leq\frac{1}{(L+1)^2}\sum_{l=0}^{+\infty}l(l+1)(\targ_l-\prior_l)^2.
\end{align*}
Finally, employing \eqref{eq:fourier_int} from Lemma~\ref{lem:fourier_int}, we obtain
\begin{align*}
\|\vph_2(\mbf{m}_L,\prior) - \targ\|^2_2
&\leq\frac{1}{\ga(L+1)^2}\int_{-\ga}^{\ga}\left(\ga^2-x^2\right)|\targ'(x)-\prior'(x)|^2\,\der{x}\\
&\leq\frac{\ga}{(L+1)^2}\|\targ'-\prior'\|_2^2.
\end{align*}

Now let $\prior,\targ\in W^{1,2}([-\ga,\ga])$ and take two sequences $(\prior_n)_{n},(\targ_n)_{n}\in C^2([-\ga,\ga])$ such that $\prior_n\to\prior$ and $\targ_n\to\targ$ in $W^{1,2}$. Then $\vph_2(\mbf{m}_{L},\prior_n),\,\vph_2(\mbf{m}_{L},\targ_n)$ are twice continuously differentiable for each $n$, and from Lemma~\ref{lem:L2N_sob_cont} it follows that $\vph_2(\mbf{m}_{L},\prior_n)\to\vph_2(\mbf{m}_{L},\prior)$ and $\vph_2(\mbf{m}_{L},\targ_n)\to\vph_2(\mbf{m}_{L},\targ)=\targ$ in $W^{1,2}$. Hence, using the first part of the proof together with $\mbf{m}_{L}=\res_L\vph_2(\mbf{m}_{L},\targ_n)$, we get
\begin{multline*}
\|\vph_2(\mbf{m}_L,\prior) - \targ\|^2_2 =
\lim_{n\to\infty}\|\vph_2(\mbf{m}_{L},\prior_n)-\vph_2(\mbf{m}_{L},\targ_n)\|_{2}\\
\leq\frac{\sqrt{\ga}}{L+1}\lim_{n\to\infty}\|\prior_n'-\vph_2(\mbf{m}_{L},\targ_n)'\|_{2}=
\frac{\sqrt{\ga}}{L+1}\|\prior'-\targ'\|_{2},
\end{multline*}
which finishes the proof of~\eqref{eq:cons_L2N}.
\end{proof}

From Theorem~\ref{thm:cons_L2N} the following corollary follows that implies~\eqref{eq:match_cons_SDE} from Property~\ref{prop:consist-match}. Recall our convention for notation from Remark~\ref{rem:res_match_rv}.
\begin{corollary}\label{cor:cons}
Let $X$ be a solution of~\eqref{eq:sde} with a smooth density $\p$. For \mbox{every} $t,t+\De t\in I$, with $\De t>0$, put $\prior=\p(t,\cdot)$, $\targ=\p(t+\De t,\cdot)$ and assume that $\vph_2(\mbf{m}_L,\prior)\geq0$, where $\mbf{m}_L=\res_L\targ$. Then for every $g\in C([-\ga,\ga])$ we have
\begin{equation}\label{eq:match_L2_cons}
\big|\Exp g\big(\vph_2(\mbf{m}_L,X(t))\big) - \Exp g\big(X(t+\De t)\big)\big|\leq\frac{C}{L+1}\De t,
\end{equation}
with $C$ depending only on $g$, $\p$, $T$ and $\ga$.
\end{corollary}
\begin{proof}
Using formula~\eqref{eq:avg_int} we obtain
\begin{align*}
\big|\Exp g\big(\vph_2(\mbf{m}_L,X(t))\big) - \Exp g\big(X(t+\De t)\big)\big|
&\leq \int_{-\ga}^{\ga}|g(y)||\vph_2(\mbf{m}_{L},\prior) - \targ|\,\der{y}\\
&\leq 2\ga\cdot\max_{[-\ga,\ga]}|g|\cdot\|\vph_2(\mbf{m}_{L},\prior) - \targ\|_{2}
\end{align*}
According to Theorem~\ref{thm:cons_L2N}, it is enough to estimate the $L^2$ norm between the derivatives of $\prior$ and $\targ$. Applying the mean value theorem, we get
\begin{equation*}
|\pder[x]\p(t+\De t,x) - \pder[x]\p(t,x)|\leq \max_{[0,T]\times[-\ga,\ga]}|\pder[t]\pder[x]\p|\cdot\De t,
\end{equation*}
for every $x\in(-\ga,\ga)$. Hence
\begin{equation*}
\|\targ' - \prior'\|_{2}^2 = \int_{-\ga}^{\ga}|\pder[x]\p(t+\De t,x) - \pder[x]\p(t,x)|^2\,\der{x}\leq
2\ga\cdot\hspace{-1em}\max_{[0,T]\times[-\ga,\ga]}|\pder[t]\pder[x]\p|^2\cdot(\De t)^2
\end{equation*}
Combining these estimates with Theorem~\ref{thm:cons_L2N} gives the result.
\end{proof}
We provide a sufficient condition for the existence and smoothness of $\p$ in Section~\ref{sec:numerics} for the specific example of FENE dumbbells.

\subsection{Matching with \texorpdfstring{$f$}{f}-divergences}\label{sec:fdiv}
Let $p\in[1,+\infty)$ and let $\res_L\from\pds^{p}(G)\to\mbb{R}^L$ be the restriction operator~\eqref{eq:restr_with_functions} generated by the functions $R_1,\ldots,R_l\in L^p(G)$. In this Section, we consider matching operators that for a vector of macroscopic state variables ${\mbf{m}\in\im\res_L}$ and a prior distribution $\prior\in\pds^{p}(G)$ are defined as
\begin{equation}\label{eq:match_div}
\match_L(\mbf{m},\prior) = \argmin_{\vph\in\res_L^{-1}(\mbf{m})}\I(\vph\,|\prior),
\end{equation}
where $\I(\,\cdot\,|\prior)$ is a \emph{divergence functional} generated by ${f\from[0,+\infty)\to[0,+\infty)}$,
\begin{equation}\label{eq:fdiv-func}
\I(\vph|\prior)=\left\{\begin{array}{ll}
\displaystyle\int_{G}f\!\left(\frac{\vph(\mbf{x})}{\prior(\mbf{x})}\right)\prior(\mbf{x})\,\der{\mbf{x}} & \text{if}\ \supp\vph\subseteq\supp\prior,\\
\\
\displaystyle+\infty & \text{otherwise,}
\end{array}\right.\quad \vph\in\pds^{p}(G).
\end{equation}
Note that the condition $\supp\vph\subseteq\supp\prior$ is equivalent to requiring that the measure $\mu_{\vph}$ corresponding to the density $\vph$ is absolutely continuous with respect to the measure $\mu_{\pi}$ corresponding to $\pi$, which we denote as $\mu_{\vph}\ll\mu_{\pi}$.

We assume that $f$ is not identically equal to zero, lower semi-continuous, convex and such that $f(1)=0$. Let us denote by $f_+$ the extension of $f$ to the whole real line such that $f_+(t)=+\infty$ for negative $t$. If we consider the probability measure $\mu_{\prior}$ (generated by~$\prior$) and put $\ph = \vph/\prior\in L^p(G,\mu_{\prior})$ for each $\vph\in\pds^{p}(G)$ such that $\mu_{\vph}\ll \mu_{\pi}$, we see that \eqref{eq:match_div} is equivalent to the \emph{primal (entropy) problem} $(EP)_p$ in $L^p(G,\mu_{\prior})$, see Appendix~\ref{app:optimality}, with vector $\otilde{\mbf{m}}=(1,\mbf{m})\in\mbb{R}^{L+1}$. Recall that in $(EP)_p$ we ensure that the optimal solution, if it exists, is a probability density by using the extended restriction operator $\otilde{\res}_L$, with additional moment function $R_0\equiv 1$, and a~convex constraint to the cone $L_{+}^{p}(G,\mu_{\prior})$.

If all assumptions of Theorem~\ref{thm:duality} are satisfied, it follows from \eqref{eq:primal_opt} that the explicit formula for the matching operator defined by \eqref{eq:match_div} is given by
\begin{equation}\label{eq:matching_explicit_general}
\match_L(\mbf{m},\prior) = (\f_+^*)'\!\left(\sum_{l=0}^{L}\oline{\la}_lR_l\right)\cdotp\prior,
\end{equation}
where $(\f_+^*)'$ is the derivative of the convex conjugate $f_+^*$ (defined in~\eqref{eq:conv_conj}) and  the multipliers $\oline{\la}_0,\ldots,\oline{\la}_L$ satisfy the Lagrangian dual problem (cf.~\eqref{eq:grad_dof})
\begin{equation}\label{eq:dp_text}
	\otilde{\mbf{m}} - \otilde{\res}_L\!\left((f_+^*)'\!\Big(\sum_{l=0}^{L}\oline{\la}_lR_l\Big)\right)=0.
\end{equation}
The derivation of this dual problem is presented in Appendix~\ref{app:optimality}. The general result on the existence of solution is given as Theorem~\ref{thm:duality}, which is a specialisation of a result from \cite{BorLew1991b} to our setting. In particular, note that condition~\eqref{eq:CQ} requires additionally the vector $\mbf{m}$ to lie in the interior of the image of $\res_L$ (cf.~Remark~\ref{rem:match_domain})

We now consider two specific examples: the Kullback-Leibler divergence and the $L^2$-divergence.  While we do not show rigorously that Property~\ref{prop:consist-match} holds for these matching operators at this point, we will illustrate in Section~\ref{sec:numerics} that the numerical behaviour of these matchings is very similar to that of the $L^2$-norm matching of Section~\ref{sec:l2n}, while offering two significant advantages.  First, $f$-divergence matching guarantees positivity of the resulting matched probability density. Second, $f$-divergence matching leads to numerical methods that are easier to implement numerically for ensembles of finite size.

\paragraph{Kullback-Leibler divergence}
Here, we take as generating function
\begin{equation}\label{eq:f-kld}
f(t) = \left\{\begin{array}{ll}
t\ln t- t+1 & t>0,\\
1 & t=0,
\end{array}\right.
\end{equation}
and consider the resulting divergence distance on $\pds(G)$ (thus $p=1$ in this case). With this choice of $f$ the resulting functional is (the term generated by $-t+1$ cancels out)
\[
\I(\vph|\prior) = \int_G\frac{\vph(\mbf{x})}{\prior(\mbf{x})}\ln\left(\frac{\vph(\mbf{x})}{\prior(\mbf{x})}\right)\,\prior(\mbf{x})\der{\mbf{x}},\quad \mu_{\vph}\ll\mu_{\prior},
\]
which is called Kullback-Leibler divergence (KLD) or logarithmic relative entropy. \added[id=r1,remark={1}]{KLD is a common choice in the information-theoretic methods for the analysis of stochastic models. For example, the entropy optimization principle is used to obtain moment closures~\cite{IlgKarOtt2002,HauLevTit2008}, construct optimal coarse-grained dynamics~\cite{KatPle2013,HarKalKatPle2016}, and approximate the spectral densities~\cite{GeoLin2003,FerRamTic2011}. Note that, in contrast with some other approaches, we focus on the convergence of the micro-macro acceleration method to the true microscopic dynamics.}

The dual of $f_+$ is $f_+^*(s) = \exp(s)-1$, $s\in\mbb{R}$. Note that $f$ satisfies the assumptions in Theorem~\ref{thm:duality}. Hence, if the moment functions $R_1,\ldots,R_l$ fulfil Assumption~\ref{ass:hierarchy}(a) and~\eqref{eq:CQ} holds,  the formula for matching~\eqref{eq:matching_explicit_general} reads
\begin{equation}\label{eq:match_KLD}
\match_L(\mbf{m},\prior) = \exp\!\left(\sum_{l=0}^L\oline{\la}_lR_l\right)\cdotp\prior,
\end{equation}
with $\oline{\la}_0,\ldots,\oline{\la}_L$ the solution to the non-linear system of integral equations (recall $R_0\equiv 1$)
\begin{equation}\label{eq:lm_KLD}
\left\{\begin{array}{l}
\displaystyle e^{\la_0}\int_GR_l\exp\!\left(\sum_{k=1}^L\la_k R_k\right)\prior\,\der{\mbf{x}} = m_l,\quad l=1,\ldots,L,
\\
\displaystyle\la_0 = - \ln\int_G\exp\!\left(\sum_{k=1}^L\la_k R_k\right)\prior\,\der{\mbf{x}}.
\end{array}\right.
\end{equation}
It is not possible to analytically solve \eqref{eq:lm_KLD}, so we will use a Newton-Raphson procedure to perform the optimisation numerically, see Appendix~\ref{app:newt}.

\begin{remark}[Reducing dimensions for KLD based matching]
To keep the mass of the matching equal to $1$ we include the additional linear constraint with moment function $R_0$ and obtain $0$-th Lagrange multiplier $\la_0$ in the dual problem. This implies that we perform the optimisation in $L+1$ dimensions and that the resulting numerical solution has mass one only up to the tolerance of the Newton-Raphson procedure. With KLD based matching we can express $\la_0$ as a function of the other multipliers, see \eqref{eq:lm_KLD}, thus reducing the dimensionality by one and keeping the mass of the numerical solution (and in fact solutions corresponding to all steps of Newton-Raphson iteration) equal to one up to the machine precision. However, in this case we observed that the convergence of the Newton-Raphson procedure is slower compared to $L+1$ dimensional optimisation, so we decided not to use this approach.
\end{remark}

\paragraph{$L^2$ divergence}
Our second example is based on the generating function
\begin{equation}
f(t) = \frac{1}{2}(t-1)^2,\quad t\geq0.
\end{equation}
Hence, if $\vph,\prior\in\pds^2(G)$ and $\mu_{\vph}\ll\mu_{\prior}$, the divergence distance under consideration is
\[
\I(\vph\,|\prior) = \frac{1}{2}\int_{G}\left(\frac{\vph(\mbf{x})}{\prior(\mbf{x})}-1\right)^{\!2}\,\prior(\mbf{x})\der{\mbf{x}},
\]
which is called $L^2$ divergence (L2D) or quadratic relative entropy. The dual of $f_+$ in this case is $f_+^*(s) = 1/2\,[\max(0,s+1)^2-1]$. Since $f$ satisfies the assumptions in Theorem~\ref{thm:duality}, under Assumption~\ref{ass:hierarchy}(a) and~\eqref{eq:CQ}, the formula for matching~\eqref{eq:matching_explicit_general} reads
\begin{equation}\label{eq:match_L2D}
\match_L(\mbf{m}, \prior) = \max\left(0, \sum_{l=0}^{L}\oline{\la}_lR_l\right)\cdot\prior,
\end{equation}
where $\oline{\la}_0,\ldots,\oline{\la}_L$ solve
\begin{equation}\label{eq:lm_L2D}
m_l - \int_GR_l(\mbf{x})\max\left(0,\,\sum_{k=0}^{L}\la_kR_k(\mbf{x})\right)\prior(\mbf{x})\,\der{\mbf{x}} = 0,\quad l=0,\ldots,L.
\end{equation}
We again need to find the solution to~\eqref{eq:lm_L2D} numerically; see Appendix~\ref{app:newt} for the corresponding Newton-Raphson procedure.

\section{Proof of the convergence result}\label{sec:convergence-proof}
Throughout this Section we use the notation of Section~\ref{sec:alg} and suppose that all the assumptions of Theorem~\ref{thm:convergence} hold. The proof requires the following Lemmas:
\begin{lemma}\label{lem:bounded}
For every sufficiently large $r$ there is a constant $C_{L,r,T}>0$ such that
\[
\Exp\|\mbf{X}^{n}\|^{2r}\leq C_{L,r,T}(1+\|\mbf{X}^0\|^{2r})
\]
for all $N,K$ and $n=1,\ldots,N,\ k=1,\ldots,K$.
\end{lemma}
\begin{proof}
Fix $n<N$. According to the uniform boundedness of the moments for the one-step microscopic discretisation $S$ (Assumption~\ref{ass:tdiscr_moments}), for every sufficiently large $r>0$ there is a constant $C_{r,T}>0$ such that $\Exp\|X^{n,k}\|^{2r}\leq C_{r,T}\Exp\|X^{n,0}\|^{2r}$ for all $k=1,\ldots,K$. Let us also take constants $\ohat{C}_{l,T}$, $\ohat{\ka}_{l,T}$ from Lemma~\ref{lem:tdiscr_cons_inter} and $\otilde{C}_{l,T}$, $\otilde{\ka}_{l,T}$ from estimate~\eqref{eq:sde_wcont}, corresponding to the functions $R_l$, $l=1,\ldots,L$. For each $r$ sufficiently large it holds with $C_{L,r,T}=2(\ohat{C}_{l,T}+\otilde{C}_{l,T})$
\[
\ohat{C}_{l,T}(1+\|\mbf{z}\|^{2\ohat{\ka}_{l,T}})+\otilde{C}_{l,T}(1+\|\mbf{z}\|^{2\otilde{\ka}_{l,T}})\leq C_{L,r,T}(1+\|\mbf{z}\|^{2r})
\]
for all $\mbf{z}\in G$ and $l=1,\ldots,L$.

Let us now fix $r$ large enough so that the discussion in the previous paragraph applies. For every $l\leq L$, using Lemma~\ref{lem:tdiscr_cons_inter} and formula~\eqref{eq:sde_wcont}, we have
\begin{gather}\label{eq:mom_diff_est}
\begin{multlined}
\left|\Exp R_l(\mbf{X}^{n,K})-\Exp R_l(\mbf{X}^{n,0})\right|
\leq\left|\Exp R_l(\mbf{X}^{n,K})-\Exp R_l(\mbf{X}(t^{n,K};t^{n,0},\mbf{X}^{n,0}))\right|\\[0.5em]
+\left|\Exp R_l(\mbf{X}(t^{n,K};t^{n,0},\mbf{X}^{n,0}))-\Exp R_l(\mbf{X}^{n,0})\right|
\leq\ohat{C}_{l,T}(1+\Exp\|\mbf{X}^{n,0}\|^{2\ohat{\ka}_{l,T}})K(\de t)^{p+1}\\[0.5em]
+\otilde{C}_{l,T}(1+\Exp\|\mbf{X}^{n,0}\|^{2\otilde{\ka}_{l,T}})K\de t\leq C_{L,r,T}(1+\Exp\|\mbf{X}^{n,0}\|^{2r})K\de t,
\end{multlined}
\end{gather}
where we estimated $(\de t)^{p}\leq T^p$ and absorbed the factor $T^p$ into the constant $C_{L,r,T}$.

Recall that $\mbf{X}^{n+1}=\match_{L}(\mbf{m}^{n+1},\mbf{X}^{n,K})$, from the definition in Algorithm~\ref{alg:accel}, and $\mbf{X}^{n,K}=\match_{L}(\mbf{m}^{n,K},\mbf{X}^{n,K})$, from the properties of matching. Thus employing the continuity of matching (Property~\ref{prop:cont-match}) and formula~\eqref{eq:CFE} (for the extrapolation) we obtain
\begin{multline*}
\Exp\|\mbf{X}^{n+1}\|^{2r}\leq\left|\Exp\|\mbf{X}^{n+1}\|^{2r}-\Exp\|\mbf{X}^{n,K}\|^{2r}\right|+\Exp\|\mbf{X}^{n,K}\|^{2r}\\
\leq C_{L,r}\left\|\mbf{m}^{n+1}-\mbf{m}^{n,K}\right\|
+C_{r,T}\Exp\|\mbf{X}^{n,0}\|^{2r}\\
\leq C_{L,r}\frac{\De t}{K\de t}\left\|\mbf{m}^{n,K}-\mbf{m}^{n,0}\right\|+C_{r,T}\Exp\|\mbf{X}^{n}\|^{2r}.
\end{multline*}
Since $\|\mbf{m}^{n,K}-\mbf{m}^{n,0}\|\leq L\cdot\max_{l\leq L}|\Exp R_l(\mbf{X}^{n,K})-\Exp R_l(\mbf{X}^{n,0})|$, from~\eqref{eq:mom_diff_est} we get
\[
\Exp\|\mbf{X}^{n+1}\|^{2r}\leq L\,C_{L,r}C_{L,r,T}\De t(1+\Exp\|\mbf{X}^{n}\|^{2r})+C_{r,T}\Exp\|\mbf{X}^{n}\|^{2r}.
\]
This recurrence implies by~\cite[Lem.~1.3.]{Milstein1995} that $\Exp\|\mbf{X}^{n+1}\|^{2r}\leq e^{C}(1+\Exp\|\mbf{X}^{0}\|^{2r})$ for a~constant $C$ depending only on $r$, $L$ and $T$.
\end{proof}
\begin{lemma}\label{lem:sde_extrap}
There is a constant $C_{L,T}^{\ext}$ such that for all $N$, $K$ and $n=0,\ldots,N$, $k=0\ldots,K$
\[
\|\ext\big((\res_{L}\mbf{X}(t^{n,k};t^{n},\mbf{X}^{n}))_{k=0}^{K};\De t,\de t\big) - \res_{L}\mbf{X}(t^{n+1};t^{n},\mbf{X}^{n})\|\leq C_{L,T}^{\ext}\cdot(\De t)^2.
\]
\end{lemma}
\begin{proof}
Put $\otilde{\mbf{m}}^n(s)= \res_{L}\mbf{X}(s;t^{n},\mbf{X}^{n})$, $s\in[t^{n},t^{n+1}]$. Assumption~\ref{ass:sde}.\ref{ass:sde_forw} together with the fact that $R_l\in C_{P}^{4}(G,\mbb{R})$ imply: $\otilde{\mbf{m}}^n\in C^2([t^{n},t^{n+1}],\mbb{R})$. Thus from Lemma~\ref{lem:extrap} we obtain
\begin{multline*}
\|\ext\big((\res_{L}\mbf{X}(t^{n,k};t^{n},\mbf{X}^{n}))_{k=0}^{K};\De t,\de t\big) - \res_{L}\mbf{X}(t^{n+1};t^{n},\mbf{X}^{n})\|\\
\leq\max_{s\in[t^{n},t^{n+1}]}\|(\otilde{\mbf{m}}^n)''(s)\|\cdot(\De t)^2
\end{multline*}
and from Lemma~\ref{lem:bounded} we can find $\ka_{L,T}$ and $C_{L,T}$ large enough such that
\begin{multline*}
\max_{s\in[t^{n},t^{n+1}]}\|(\otilde{\mbf{m}}^n)''(s)\|\leq
L\,\cdot\hspace{-1em}\max_{s\in[t^{n},t^{n+1}]}\max_{l\leq L}\big|\pder[s]^{2}\,\Exp R_l\big(\mbf{X}(s;t^{n},\mbf{X}^{n})\big)\big|\\
\leq L\cdot\max_{l\leq L} \otilde{C}_{l,T}\big(1+\Exp\|\mbf{X}^{n}\|^{2\otilde{\ka}_{l,T}}\big)^2\leq C_{L,T}(1+\Exp\|\mbf{X}^{0}\|^{2\ka_{L,T}})^2,
\end{multline*}
where $\otilde{C}_{l,T}$, $\otilde{\ka}_{l,T}$ are constants from Assumption~\ref{ass:sde}.\ref{ass:sde_forw} (corresponding to $R_{l}$).
\end{proof}

We now proceed to the proof of the Theorem. We will bound the weak error between random variable $\mbf{X}^N=\mbf{X}^{N}_{L}$, obtained from micro-macro Algorithm~\ref{alg:accel}, and the true solution $\mbf{X}(T)=\mbf{X}(T;t^0,\mbf{X}^0)$ of~\eqref{eq:sde} in three steps. In the first part, we derive a recursion for the error at time $t^{n}=t^{n,0}$, $n=1,\ldots,N$. In the second, we bound the local truncation error. Finally, in the third part we gather the estimates to obtain \eqref{eq:error_mic-mac}.

\paragraph{Part 1}
Let us fix $g\in C_P^{2(p+1)}(G,\mbb{R})$ and $L\in\mbb{N}$. From Assumption~\ref{ass:sde}.\ref{ass:sde_init} we have
\begin{align}\label{eq:total_err}
\left|\Exp g(\mbf{X}^N) - \Exp g(\mbf{X}(T))\right|\nonumber
&\leq\sum_{n=1}^N\left|\Exp g\big(\mbf{X}(T;t^n,\mbf{X}^n)\big) - \Exp g\big(\mbf{X}(T;t^{n-1},\mbf{X}^{n-1})\big)\right|\\
&\leq C_{g,T}\sum_{n=1}^N\left|\Exp g\big(\mbf{X}^n\big) - \Exp g\big(\mbf{X}(t^{n};t^{n-1},\mbf{X}^{n-1})\big)\right|.
\end{align}
Here we used the identity $\mbf{X}(T;t^{n-1},\mbf{X}^{n-1})=\mbf{X}(T;t^{n},\mbf{X}(t^{n};t^{n-1},\mbf{X}^{n-1}))$. The terms in the last sum of~\eqref{eq:total_err} are exactly the errors produced by one step of the micro-macro acceleration method when there is no initial error, i.e., they are the local truncation errors.

\paragraph{Part 2}
Let us now fix $n\in\{1,\ldots,N\}$ and estimate a single term in the righthand side of~\eqref{eq:total_err}. Recall that Algorithm~\ref{alg:accel} first computes an intermediate solution $\mbf{X}^{n-1,K}$ by performing $K$ steps of the microscopic scheme at times $t^{n,k}$ ($k=1,\ldots,K$), and starting from $\mbf{X}^{n-1}$ at time $t^{n-1}=t^{n-1,0}$ (see \eqref{eq:alg-micro-step}).  Hence we split
\begin{multline}\label{eq:local_err_split}
\Exp g\big(\mbf{X}^n\big) - \Exp g\big(\mbf{X}(t^{n};t^{n-1},\mbf{X}^{n-1})\big)\\
=\Exp g\big(\mbf{X}^n\big) - \Exp g\big(\mbf{X}(t^n; t^{n-1,K},\mbf{X}^{n-1,K})\big)\\
+\Exp g\big(\mbf{X}(t^n; t^{n-1,K},\mbf{X}^{n-1,K})\big) - \Exp g\big(\mbf{X}(t^{n};t^{n-1},\mbf{X}^{n-1})\big),
\end{multline}
and use Lemmas~\ref{lem:tdiscr_cons_inter} (with $s=t^{n}$) and~\ref{lem:bounded} together with $K\de t\leq\De t$ to get
\begin{equation}\label{eq:scheme_conv}
\left|\Exp g\big(\mbf{X}(t^n; t^{n-1,K},\mbf{X}^{n-1,K})\big) - \Exp g\big(\mbf{X}(t^{n};t^{n-1},\mbf{X}^{n-1})\big)\right|\leq
\bar{C}_{L,T}\De t(\de t)^{p_S}
\end{equation}
with constant $\bar{C}_{L,T}$ depending also on $\mbf{X}^0$. To estimate the first summand in \eqref{eq:local_err_split} we further decompose
\begin{multline*}
 \Exp g\big(\mbf{X}^n\big) - \Exp g\big(\mbf{X}(t^n; t^{n-1,K},\mbf{X}^{n-1,K})\big) =
\Exp g\big(\mbf{X}^n\big) - \Exp g\big(\match_L(\otilde{\mbf{m}}^n,\mbf{X}^{n-1,K})\big)\\
+\Exp g\big(\match_L(\otilde{\mbf{m}}^n,\mbf{X}^{n-1,K})\big) - \Exp g\big(\mbf{X}(t^n; t^{n-1,K},\mbf{X}^{n-1,K})\big),
\end{multline*}
where $\otilde{\mbf{m}}^n = \res_L\mbf{X}(t^n;t^{n-1,K},\mbf{X}^{n-1,K})$. Using \eqref{eq:match_cons_SDE} in Property~\ref{prop:consist-match}, the second term can be estimated as follows
\begin{equation}
\left|\Exp g\big(\match_L(\otilde{\mbf{m}}^n,\mbf{X}^{n-1,K})\big) - \Exp g\big(\mbf{X}(t^n; t^{n-1,K},\mbf{X}^{n-1,K})\big)\right|\leq C_L\De t,
\end{equation}
where $C_L$ depends on $g$ and $C_L\to 0$ as $L\to+\infty$. To estimate the first term recall that $\mbf{X}^n = \match_L(\mbf{m}^n,\mbf{X}^{n-1,K})$ where $\mbf{m}^n$ is obtained via the extrapolation step~\eqref{eq:alg-extrap-step}.  We can use the continuity of the matching operator (Property~\ref{prop:cont-match}) to get
\begin{multline}\label{eq:intermediate-1}
\left|\Exp g\big(\mbf{X}^n\big) - \Exp g\big(\match_L(\otilde{\mbf{m}}^n,\mbf{X}^{n-1,K})\big)\right|\leq \hat{C}_L\left\|\mbf{m}^n-\otilde{\mbf{m}}^n\right\|\\
\leq\hat{C}_L\left\|\mbf{m}^n-\ext\left((\res_L\mbf{X}(t^{n-1,k};t^{n-1},\mbf{X}^{n-1}))_{k=0}^{K},\de t,\De t\right)\right\|\\
+\hat{C}_L\left\|\ext\left((\res_L\mbf{X}(t^{n-1,k};t^{n-1},\mbf{X}^{n-1}))_{k=0}^{K},\de t,\De t\right)-\otilde{\mbf{m}}^n\right\|.
\end{multline}
Next we apply the continuity of extrapolation, equation~\eqref{eq:cont-cfe}, together with Lemma~\ref{lem:tdiscr_cons_inter} (with functions $R_l$ and $s=t^{n-1,K}$) to the first term on the right hand side to obtain
\begin{multline}
\left\|\mbf{m}^n-\ext\left((\res_L\mbf{X}(t^{n-1,k};t^{n-1},\mbf{X}^{n-1}))_{k=0}^{K},\de t,\De t\right)\right\|\\
\leq \frac{\De t}{K\de t}\left\|\res_L\mbf{X}^{n-1,K}-\res_L\mbf{X}(t^{n-1,K};t^{n-1},\mbf{X}^{n-1})\right\|\\
\leq C_{L,T}(1+\Exp\|\mbf{X}^{n-1}\|^{\ka_{L,T}})\De t(\de t)^{p_S}\leq \bar{C}_{L,T}\De t(\de t)^{p_S},
\end{multline}
where in the last inequality we used Lemma~\ref{lem:bounded}. For the second term of~\eqref{eq:intermediate-1} we have
\begin{multline}\label{eq:intermediate-2}
\left\|\ext\left((\res_L\mbf{X}(t^{n-1,k};t^{n-1},\mbf{X}^{n-1}))_{k=0}^{K},\de t,\De t\right)-\otilde{\mbf{m}}^n\right\|\\
\leq\left\|\ext\left((\res_L\mbf{X}(t^{n-1,k};t^{n-1},\mbf{X}^{n-1}))_{k=0}^{K},\de t,\De t\right)-\res_L\mbf{X}(t^n;t^{n-1},\mbf{X}^{n-1})\right\|\\
+\left\|\res_L\mbf{X}(t^n;t^{n-1},\mbf{X}^{n-1}))-\otilde{\mbf{m}}^n\right\|.
\end{multline}
Lemma~\ref{lem:sde_extrap} establishes the estimate on the first term on righthand side of~\eqref{eq:intermediate-2}; for the second term we can once more use Lemma~\ref{lem:tdiscr_cons_inter} (with functions $R_l$ and $s=t^{n}$) together with Lemma~\ref{lem:bounded} to get
\begin{equation*}
\left\|\res_L\mbf{X}(t^n;t^{n-1},\mbf{X}^{n-1})-\otilde{\mbf{m}}^n\right\|\leq \bar{C}_{L,T}K(\de t)^{p_S+1}\leq \bar{C}_{L,T}\De t(\de t)^{p_S}.
\end{equation*}
Combining all these estimates, we obtain
\begin{equation}\label{eq:local_err}
\left|\Exp g\big(\mbf{X}^n\big) - \Exp g\big(\mbf{X}(t^{n};t^{n-1},\mbf{X}^{n-1})\big)\right|\leq C_L\De t + 3\bar{C}_{L,T}\De t(\de t)^{p_S} +C_{L,T}^{\ext}(\De t)^{2}.
\end{equation}

\paragraph{Part 3}
Using~\eqref{eq:total_err} and~\eqref{eq:local_err}, we get
\begin{align*}
\left|\Exp\left[g(\mbf{X}^N) - g(\mbf{X}(T))\right]\right|
&\leq C_{g,T}N\De t(C_L + 3\bar{C}_{L,T}(\de t)^{p_S} +C_{L,T}^{\ext}\De t)\\
&\leq C_{g,T}T(C_L + 2\bar{C}_{L,T}(\de t)^{p_S} +C_{L,T}^{\ext}\De t),
\end{align*}
where we used the estimate $N\De t\leq T$. To conclude the proof note that $C_{g,T}$ does not depend on $L$, so we can simply absorb the factor $C_{g,T}T$ into the other constants to obtain~\eqref{eq:error_mic-mac}.

\section{Numerical implementation of matching}\label{sec:num_match}
The numerical matching consists of a procedure that computes the estimate of the Lagrange multipliers parameterising the PDF of the matching operator (derived in Section~\ref{sec:match_oper}). In Section~\ref{sec:lagr_mult}, we detail the exact formulas used for the various types of matchings in consideration. Here, we assume that the vector of target moments is given, as well as a weighted ensemble of particles from the prior distribution. A complication arises due to the presence of noise, a consequence of the finite size of the ensemble. In Section~\ref{sec:avg_res}, we present the \emph{resampling strategy} used to avoid the degeneracy of weights and in Section~\ref{sec:match_fail} we describe the technique to adapt the macroscopic time step of our micro-macro method according to the performance of the numerical matching.

\subsection{Computation of Lagrange multipliers}\label{sec:lagr_mult}
Let $\mbf{m}\in\mbb{R}^L$ be a given vector of target moments corresponding to a restriction operator $\res_L$ generated by functions $R_1,\ldots,R_L$.
Assume also that we have a weighted ensemble $\ens{X}^{\prior}=(X_j, w_j)_{j=1}^{J}$ sampled from a prior distribution $\prior$ \added[id=r2,remark={maj2}]{absolutely continuous with respect to a Borel measure $\mu$ on $G\subset\mbb{R}$ (see Section~\ref{sec:match_oper_not})}, where $X_j$ is the $j$-th particle and $w_j$ is its normalised weight (all $w_j$ are non-negative and sum to $1$). \added[id=r2,remark={maj2}]{To be more precise, we suppose that each $X_j=X(\om_j)$ is a realisation of a random variable $X$ with law $\mu$, and the weights are obtained by taking $w_j=w(X_j)/\sum_{i=1}^J w(X_i)$, where $w$ is a density given by $\prior = w\cdot\mu$. Then, the \emph{importance sampling} theorem together with the Monte Carlo method from Section~\ref{sec:MC_sim}, yield the consistency of the approximation}
\begin{equation}\label{eq:wMC_approx}
\Exp[g(Y)] = \frac{\Exp[g(X)w(X)]}{\Exp[w(X)]} \approx \sum_{j=1}^{J}g(X_j)w_j,
\end{equation}
\added[id=r2,remark={maj2}]{for any test function $g$ and any random variable $Y$ distributed according to $\prior$.}


The matching operator is implemented by computing the appropriate Lagrange multipliers $\ohat{\bsm{\la}}_{\mbf{m}, \prior}$, which solve the corresponding dual optimisation problem, and results in a new weight $w_j(\ohat{\bsm{\la}}_{\mbf{m}, \prior})$ for each particle.
\added[id=r2,remark={maj3}]{The reweighing strategy is based on the explicit formulas for the matchings presented in Section~\ref{sec:match_oper}. These distributions belong to the parametrised family $w(\bsm{\la})\cdot\prior$, where $\bsm{\la}$ is any vector of Lagrange multipliers and $w(\bsm{\la})=w(\bsm{\la},\cdot)$ is a weighing function dependent on the type of matching, see~\eqref{eq:optimal_L2N} for L2N and~\eqref{eq:matching_explicit_general} for f-divergence based matching. Hence, if we now consider $Y$ distributed according to $w(\bsm{\la})\cdot\prior$, applying twice formula~\eqref{eq:wMC_approx} shows that}
\begin{equation}\label{eq:wMC_avg}
\sum_{j=1}^{J}g(X_j)w(\bsm{\la},X_j)w_j = \sum_{j=1}^{J}g(X_j)w_j(\bsm{\la}),
\end{equation}
\added[id=r2,remark={maj3}]{approximates $\Exp[g(Y)]$, where we define $w_j(\bsm{\la}) = w(\bsm{\la},X_j)w_j$. Accordingly, we consider $(X_j,w_j(\bsm{\la}))_{j=1}^{J}$ to sample $w(\bsm{\la})\cdot\prior$ and we use the weighted averages~\eqref{eq:wMC_avg}, with appropriate $g$, to compute the integrals in the formulas for the dual objective function, its gradient and Hessian. Note that for the solution $\ohat{\bsm{\la}}_{\mbf{m}, \prior}$, the resulting weights $w_j(\ohat{\bsm{\la}}_{\mbf{m}, \prior})$ will be already normalised when we use the extended restriction operator $\otilde{\res}_L$ from Section~\ref{sec:match_oper_not}.
}

\begin{remark}[On matching with empirical distributions]
\added[id=r2,remark={maj3,min2}]{Every weighted ensemble $(X_j,w_j)_{j=1}^{J}$ from $\prior$ gives rise to an empirical distribution $\prior_{J}=\sum_{j=1}^{J}w_j\de_{X_j}$, where $\de_{X_j}$ is the Dirac mass at $X_j$. Thus, in this case, we have $\mu= \sum_{j=1}^{J}\de_{X_j}/J$. The equivalent strategy of approximating $\match_{L}(\mbf{m},\prior)$ is to compute the matching $\nu_{J}=\match_{L}(\mbf{m},\prior_{J})$. Indeed, as all measures absolutely continuous with respect to $\prior_J$ are empirical and differ only by the weights, and we know that the matching, if it exists, reads $w(\bsm{\la})\cdot\prior_{J}$, for a particular $\bsm{\la}$, we easily obtain that $v_J = \sum_{j=1}^{J}w_j(\bsm{\la})\de_{X_j}$ is also a weighted empirical distribution. To find $\bsm{\la}$ we use the dual method that amounts to solving}
\begin{equation*}
m_l = \sum_{j=1}^{J}R_l(X_j)w_j(\bsm{\la}),\quad l=0,\ldots,L.
\end{equation*}
\added[id=r2]{This is the same problem we obtain by the Monte Carlo estimates for the Newton-Raphson procedure for $\prior$ (see Appendix~\ref{app:newt}) with the weighted ensemble $(X_j,w_j)_{j=1}^{J}$.}
\end{remark}

We now give the specific formulas for the three matching operators discussed so far.

\begin{example}\label{ex:L2N}
In the case of matching with $L^2$ norm, the Newton-Raphson procedure reduces to the appropriate least-squares problem, as described in Section~\ref{sec:l2n}. Hence, according to formula \eqref{eq:optimal_L2N}, the approximate Lagrange multipliers are computed as the numerical solution to the linear system of equations
\[
H\bsm{\la} = (0,\mbf{m}-\ohat{\mbf{m}}(\prior)),
\]
where $H$ is an $L+1\times L+1$ (Hessian) matrix with entries given by \eqref{eq:hess_L2N} and $\ohat{\mbf{m}}(\prior)$ is the Monte Carlo estimate of the moment vector of the prior
\[
\ohat{m}(\prior)_l =\sum_{j=1}^{J}R_l(X_j)w_j,\quad l=1,\ldots,L.
\]
We can compute the new weights associated with the matched sample as \added[id=r2,remark={3}]{(see~\eqref{eq:optimal_L2N})}
\[
w_j(\bsm{\la}) = w_j\cdot\left(\frac{\sum_{l=0}^{L}\la_lR_l(X_j)}{\prior(X_j)} + 1\right).
\]
However, we are facing two problems here. First, the non-negativity of these weights is guaranteed only when the assumptions of Lemma~\ref{lem:L2N_pos} hold. Second, we need to evaluate the prior distribution at all particle values $X_j$. Since we generally do not have a closed formula for $\prior$, this requires some density estimation (based on, e.g., histograms or kernel densities, see \cite{Scott2015}), and can therefore only be done approximately, potentially introducing additional statistical error and/or bias and increasing the time of computation. These complications do not arise with weights of the KLD and L2D based matchings (discussed in the next two examples), which makes them fit better the complete method.
\end{example}

\begin{remark}[Averages for L2N based matching]
For the numerical illustration in Section~\ref{sec:results_match}, we employ formula \eqref{eq:optimal_L2N} directly and represent the average of a function of interest $g\from G\to\mbb{R}$, with respect to the L2N numerical matching, as
\begin{equation}\label{eq:rem-l2n-moments}
\bar{g} = \sum_{l=0}^{L}(\ohat{\bsm{\la}}_{\mbf{m}, \prior})_l\int_{G}g(x)R_l(x)\,\der{x} + \int_Gg(x)\prior(x)\,\der{x}.
\end{equation}
To compute the estimate $\ohat{g}$, we evaluate $L$ deterministic integrals in the sum numerically and estimate the last integral using importance sampling with prior ensemble $\ens{X}^{\prior}$.
\end{remark}

For the matchings based on $f$-divergences, the approximate Lagrange multipliers are obtained via Newton-Raphson iteration applied to~\eqref{eq:dp_text}, with the additional constraint $\int_GR_0(x)\ph(x)\der{x}=1$, where $R_0\equiv 1$.
\begin{example}\label{ex:KLD}
For the matching based on KLD, one step of the Newton-Raphson iteration is given as
\begin{equation*}
\bsm{\la}^{\mrm{new}} = \bsm{\la}^{\mrm{old}} - \left(\ohat{\grad^2D}(\bsm{\la}^{\mrm{old}})\right)^{\!-1}\ohat{\grad D}(\bsm{\la}^{\mrm{old}}),
\end{equation*}
where (see Appendix~\ref{app:newt}) the Monte Carlo estimates of the gradient and Hessian of the objective function $D$ at point $\bsm{\la}\in\mbb{R}^{L+1}$ are for $k,l=0,\ldots,L$ given by
\begin{align*}
\ohat{\grad D}(\bsm{\la})_l &= m_l - \sum_{j=1}^{J}R_l(X_j)w_j(\bsm{\la})\\
\ohat{\grad^2D}(\bsm{\la})_{k,l} &=-\sum_{j=1}^{J}R_k(X_j)R_l(X_j)w_j(\bsm{\la})
\end{align*}
with weights \added[id=r2,remark={3}]{(see~\eqref{eq:match_KLD})}
\begin{equation*}
w_j(\bsm{\la}) = w_j\cdot\exp\left(\sum_{l=0}^{L}\la_lR_l(X_j)\right)\quad j=1,\ldots,J.\qedhere
\end{equation*}
\end{example}

\begin{example}\label{ex:L2D}
For matching based on L2D the Lagrange multipliers, according to Appendix~\ref{app:newt}, are computed as
\begin{equation*}
\bsm{\la}^{\mrm{new}} = \left(\ohat{\grad^2D}(\bsm{\la}^{\mrm{old}})\right)^{\!-1}\left(\mbf{m} - \ohat{\mbf{m}}(\bsm{\la}^{\mrm{old}})\right),
\end{equation*}
where, with the (non-negative) weights \added[id=r2,remark={3}]{(see~\eqref{eq:match_L2D})}
\begin{equation*}
w_j(\bsm{\la}) = w_j\cdot\max\left\{0,\, \sum_{l=0}^{L}\la_lR_l(X_j)+1\right\},\quad j=1,\ldots,J,
\end{equation*}
and with $\sgn{0}=0$ and $\sgn{c}=1$ for $c>0$, the Hessian estimate is given by
\begin{equation*}
\ohat{\grad^2D}(\bsm{\la})_{k,l}	 =  -\sum_{n=1}^{J}R_k(X_j)R_l(X_j)\,w_j\,\sgn{w_j(\bsm{\la})},\quad k,l=0,\ldots,L.
\end{equation*}
Here $\ohat{\mbf{m}}(\bsm{\la})$ is the Monte Carlo estimate of the moment vector corresponding to the prior density $\prior$ restricted to the set $\{x:\ \sum_{l=1}^L\la_lR_l(x)+1\geq 0\}$, i.e,
\begin{align*}
\ohat{\mbf{m}}(\bsm{\la})_l &= \sum_{j=1}^{J}R_l(X_j)\,w_j\,\sgn{w_j(\bsm{\la})}
\end{align*}
for $l=0,\ldots,L$.
\end{example}

To sum up, as estimator for the matching $\match_L(\mbf{m},\prior)$ we use the vector $\ohat{\bsm{\la}}_{\mbf{m}, \prior} = \ohat{\bsm{\la}}(\mbf{m},\ens{X}^{\prior})$ of $L+1$ Lagrange multipliers obtained from the Newton-Raphson procedure. The iteration is stopped when the gradient of the objective function becomes smaller than a fixed value. This vector depends deterministically on the (stochastic) ensemble $\ens{X}^{\prior}$ we use to discretise the prior distribution and on the (stochastic) vector $\mbf{m}$ of target moments.

\subsection{Resampling}\label{sec:avg_res}
The numerical matching associates new weights with the particles in the ensemble corresponding to the prior distribution. However, as the variance of new weights tends to increase, this may result in a degeneracy of the matched ensemble -- the weights of a few particles may become very large, while the others become small -- leading to large statistical errors in the estimates.
To measure degeneracy we will use the divergence (relative entropy) between the new weights $w_j(\ohat{\bsm{\la}}_{\mbf{m}, \prior})$ and the uniform weights, all equal to $1/J$, (cf.\ \cite[Ch. 6.1.2.1]{LelRouSto2010}) computed as
\begin{align*}
\sum_{j=1}^{J}w_j(\ohat{\bsm{\la}}_{\mbf{m}, \prior})\ln\left(Jw_j(\ohat{\bsm{\la}}_{\mbf{m}, \prior})\right)\in[0,\ln(J)],\quad &\text{for KLD based matching,}\\
\frac{1}{J}\sum_{j=1}^{J}\left(Jw_j(\ohat{\bsm{\la}}_{\mbf{m}, \prior})-1\right)^2\in\left[0,\frac{(J-1)^2}{J}\right],\quad &\text{for L2D based matching.}
\end{align*}
When the divergence of the new weights is larger than a chosen threshold $\al$, we initiate a \emph{resampling algorithm} (see \cite[Ch. 6.1.2.2]{LelRouSto2010}) that generates so-called \emph{branching numbers} -- random integers $n_j,\ j=1,\ldots, J$, representing the particle duplication count -- such that they sum to $J$ and satisfy the unbiasedness condition $\Exp(n_j\,|\, \{w_1(\ohat{\bsm{\la}}_{\mbf{m}, \prior}),\ldots,w_1(\ohat{\bsm{\la}}_{\mbf{m}, \prior})\}) = Jw_j(\ohat{\bsm{\la}}_{\mbf{m}, \prior})$. Thus, after resampling, we obtain a new ensemble of particles $(\otilde{X}_j)_{j=1}^{J}$ with uniform weights $1/J$, in which there are exactly $n_j$ particles equal to $X_j$. In this paper, we employ the \emph{stratified resampling} strategy \cite{HolSchGus2006,DouCapMou2005,MakJouLel2007} that generates random numbers
\[
u_k = \frac{(k-1)+\otilde{u}_k}{J},\quad \otilde{u}_k\sim U[0,1),
\]
and takes $n_j = \#\left\{u_k:\ u_k\in\left[\sum_{i=1}^{j-1}w_i(\ohat{\bsm{\la}}_{\mbf{m}, \prior}),\,\sum_{i=1}^{j}w_i(\ohat{\bsm{\la}}_{\mbf{m}, \prior})\right)\right\}$.

\subsection{Matching failure and adaptive time stepping}\label{sec:match_fail}
During the simulation the distributions may evolve, for some period of time, on time-scales that are similar to those of the macroscopic functions of interest. In that case, when taking a~large macroscopic time step, the extrapolated macroscopic state differs significantly from the last available one and can even fall outside of the domain of the matching operator (cf.~Remark~\ref{rem:match_domain}). Numerical matching of the prior ensemble with such macroscopic state results in a large number of Newton-Raphson iterations or even the lack of convergence. This ''failure" in the matching indicates the need to decrease the extrapolation time step.  Consequently, we set a maximal number of Newton-Raphson iterations (as described in Section~\ref{sec:lagr_mult}) and consider the matching to fail if the optimisation procedure does not reach the desired tolerance within the given number of iterations.  In our experiments, we set the maximal number of Newton-Raphson steps to five.

Based on this observation, we propose the following criterion to adaptively determine the macroscopic step size $\De t$ in the micro-macro acceleration algorithm. If the matching fails, we reject the step and try again with a time step
\begin{equation*}
\De t_{\mrm{new}} = \max(\underline{\al}\De t, K\de t),\quad \underline{\al}<1,
\end{equation*}
whereas, when matching succeeds, we accept the step and propose
\begin{equation*}
\De t_{\mrm{new}} = \min(\oline{\al}\De t,\De t_{\mrm{max}}),\quad \oline{\al}>1
\end{equation*}
for the next step. If the macroscopic step size $\De t_{\mrm{new}}=K\de t$, there is no extrapolation and matching becomes trivial (the identity operator). When this happens, the criterion will ensure that larger time steps are tried after the next burst of microscopic simulation. In the numerical experiments we use $\underline{\al}=0.5$ and $\oline{\al}=1.2$. This choice results in rapid decrease when matching fails and gradual increase when it succeeds.

\section{Numerical experiments: FENE dumbbells}\label{sec:numerics}
For the numerical illustration, we consider the most elementary non-linear kinetic model of a dilute polymer solution -- the so-called Finitely Extensible Non-linear Elastic (FENE) dumbbell model -- where the polymer chain is represented by two beads linked by a spring. In this case, we describe the state of the polymer configuration at time $t\geq 0$ with the end-to-end (random) vector $\mbf{X}(t)\in\mbb{R}^d$ that connects both beads. As the dumbbells move through the solvent, the beads experience Brownian motion, Stokes drag and the spring force that reads
\begin{equation}\label{eq:fene_force}
\mbf{F}\from B(\sqrt{b})\to\mbb{R}^d, \mbf{x} \mapsto \mbf{F}(\mbf{x}) = \frac{b}{b-\|\mbf{x}\|^2}\,\mbf{x},
\end{equation}
where $B(\sqrt{b})=\{\mbf{x}\in\mbb{R}^d:\ \|\mbf{x}\|^2<b\}$,
with $b>0$ a non-dimensional parameter that is related to the maximal polymer length. Note that $\mbf{F}=\grad U(\|\cdot\|)$ where
\begin{equation}\label{eq:fene_pot}
U(r) = -\frac{b}{4}\ln\left(1-\frac{r^2}{b}\right),\quad r\in [0,\sqrt{b}),
\end{equation}
is the FENE spring potential.

The Newtonian contribution of the solvent, modelled with the incompressible Navier-Stokes system, is coupled to the polymer configuration through the stress tensor
\begin{equation}\label{eq:K}
\bsm{\ta} = \frac{1}{\wei}(\Exp\left[\mbf{X}\otimes\mbf{F}(\mbf{X})\right] - \id),
\end{equation}
where $\wei>0$ is the Weissenberg number. We refer to \cite{BriLel2009} for the derivation of the full system and the definition of the non-dimensional number $\wei$. The calculation of polymer stress poses the most demanding task in the simulation of the coupled system, since we need to simulate an ensemble of polymer configurations in each mesh point of the spatial and temporal discretisation. In the presence of a large time-scale separation between the (fast) evolution of individual end-to-end vectors $\mbf{X}$ and the (slow) evolution of the stress tensor~\eqref{eq:K}, the cost of this Monte Carlo simulation may quickly become prohibitive.

In this Section, we consider only the simulation of the microscopic model, leaving the coupling with the Navier-Stokes equations for future work. Thus, we assume that the velocity gradient of the solvent is given by the time-dependent matrix-valued function $\bsm{\ka}\in C(I,\mbb{R}^{d\times d})$. In this case the evolution of dumbbells is modelled using the following SDE
\begin{equation}\label{eq:SDE_FENE}
\der{\mbf{X}(t)} = \left(\bsm{\ka}(t)\mbf{X}(t) - \frac{1}{2\wei}\mbf{F}(\mbf{X}(t))\right)\!\der{t} + \frac{1}{\sqrt{\wei}}\der{\mbf{W}(t)},\quad t\in I=[0,T],
\end{equation}
supplemented with the initial condition
\begin{equation}\label{eq:SDE-IC}
\mbf{X}(0) = \mbf{X}^0\ \ \text{a.s.~with}\ \ \Pr(\|\mbf{X}^0\|^2<b)=1.
\end{equation}
In \cite{JouLel2003,JouLelBri2004} the authors established the existence of the unique global solution to \eqref{eq:SDE_FENE}--\eqref{eq:SDE-IC} in two space dimensions ($d=2$) when the velocity gradient $\bsm{\ka}$ is unidirectional and of particular form $\bsm{\ka}(\cdot)=(\ka_1(\cdot),0)^T$. The trajectorial uniqueness is valid only when $b\geq2$ and is guaranteed by the fact that the solution almost surely does not reach the boundary of $B(\sqrt{b})$. \added[id=r1,remark={5}]{Thus, in this case, Assumption~\ref{ass:sde_exist} is fulfilled with $G=\cl{B}(\sqrt{b})$. Moreover,  due to the compactness of $G$, the class of polynomially bounded observables reduces to the space of bounded continuous functions. Since the FENE model has additive noise with constant intensity $1/\sqrt{\wei}>0$, Assumption~\ref{ass:sde} follows from the elliptic regularity theory~\cite[Ch.~3]{Stroock2008a}. Note also that, as ellipticity is a local property of SDEs~\cite[Ch.~5]{Stroock2008a}, the singularity of the drift does not play a role in this context.}

\added[id=r1,remark={5}]{In the proof of Corollary~\ref{cor:cons}, we rely on the existence and regularity of the densities of SDE~\eqref{eq:sde}.} For FENE SDE~\eqref{eq:SDE_FENE}, the corresponding Fokker-Planck equation~\eqref{eq:FP} reads
\begin{equation}\label{eq:FP_FENE}
\pder[t]\p = -\div[\mbf{x}] \left(\p\,\big[\bsm{\ka}\,\id-\frac{1}{2\wei}\mbf{F}\big]\right) + \frac{1}{2\wei}\lap[\mbf{x}]\p,\quad \text{on}\ (0,T)\times B(\sqrt{b}).
\end{equation}
In \cite{LiuShin2012, HongYang2013} the authors provide requirements on the initial condition $\p_0$ that ensure the well-posedness of the system \eqref{eq:FP_FENE}--\eqref{eq:FP_IC} and regularity of its solutions, which is appropriate for our analysis. In particular, Theorem~1.2 in~\cite{HongYang2013} implies that if $\bsm{\ka}\in C^{\infty}(I)$, $\p_0\in C^{\infty}(\cl{B})$ and
\begin{equation*}
\pder[\mbf{x}]^\al(\p_0e^U) = 0\ \text{at}\ \pder B,\ \text{for all}\ \al\in\mbb{N}^d_{0},
\end{equation*}
there exists a unique smooth solution to \eqref{eq:FP_FENE}--\eqref{eq:FP_IC}.


For the discretisation in time we will use the explicit Euler-Maruyama scheme \eqref{eq:em_step}, combined with an accept-reject (truncation) strategy, that takes the specific form
\begin{equation}\label{eq:em_step_fene}
\otilde{\mbf{X}}^{k+1} = \mbf{X}^{k} + \left(\bsm{\ka}(t^{k})\mbf{X}^{k}-\frac{1}{2\wei}\mbf{F}(\mbf{X}^{k})\right)\de t + \frac{1}{\sqrt{\wei}}\sqrt{\de t}\bsm{\xi}^k,
\end{equation}
where we added a tilde on $\otilde{\mbf{X}}^{k+1}$ to emphasise that this is an intermediate result.
An accept-reject strategy is necessary because~\eqref{eq:em_step_fene} might take the spring length out of the domain of definition, resulting in a random variable $\otilde{\mbf{X}}^{k+1}$ for which $\Pr(\|\otilde{\mbf{X}}^{k+1}\|^2\geq b)>0$.
To avoid this, $\otilde{\mbf{X}}^{k+1}$ is rejected if $\|\otilde{\mbf{X}}^{k+1}\|> \al \sqrt{b}$, with $0<\al\leq1$, and accepted otherwise.  Upon rejection, we repeat the step~\eqref{eq:em_step_fene} using a different Brownian increment $\bsm{\xi}^k$. On the one hand, the parameter $\al$ should be chosen carefully in order to maintain consistency; in particular $\al$ needs to tend to $1$ as $\de t$ tends to zero. On the other hand, with a finite time step, extensions very close to $\sqrt{b}$ produce a~large displacement in the next step, and thus many rejections. In the numerical experiments, following~\cite[Sec.~4.3.2]{Ott1996}, we choose $\alpha = 1-\sqrt{\de t}$.

\begin{remark}[On truncation and other strategies]
The introduction of the cut-off factor $\al\in(0,1)$ results in a truncation of the random variable,
\begin{equation}\label{eq:trunc_step}
\mbf{X}^{k+1} = \otilde{\mbf{X}}^{k+1}{\big|B(\al\sqrt{b})}.
\end{equation}
Here, for any random variable $\mbf{Y}$ with distribution $p_{\mbf{Y}}$ and a Borel set $C\subseteq\mbb{R}^d$ with $p_{\mbf{Y}}(C)>0$, the random variable $\mbf{Y}{|C}$ has \emph{truncated distribution} $p_{\mbf{Y}}(\,\cdot\,|C)$. In particular, if $\mbf{Y}$ has density $f$, the random variable $\mbf{Y}{|C}$ has density equal to $f/p_{\mbf{Y}}(C)$ on $C$ and $0$ otherwise. \added[id=r1,remark={6}]{Besides this simple strategy to preserve the admissible set, other methods exist to devise schemes that more naturally eliminate the unphysical moves. In this direction, let us mention only the implicit schemes~\cite[Ch.~12]{KloPla1999}, in particular~\cite[Sec.~4.3.2]{Ott1996} for FENE, and methods based on the discretisation of the generator of SDE~\cite{Bou-Rabee2015}.}
\end{remark}

\begin{remark}[Matching for FENE dumbbells]
Lemma~\ref{lem:L2N_pos} allow us to define, in particular, the $L^2$ matching operator for prior densities $\pi$ which correspond to the truncated random variables obtained in~\eqref{eq:em_step_fene}-\eqref{eq:trunc_step}. In fact, let $\pi\in L^2(G)$ and for $\al\geq0$ let us denote $G_{\al}=\{\mbf{x}\in G:\ \|\mbf{x}-\bsm{\eta}\|>\al, \forall \bsm{\eta}\in\partial G\}$. Assume that $\pi$ has support in $G_{\al}$ and that $\pi>c$ a.e.~on $G_{\al}$ for some $c>0$. Then, for every $\mbf{m}\in\mbb{R}^L$ close enough to $\res_L\pi$ we can define the matching operator as
\begin{equation}\label{eq:match_l2n_fene}
\match_L(\mbf{m},\pi)=\left\{\begin{array}{cl}
\displaystyle\vph_2(\mbf{m},\pi_{|G_{\al}}) & \text{on}\ G_{\al},\\
0 & \text{on}\ G\bsh G_{\al}.
\end{array}\right.\qedhere
\end{equation}
\end{remark}

\subsection{Properties of numerical matching}\label{sec:results_match}
We consider the FENE dumbbell model \eqref{eq:SDE_FENE} in one space dimension (thus $d=1$ and we do not use bold symbols), with constant velocity gradient $\ka(\cdot)\equiv 2$, Weissenberg number $\wei=1$, and maximal spring length $\ga=\sqrt{b}=7$. We discretise in time with the accept-reject Euler-Maruyama scheme \eqref{eq:em_step_fene}--\eqref{eq:trunc_step} with time step $\de t=2\cdot 10^{-4}$. The probability space is sampled by $J=10^{5}$ independent particles. The initial condition at time $t^0 = 0$ is always the invariant distribution of the same FENE dumbbell model but with zero velocity gradient, i.e., the probability distribution proportional to $\exp(2\wei\,U)$ where $U$ is the FENE spring potential~\eqref{eq:fene_pot}.
We plot the evolution of stress and the PDFs in Figure~\ref{fig:stress_evol}. In this model, we initially have a gradual spread of the PDF. From $t\approx 1.0$, a sharp peak develops close to the maximum polymer length $\ga$. The change in the stress is particularly fast between $t=0.5$ and $t=2.5$, and later it slows down while the PDF approaches equilibrium (see also \cite{Keunings1997}).

\begin{figure}[!htb]
\centering
\includegraphics[width=0.95\textwidth]{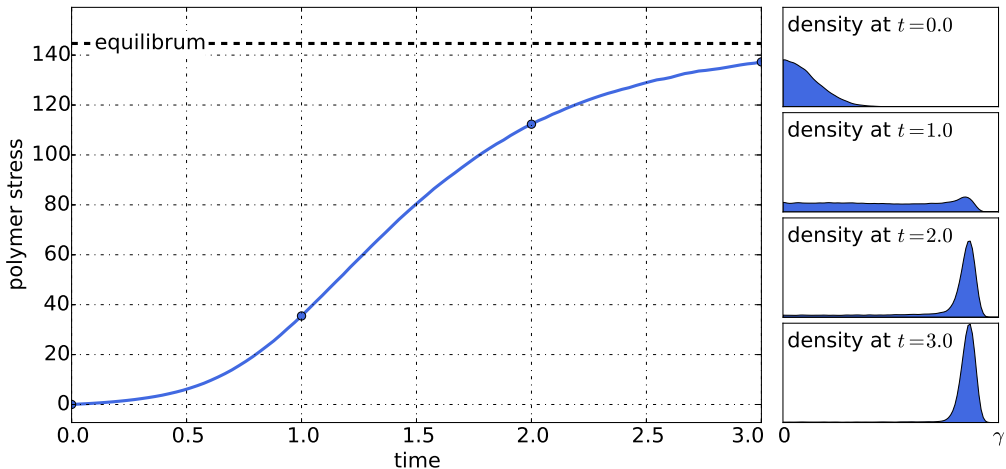}
\caption{Evolution of polymer stress and the profile of empirical PDFs in the FENE dumbbell model with constant velocity gradient $\ka=2.0$ and maximal polymer length $\ga=7.0$. The dashed horizontal line denotes the stress in equilibrium.}
\label{fig:stress_evol}
\end{figure}

In the matching, as macroscopic state vector $\mbf{m}\in\mbb{R}^{L}$, we consider the first $L$ normalized even raw moments, corresponding to the restriction operator $\res_L$ generated by the functions $R_l(x) = (x/\ga)^{2l}$ for $x\in(-\ga,\ga)$ and $l=1,\ldots,L$. Note that Assumption~\ref{ass:hierarchy} is then satisfied. The Newton-Raphson iteration is stopped when the gradient of an appropriate objective function is smaller than $10^{-9}$. To obtain the \emph{empirical PDFs} corresponding to the ensembles calculated during the simulation, we use \emph{kernel density estimation} with Gaussian kernels \cite{Scott2015}.

\subsubsection{Empirical PDFs}\label{sec:epdfs}
In a first experiment, we visually inspect the empirical PDFs obtained from the algorithm for the numerical matching discussed in Section~\ref{sec:lagr_mult}. To this end, we carry out the full microscopic simulation up to time $t^*=1.1$, and we record the corresponding target ensemble $\ens{Y}^*$ and the vector $\mbf{m}^*_L$ of its first $L$ positive even raw moments, with $L=3,5,7$. As a prior $\ens{X}^{\prior}$, we use the ensemble corresponding to the microscopic simulation at time $t_{\prior}=1.0$. Next, we perform the numerical matching with $(\mbf{m}^*_L,\ens{X}^{\prior})$ and record the vectors $\ohat{\bsm{\la}}_{\mbf{m}^*_L, \prior}$ for all considered values of $L$. We compare the empirical PDFs of matchings with the ones obtained from prior and target ensembles in Figure~\ref{fig:match_pdfs}. We focus on the region where the peak of the density in FENE model forms (see Figure~\ref{fig:stress_evol}).

\begin{figure}[!tb]
\centering
\includegraphics[width=\textwidth]{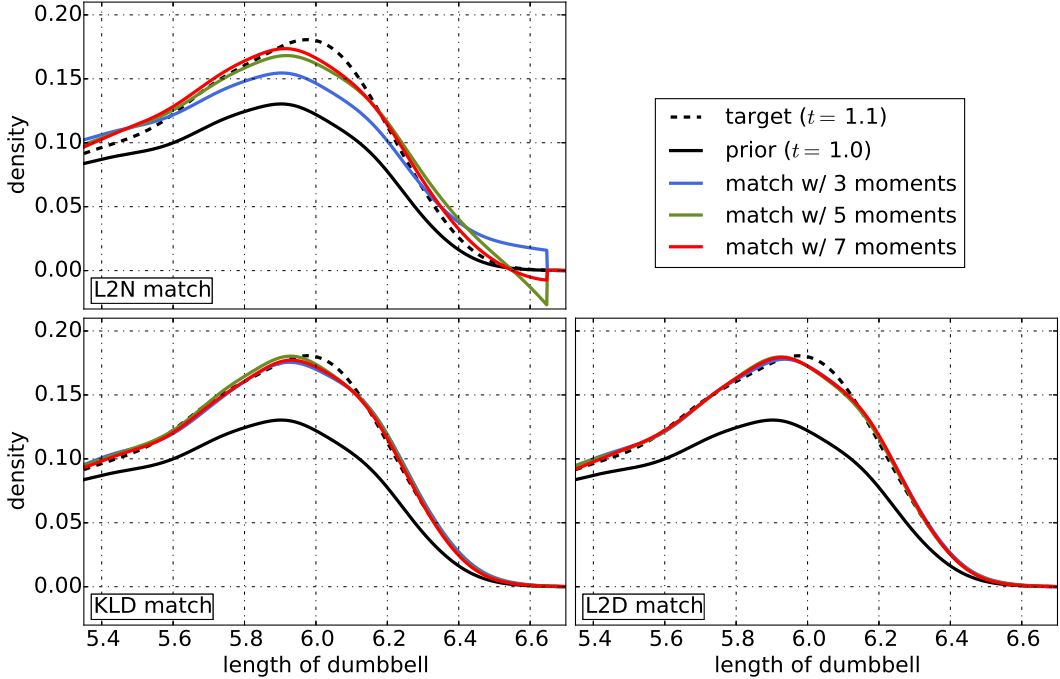}
\caption{Empirical PDFs of L2N (top left), KLD (bottom left) and L2D (bottom right) based matchings with $L$ moments. The matchings are performed with moments of the target distribution at time $t^*=1.1$, and prior distribution at time $t_{\prior}=1.0$, both taken from microscopic simulation.  Computation details are given in Section~\ref{sec:epdfs}.}
\label{fig:match_pdfs}
\end{figure}

For L2N based matching, the densities were obtained as the sum of the empirical PDF of $\prior$, based on the ensemble $\ens{X}^{\prior}$, and the functions $\big(\ohat{\bsm{\la}}_{\mbf{m}_L^*, \prior}\big)^T(R_0,R_1,\ldots,R_L)$, where $R_0\equiv 1$, see equation~\eqref{eq:rem-l2n-moments}. We use formula~\eqref{eq:match_l2n_fene} with $\al = \max_{j}\{|X_j^{\prior}|\}$. The figure visually suggests that increasing the number of moments used for matching makes the approximation of the target PDF more accurate. Moreover, in this example the assumptions of Lemma~\ref{lem:L2N_pos} are not satisfied and the matched density based on L2N can be negative ($L=5,7$), even though it is positive for the matching with smaller number of moments ($L=3$).

The results for the two types of divergence based matchings are presented on the bottom. Here we compute the weights $w_j(\ohat{\bsm{\la}}_{\mbf{m}_L^*, \prior}), j=1,\ldots,J$ as described in Section~\ref{sec:lagr_mult}, and use the resampling algorithm from Section~\ref{sec:avg_res} combined with kernel density estimation to obtain all empirical PDFs. We see that the target density is approximated more accurately than for L2N based matching even with a small number of moments. The density curves are all similar and are almost indistinguishable on the plot.

\subsubsection{Error dependence on the number of moments}\label{sec:err_nb_moms}
In the next experiment, we also simulate up to time $t^*=1.1$ and perform the matching with the moment vector $\mbf{m}^*_L$, corresponding to the first $L$ even raw moments of the target ensemble at $t^*$ with $L=3,5,7$, and the prior taken from the simulation at time $t_{\prior} = 1.0$, thus the matching time step $\De t$ is $0.1=500\cdot\de t$. We compute the relative difference between the $l$-th normalized even raw moment $m^*_l$ of the target and the corresponding moment $m_l$ of the matched ensemble: $|m_l^*-m_l|/m_l^*$, for $l=1,\ldots,20$. We present the averaged results of $100$ i.i.d.~experiments with L2N based matching in the left plot of Figure~\ref{fig:error_moms}, and for both KLD and L2D based matchings in the right plot of Figure~\ref{fig:error_moms}.

First, note that for $l\leq L$, the relative difference in the moments is below the tolerance of the Newton-Raphson procedure, indicating that, as expected, it converged. Second, for $l>L$, the relative error decreases with increasing $L$. Thus, for the matchings considered here, the error of the moments also decreases for moments that are not constrained during matching.

\subsubsection{Error dependence on the matching time step}\label{sec:err_timestep}
In the next experiment, as previously, we simulate up to time $t^*=1.1$. We record the prior ensemble $\ens{X}^{\prior}$ at time $t_{\prior} = 1.0$ and a number of target ensembles for times $t=t_{\prior} + \De t$, with $\De t \in[5\cdot\de t,\, 500\cdot\de t]$, where $\de t = 2\cdot10^{-4}$ is the microscopic time step. We perform matchings with moment vectors $\mbf{m}_L(t)$, corresponding to the first $L = 3,5,7$ normalized even raw moments of the target ensemble at time $t$, and the prior. We compute the stress tensor $\ohat{\ta}_p(t)$ of the matching at time $t$ and record the relative difference $|\ta_p(t)-\ohat{\ta}_p(t)|/\ta_p(t)$, where $\ta_p(t)$ is the stress tensor of the corresponding target.

\begin{figure}[!t]
\centering
\includegraphics[width=\textwidth]{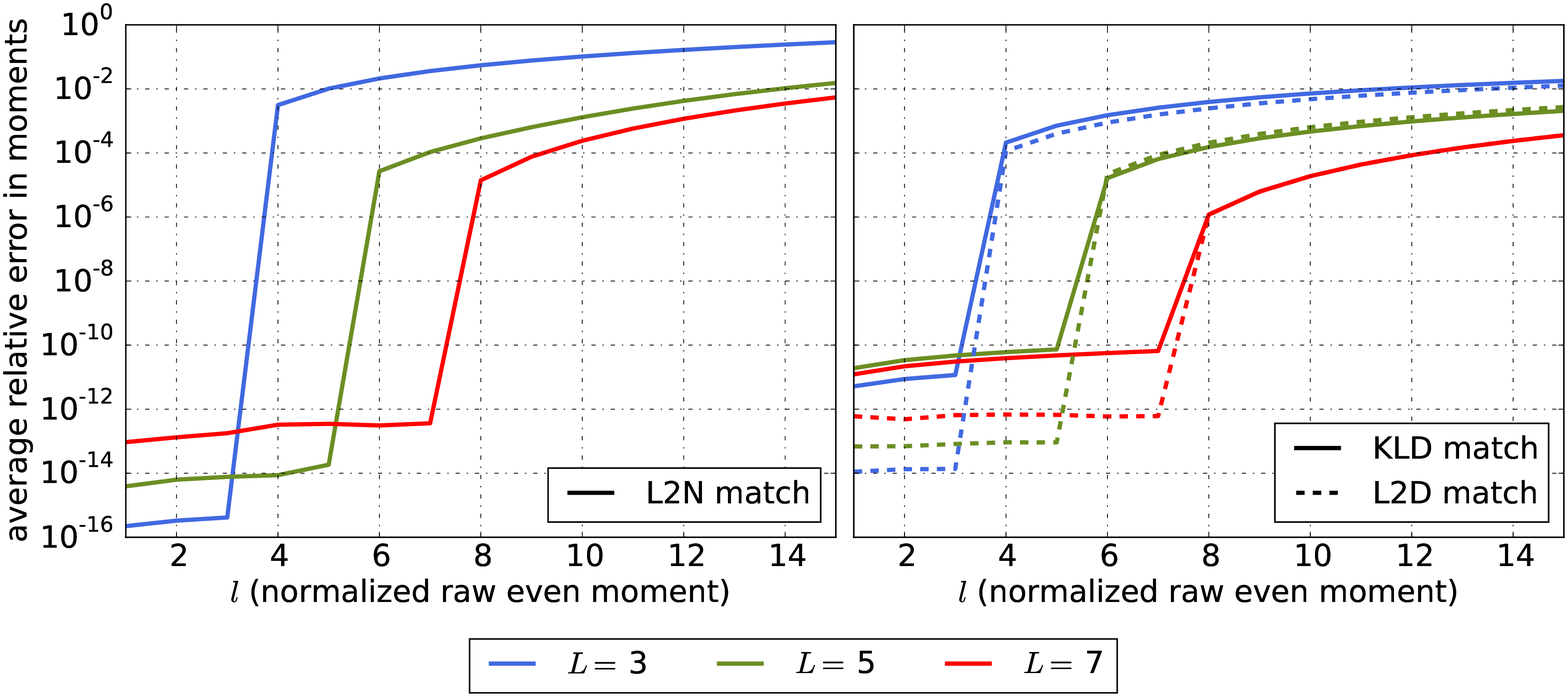}
\caption{Average relative error of $l$-th moment of L2N (left) and KLD, L2D (right) based matching with $L$ moments as a function of $l$ for 100 i.i.d.~runs of the experiment with matching time step $\De t = 0.1$. Computation details are given in Section~\ref{sec:err_nb_moms}.}
\label{fig:error_moms}
\end{figure}

\begin{figure}[t]
\centering
\includegraphics[width=\textwidth]{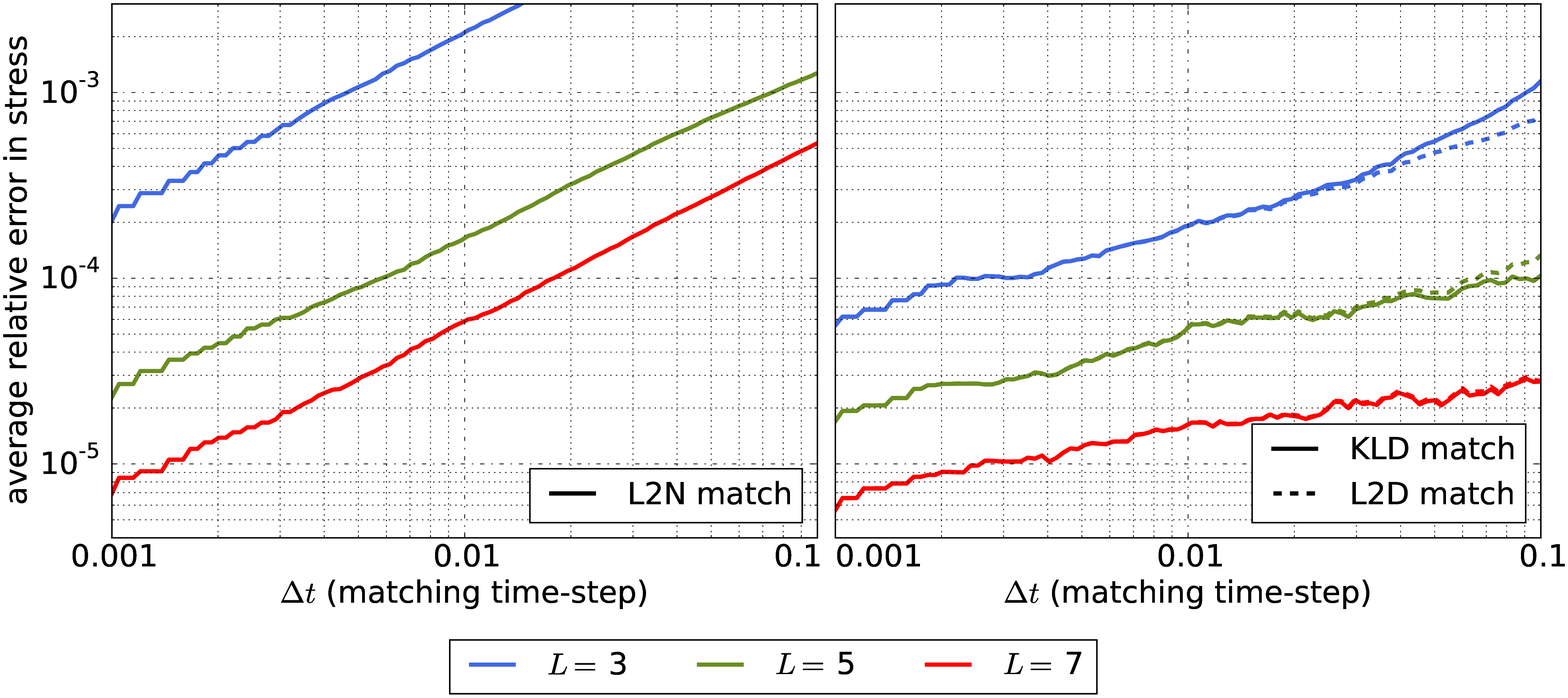}
\caption{Average relative difference in stress as a function of L2N (left) and KLD, L2D (right) based matching time step for 100 i.i.d.~runs of the experiment. Computation details are given in Section~\ref{sec:err_timestep}.}
\label{fig:error_stress}
\end{figure}

The averaged results of this experiment for $100$ i.i.d.~runs are presented in Figure~\ref{fig:error_stress}. We indeed see the linear increase of the matching error as the function of $\De t$. Notice also that, as already observed in two previous experiments, the matching error decreases with increasing $L$. From these two figures we can also see that the two divergence based matchings give the same accuracy for modest $\De t$ and are significantly better than L2N based matching.

\subsubsection{Dependence of Newton-Raphson iteration on the matching time step}
In the last experiment, we compare the performance of the optimisation procedure for the divergence based matchings. To this end, we record the number of iterations of the Newton-Raphson procedure needed in the previous experiment for each matching time step. We compute the average over all runs of the experiment.

For KLD based matching, the average number of iterations increases with increasing matching time step $\De t$ and ranges between $2$ and $4$ in the considered range of time steps. Moreover, the number of iterations increases monotonically with increasing number of moments used. For L2D based matching we didn't observe any such dependence on $\De t$ or the number of moments in this regime. The optimisation procedure converges almost always within only one iteration.


\subsection{Performance of micro-macro acceleration algorithm}\label{sec:results_extrap}

In this Section, we provide numerical results for the full micro-macro acceleration Algorithm~\ref{alg:accel} by performing a simulation of the microscopic FENE dumbbell model with a time-dependent periodic velocity gradient $\ka$. Hence, we consider equation \eqref{eq:SDE_FENE} in one space dimension with Weissenberg number $\wei = 1$ and maximal spring length $\ga=\sqrt{b}=7$. As the velocity field we choose
\begin{equation*}
\ka(t) = 2\cdot\big(1.1 + \sin(\pi t)\big).
\end{equation*}
For the microscopic simulation, we discretise in time using the accept-reject Euler-Maruyama scheme \eqref{eq:em_step_fene}-\eqref{eq:trunc_step}, with time step $\de t= 2\cdot 10^{-4}$, and in probability space with $J=10^{4}$ initial i.i.d.~particles sampled from the invariant distribution of equation \eqref{eq:SDE_FENE} for $\ka\equiv 0$. During the extrapolation, based on the projective forward Euler method (see Section~\ref{sec:extrapolation}), we use the first $L$ normalised even raw moments,  corresponding to the restriction operator $\res_L$ generated by the functions $R_l(x) = (x/\ga)^{2l}$ for $x\in(-\ga,\ga)$ and $l=1,\ldots,L$. In all cases we perform $K=1$ microscopic steps before extrapolation. To restart the microscopic simulation, we use KLD based numerical matching combined with the resampling strategy (described in Section~\ref{sec:avg_res}). We set the degeneracy threshold $\al = \ln(J)/10$ and check the entropy of the weights each $10$ macroscopic time steps. We also use the adaptive time stepping described in Section~\ref{sec:match_fail}.

\subsubsection{Error dependence on the number of moments}
In the first experiment, we use $\De t_{\mrm{max}} = 0.001= 5.0\cdot\de t$ and vary the number $L$ of macroscopic state variables. We simulate up to time $T=6$. The results of $50$ i.i.d.~runs of this experiment with $L=2,3,4$ are presented in Figure~\ref{fig:cpi_stress-evol_moments}. We make two observations.

\begin{figure}[!t]
\centering
\includegraphics[width=\textwidth]{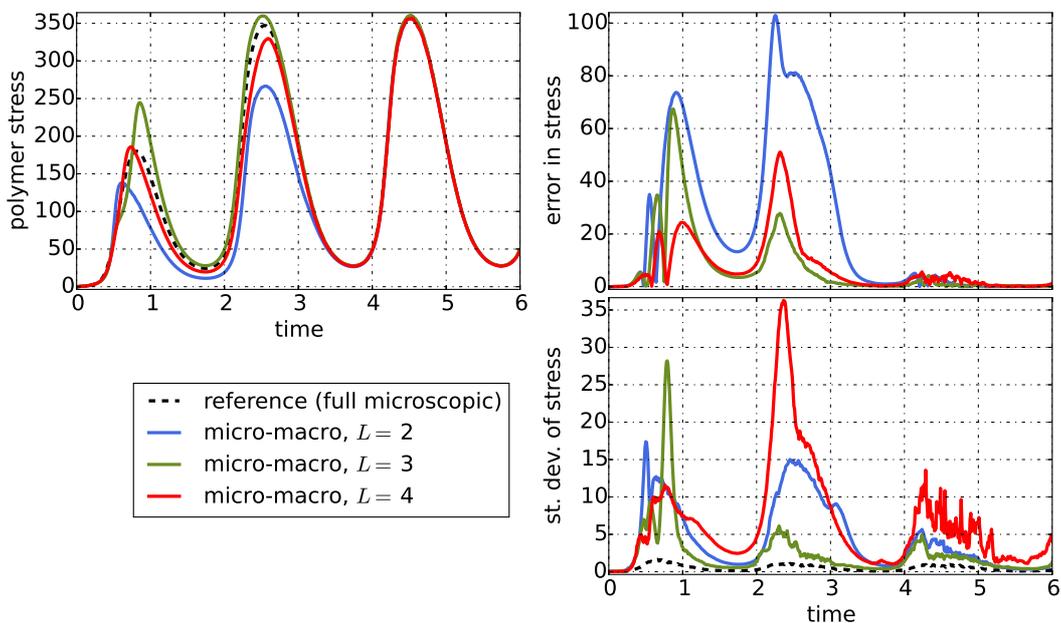}
\caption{Evolution of the average stress, its absolute error with respect to the reference simulation and the standard deviation for micro-macro simulation with $L$ normalised even raw moments and KLD based matching. Results based on 50 i.i.d.~runs of the experiment.}
\label{fig:cpi_stress-evol_moments}
\end{figure}

\begin{figure}[!tb]
\centering
\includegraphics[width=\textwidth]{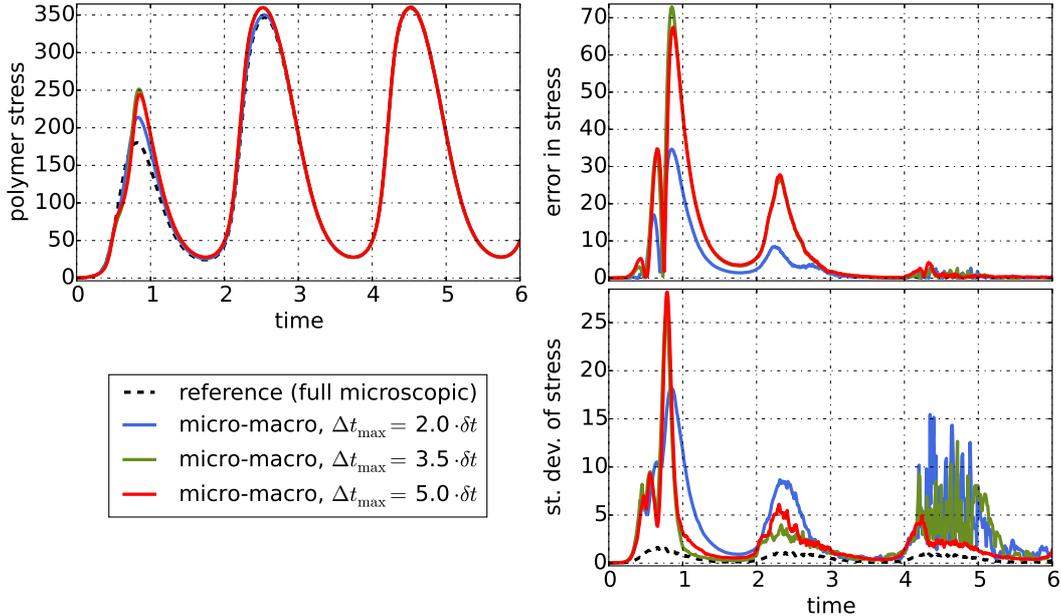}
\caption{Evolution of the average stress, its absolute error with respect to the reference simulation and the standard deviation for micro-macro simulation with $3$ normalised even raw moments and KLD based matching. Results based on 50 i.i.d.~runs of the experiment.}
\label{fig:cpi_stress-evol_timestep}
\end{figure}

First, the deterministic error decreases with increasing $L$, whereas this tendency is not present in the sample standard deviation. The large variability of the standard deviation, especially for $L=4$, stems from the ill-posedness of the Hessian matrices used in the Newton-Raphson procedure. Second, we see that the error of the micro-macro acceleration algorithm with respect to the reference simulation decreases as a function of time, and vanishes when the simulation reaches a periodic regime in the third cycle. This behaviour can be attributed to the fact that the macroscopic behaviour of the system on long time scales is determined by only a few macroscopic state variables.

\subsubsection{Error dependence on the macroscopic time step}
In the second experiment, we fix $L=3$ and vary $\De t_{\mrm{max}}$. Figure~\ref{fig:cpi_stress-evol_timestep} shows the results for $50$ independent runs. While the deterministic error grows with increasing $\De t_{\mrm{max}}$, the adaptive time stepping prevents it from becoming too large and thus we do not see a significant difference between $\De t_{\mrm{max}}=3.5\cdot\de t$ and $5.0\cdot\de t$. The righthand plots illustrate the decrease of the error and sample standard deviation as a function of time. However, for small time steps the variability can still be significant at later times (bottom right plot) due to a large number of matchings that need to be performed and that require solving ill-posed problems. Large $\De t_{\mrm{max}}$ reduces this number, especially for larger times where we observed that extrapolation works the best.

To examine the performance of the method for larger numbers of macroscopic time steps we plot in Figure~\ref{fig:cpi_interr_extrper} the relation between the mean error and the percent of extrapolation in the time range. On the $y$-axis we present the ratio of the average value over time of the error in stress  to the average value over time of stress in a full microscopic simulation. On the $x$-axis we plot the ratio of the time domain covered with extrapolation, computed as $\sum(\De t_{i}-K\de t)$ with sum over all macroscopic steps in the simulation, to the total time range $T=6$. For a given $\De t_{\mrm{max}}$, if all macroscopic time steps were equal to $\De t_{\mrm{max}}$, the percent of extrapolation would be $(m-1)/m$, where $m=\De t/\de t$. For $m=1.5,2.0,2.5$ the $x$-coordinates of corresponding points on the plot are close to this maximum, meaning the simulation was performed with maximal extrapolation step for almost all times. For higher ratios $m$ this is no longer the case and we can see the effects of the adaptive time-stepping that makes the points clump together in the top right corner of the Figure.

\begin{figure}[!tb]
\centering
\includegraphics[width=\textwidth]{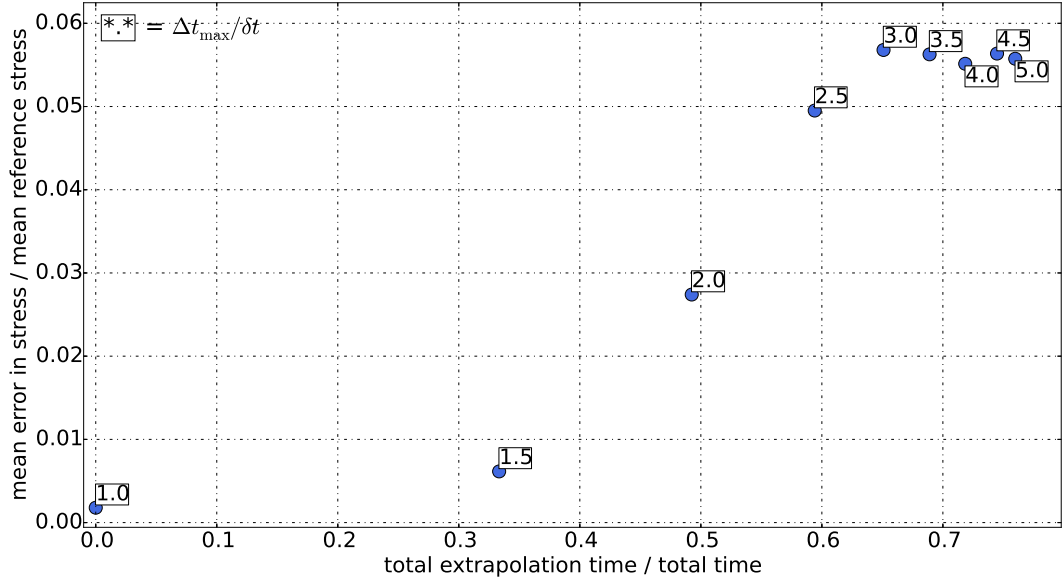}
\caption{Ratio of time-averaged error in stress to time-averaged reference stress vs percent of extrapolation time for different values of macroscopic time step in the micro-macro simulation with $3$ normalised even raw moments. Results based on 50 i.i.d.~runs of the experiment.}
\label{fig:cpi_interr_extrper}
\end{figure}

\section{Conclusions and outlook}\label{sec:concl}

We presented and analysed a micro-macro acceleration technique for the Monte Carlo simulation of stochastic differential equations (SDEs), in which short bursts of simulation using an ensemble of microscopic SDE realisations are combined with an extrapolation of an estimated macroscopic state forward in time.  The method is designed for problems in which the required time step for each realisation of the SDE is small compared to the time scales on which the function of interest evolves. We proved (rigorously) that the proposed procedure converges in the absence of statistical error provided the matching operator satisfies a number of natural conditions. We introduced a matching operator based on the $L2$-norm that satisfies these conditions and illustrated the behaviour of the method numerically, also for matching operators based on $f$-divergences.

In future work, we will perform a more detailed study of the matching operators based on $f$-divergences and investigate stability and propagation of statistical error on long time scales. From an algorithmic point of view, this work raises questions on the adaptive/automatic selection of all method parameters (number of moments to extrapolate, macroscopic time step, number of SDE realisations) to ensure a reliable computation with minimal computational cost.  Also, a numerical comparison with other approaches, such as implicit approximations, could be envisaged.

\section*{Acknowledgements}
This work was partially supported by the Research Council of the University of Leuven through grant OT/13/66, by the Interuniversity Attraction Poles Programme of the Belgian Science Policy Office under grant IUAP/V/22, and by the Research Foundation -- Flanders (FWO -- Vlaanderen) under the grant G.A003.13. The computational resources and services used in this work were provided by the VSC (Flemish Supercomputer Center), funded by the Hercules Foundation and the Flemish Government -- department EWI. The scientific responsibility for this work rests with its authors.

\appendix

\section{Convex optimization with integral functionals}\label{app:optimality}

Our divergence matching operators are defined using an optimisation problem with an integral operator given by the convex function $\f\from[0,+\infty)\to(-\infty, +\infty]$. The \emph{effective domain} of $\f$ is defined as $\edom\f=\{t\in\mbb{R}:\ f( t) < +\infty\}$.
Let $\otilde{\mbf{m}}\in\mbb{R}^{L+1}$ and consider the following \emph{primal (entropy) problem} in $L^p(G,\mu)$
\begin{align*}
&(EP)_p  && \left\{\begin{array}{ll}\medskip
			\text{inf} & \displaystyle \I(\ph)=\int_{G}\f(\ph(\mbf{x}))\,\der{\mu(\mbf{x})}\\\medskip
            \text{subject to} & \otilde{\res}_L\ph = \otilde{\mbf{m}},\\
                              & \ph\in L_+^p(G,\mu),
            \end{array}\right.
\end{align*}
where $\otilde{\res}_L\from L^p(G,\mu)\to\mbb{R}^{L+1}$ is a~linear operator generated by functions $R_l\in L^q(G,\mu)$, $l=0,\ldots,L$ (cf.~\eqref{eq:res_Lp}). The mapping $\I$ is a~well-defined convex functional on $L^p(G,\mu)$ with values in $(-\infty, +\infty]$ as long as $f$ is a~lower semi-continuous proper convex function (see \cite[Lem. 1 and Thm. 1]{Rockafellar1968}). We say that $(EP)_p$ is \emph{consistent} if there exists a~function $\ph\in L_+^p(G,\mu)$ such that $\otilde{\res}_L\ph = \otilde{\mbf{m}}$ and $\I(\ph)<+\infty$.
To compute the primal optimal we consider the corresponding unconstrained (at least formally) \emph{Lagrangian dual problem}
\begin{align*}
&(DEP)_p && \left\{\begin{array}{ll}\medskip
			\text{sup} & \displaystyle D(\bsm{\la})=\otilde{\mbf{m}}^T\bsm{\la} - \int_G\f_+^*\left(\otilde{\res}_L^T\bsm{\la}(\mbf{x})\right)\,\der{\mu(\mbf{x})}\\
            \text{subject to} & \bsm{\la}\in\mbb{R}^{L+1}.
            \end{array}\right.
\end{align*}
Here the function $f_+$, defined as being equal to $f$ on $[0,+\infty)$ and to $+\infty$ on $(-\infty,0)$, encodes the non-negativity constraint from the primal problem. The \emph{(convex) conjugate} $\f_+^*\from\mbb{R}\to\mbb{R}$ is given by
\begin{equation}\label{eq:conv_conj}
\f_+^*(s)=\sup_{t\geq0}\{s\cdot t - \f(t)\}.
\end{equation}
The \emph{transpose} $\otilde{\res}_L^T\from\mbb{R}^{L+1}\to L^q(G,\mu)$ of $\otilde{\res}_L$, defined with relation $\langle\otilde{\res}_L^T\bsm{\la},\,\ph\rangle = (\otilde{\res}_L\ph)^T\cdotp\bsm{\la}$, satisfies
\begin{equation}\label{eq:tran_res_Lp}
\otilde{\res}_L^T\bsm{\la} = \sum_{l=0}^{L}\la_lR_l,\quad \bsm{\la}\in\mbb{R}^{L+1}.
\end{equation}
We call $D$ the \emph{dual objective function} and each solution $\oline{\bsm{\la}}\in\mbb{R}^{L+1}$ to $(DEP)_p$ is referred to as \emph{dual optimal}.

Let us state here the main result we use to analyse the matching operators in Section~\ref{sec:match_oper}. This theorem is a particular case of \cite[Thm. 4.8.]{BorLew1991b}. Recall that a finite set of measurable functions on $G$ is called \emph{pseudo-Haar} if the functions are linearly independent on every non-null subset of $G$. For example, a finite collection of analytic and linearly independent functions on $G$ is pseudo-Haar \cite[Prop. 2.8]{BorLew1991b}.
\begin{theorem}\label{thm:duality}
Let $1\leq p<+\infty$ and suppose that the integrand $\f\from\mbb{R}\to(-\infty, +\infty]$ is lower semi-continuous and strictly convex on $\edom \f\supset[0,+\infty)$, and is superlinear at $+\infty$, i.e. $\lim_{t\to+\infty}f(t)/t=+\infty$. If, in addition, the moment functions $R_l$, $l=1,\ldots,L$, are pseudo-Haar and for $\otilde{\mbf{m}}=(m_0,\ldots,m_L)\in\mbb{R}^{L+1}$ it holds
\begin{equation}\label{eq:CQ}
\otilde{\mbf{m}}\in\inter\otilde{\res}_L\left(L_+^p(G,\mu)\right),
\end{equation}
the primal problem $(EP)_p$ is consistent, the primal and dual optimal values are the same, the dual optimal value is attained, and the unique primal solution is given by
\begin{equation}\label{eq:primal_opt}
\oline{\ph} = (\f_+^*)'\!\left(\sum_{l=0}^{L}\oline{\la}_lR_l\right),
\end{equation}
where $\oline{\bsm{\la}}=(\oline{\la}_0,\ldots,\oline{\la}_L)\in\mbb{R}^{L+1}$ is the dual optimal.
\end{theorem}
See also \cite[Thm. 2.9.]{BorLew1991b} for the explanation how the pseudo-Haar property of moment functions is useful to ensure the consistency of the primal problem. Since $f$ is lower semi-continuous and superlinear at $+\infty$, the mapping $t\mapsto st-f(t)$ is, for each $s$, upper semi-continuous and bounded from above on $[0,+\infty)$. This implies that the supremum in~\eqref{eq:conv_conj} is attained, thus $f_+^*(s)<+\infty$ for all $s$, and in consequence $\edom f_+^*=\mbb{R}$. Moreover, assumptions $(1)$ and $(2)$ guarantee that the conjugate function $f_+^*$ is differentiable on the interior of its domain, see \cite[Thm. 4.6.]{BorLew1991b}, so the right-hand side of \eqref{eq:primal_opt} is well-defined.

As a consequence of \eqref{eq:primal_opt}, we get the formula for the gradient of the dual objective function in $(DEP)_p$
\begin{equation}\label{eq:grad_dof}
\grad D(\bsm{\la}) = \left[\otilde{\mbf{m}} - \otilde{\res}_L\left((f_+^*)'\!\big(\otilde{\res}_L^T\bsm{\la}\big)\right)\right]^{T},\quad \bsm{\la}\in\mbb{R}^{L+1}.
\end{equation}
Hence, due to concavity of $D$, the dual optimal  $\oline{\bsm{\la}}$ can be calculated as the unique vector that satisfies $
\grad D(\oline{\bsm{\la}}) = 0$.

\section{Newton-Raphson procedures for matching}\label{app:newt}

\paragraph{Kullback-Leibler}
We now present the derivation of the Newton-Raphson method to solve the system~\eqref{eq:lm_KLD}. From \eqref{eq:grad_dof} we already know the gradient of the dual optimal function, which in this case is given by
\begin{equation}\label{eq:grad_KLD}
\grad D(\mbf{\la})_l = m_l - \int_GR_l(\mbf{x})\exp\left(\sum_{i=0}^{L}\la_iR_i(\mbf{x})\right)\prior(\mbf{x})\,\der{\mbf{x}},\quad l=0,\ldots,L.
\end{equation}
Hence the $k,l$ component of the Hessian is
\begin{equation}\label{eq:hess_KLD}
\begin{aligned}
\grad^2D(\la)_{k,l}\, &= - \pder[\la_k]\int_GR_l(\mbf{x})\exp\left(\sum_{i=0}^{L}\la_iR_i(\mbf{x})\right)\prior(\mbf{x})\,\der{\mbf{x}}\\
  &= -\int_G R_k(\mbf{x})R_l(\mbf{x})\exp\left(\sum_{i=0}^{L}\la_iR_i(\mbf{x})\right)\prior(\mbf{x})\,\der{\mbf{x}},
\end{aligned}
\end{equation}
for $k,l=0,\ldots,L$.
Since the function $\la_k\mapsto R_l(\mbf{x})\exp\big(\sum_{i=0}^{L}\la_iR_i(\mbf{x})\big)$ is for fixed $\mbf{x}\in G$ continuously differentiable, and for every $\de>0$ the function
\[
\mbf{x}\mapsto  \sup_{\Th\in[-\de,\de]}|R_k(\mbf{x})R_l(\mbf{x})|\exp\left(\sum_{i=0}^{L}\la_iR_i(\mbf{x})+\Th\right)
\]
is integrable with respect to $\prior\der{\mbf{x}}$, the interchangeability of integration and differentiation in~\eqref{eq:hess_KLD} is justified by \cite[Thm. A.5.2.]{Durrett2010}. To sum up, the Newton-Raphson iteration to determine the approximation of dual optimal $\oline{\bsm{\la}}\in\mbb{R}^{L+1}$ for the matching with KLD, i.e. solving \eqref{eq:lm_KLD}, is
\begin{equation}\label{eq:NR_KLD}
\bsm{\la}^{\mrm{new}} = \bsm{\la}^{\mrm{old}} - \left(\grad^2D(\bsm{\la}^{\mrm{old}})\right)^{\!-1}\grad D(\bsm{\la}^{\mrm{old}})
\end{equation}
where the gradient and the Hessian of the objective function are given by \eqref{eq:grad_KLD} and \eqref{eq:hess_KLD} respectively.

\paragraph{$L^2$ divergence}
We now present the derivation of the Newton-Raphson method to solve the system~\eqref{eq:lm_L2D}. Note that we can write
\begin{align*}
\int_GR_l(\mbf{x})\max\left(0,\,\sum_{k=0}^{L}\la_kR_k(\mbf{x})\right)\prior(\mbf{x})\,\der{\mbf{x}} &= \int_{[\bsm{\la}^T\mbf{R}\geq0]}R_l(\mbf{x})\left(\sum_{k=0}^{L}\la_kR_k(\mbf{x})\right)\prior(\mbf{x})\,\der{\mbf{x}}\\
&=  \sum_{k=0}^{L}\la_k\int_{[\bsm{\la}^T\mbf{R}\geq0]}R_l(\mbf{x})R_k(\mbf{x})\,\prior(\mbf{x})\,\der{\mbf{x}},
\end{align*}
where $[\bsm{\la}^T\mbf{R}\geq0]=\{\mbf{x}\in\mbb{R}^d:\ \sum_{k=0}^{L}\la_kR_k(\mbf{x})\geq0\}$. Similar reasoning as in the previous section reveals that the Hessian of the dual objective function is
\begin{equation}\label{eq:hess_L2D}
\grad^2D(\bsm{\la})_{k,l} = -\int_{[\bsm{\la}^T\mbf{R}\geq0]}R_l(\mbf{x})R_k(\mbf{x})\,\prior(\mbf{x})\der{\mbf{x}},\quad k,l=0,\ldots,L,
\end{equation}
thus the gradient is given by
\begin{equation}\label{eq:grad_L2D}
\grad D(\bsm{\la}) = (1,\mbf{m}) + \grad^2D(\bsm{\la})\bsm{\la}.
\end{equation}
The relation between gradient and Hessian of the objective function $D$ in \eqref{eq:grad_L2D} simplifies the Newton-Raphson iteration, originally  given as in $\eqref{eq:NR_KLD}$, to
\begin{equation}\label{eq:NR_L2D}
\bsm{\la}^{\mrm{new}} =  \left(\grad^2D(\bsm{\la}^{\mrm{old}})\right)^{\!-1}(1,\mbf{m}).
\end{equation}

\bibliographystyle{plain}
\bibliography{Articles,Books,refs}

\end{document}

%% file: preamble.tex
\usepackage{amsmath}
\usepackage{amssymb}
\usepackage{amsthm}
\usepackage{mathrsfs} 
\usepackage{mathtools}

\usepackage{thmtools}
\declaretheorem[numberwithin=section]{theorem}
\declaretheorem[sibling=theorem]{lemma}
\declaretheorem[sibling=theorem]{corollary}
\declaretheorem[sibling=theorem]{assumption}

\declaretheorem[style=definition,sibling=theorem]{definition}
\declaretheorem[style=definition,sibling=theorem]{algorithm}
\declaretheorem[style=definition,sibling=theorem]{property}

\declaretheorem[style=remark,qed=$\triangle$,sibling=theorem]{remark}
\declaretheorem[style=remark,qed=$\triangle$,sibling=theorem]{example}

\newcommand{\al}{\alpha}

\newcommand{\ga}{\gamma}

\newcommand{\de}{\delta}
\newcommand{\De}{\Delta}

\newcommand{\Th}{\Theta}

\newcommand{\ka}{\kappa}

\newcommand{\la}{\lambda}

\newcommand{\rh}{\rho}

\newcommand{\Si}{\Sigma}
\newcommand{\ta}{\tau}

\newcommand{\ph}{\phi}
\newcommand{\vph}{\varphi}

\newcommand{\Ps}{\Psi}
\newcommand{\om}{\omega}

\newcommand{\mbf}[1]{\mathbf{#1}} 
\newcommand{\mbb}[1]{\mathbb{#1}} 
\newcommand{\mrm}[1]{\mathrm{#1}} 
\newcommand{\mcl}[1]{\mathcal{#1}} 
\newcommand{\msr}[1]{\mathscr{#1}} 
\newcommand{\bsm}[1]{\boldsymbol#1} 

\newcommand{\R}{\mathbb{R}}

\newcommand{\cl}[1]{\overline{#1}} 
\newcommand{\inter}[1]{\operatorname{int}#1} 
\newcommand{\bdy}[1]{\partial#1} 

\newcommand{\oline}[1]{\overline{#1}}

\newcommand{\otilde}[1]{\widetilde{#1}}
\newcommand{\ohat}[1]{\widehat{#1}}

\newcommand{\linspan}{\operatorname{span}}

\newcommand{\bsh}{\backslash}
\newcommand{\im}{\text{Im}}

\newcommand{\sgn}[1]{\mathrm{sgn}(#1)}
\DeclareMathOperator*{\argmin}{argmin}

\newcommand{\dom}{\operatorname{dom}} 
\newcommand{\edom}{\operatorname{eff\,dom}} 
\newcommand{\img}{\operatorname{im}} 
\newcommand{\supp}{\operatorname{supp}} 
\newcommand{\from}{\colon}

\newcommand{\der}{\operatorname{d}\!}
\newcommand{\pder}[1][]{\partial_{#1}} 
\renewcommand{\div}[1][]{\operatorname{div}_{#1}}
\newcommand{\grad}[1][]{\nabla_{\!#1}}

\newcommand{\lap}[1][]{\Delta_{#1}}

%% file: diagram.tex
\begin{tikzpicture}[scale = 1.0]

\colorlet{lightblue}{blue} 
\definecolor{darkblue}{HTML}{0000FF}
\colorlet{orange}{orange}
\definecolor{dark-gray}{gray}{0}
\definecolor{light-gray}{gray}{0.75}

\newcommand{\partdraw}[3]{
	\usetikzlibrary{calc}
    
	\coordinate (bl) at #1; 
    \coordinate (ur) at #2; 
    
    \foreach \i in {0,...,#3}{
    	\path let \p1=(bl), \p2=(ur) in coordinate (a\i)
        	at ({\x1+rnd*(\x2-\x1)},{\y1+rnd*(\y2-\y1)});
        \fill[orange] (a\i) circle [radius = 1.5pt];
    }
}

\tikzstyle{thickarrow}=[line width=1mm,draw=darkblue,-triangle 45,postaction={draw=darkblue, line width=3mm, shorten >=4mm, -}]
\tikzstyle{thinarrow}=[line width=0.8pt,draw=dark-gray,-triangle 45,postaction={draw=dark-gray, line width=2.5pt, shorten >=2mm, -}]

\draw[very thick,->] (0,0.8) -- (10.5,0.8) node[pos=0.5,below=0.2em] {Time $t$};

\coordinate (tstart) at (1.0,0);
\coordinate (tmicstart) at (2.0,0);
\coordinate (tmicmid1) at (3.0,0);
\coordinate (tmicmid2) at (4.0,0);
\coordinate (tmicend) at (5.0,0);
\coordinate (textend) at (10,0);
\coordinate (tend) at (10.5,0);


\node[fill=white,very thick,draw,rounded corners,inner sep=4pt,above left = 5 and 1 of tmicstart,font=\normalfont,align=center] {Microscopic\\ level};
\node[fill=white,very thick,draw,rounded corners,inner sep=4pt,above left = 1.5 and 1 of tmicstart,font=\normalfont,align=center] {Macroscopic\\ level};

\coordinate[above = 0.9 of tstart] (sm_start);
\coordinate[above = 2.0 of tmicend] (sm_contr);
\coordinate[above = 2.3 of tend] (sm_end);
\draw [light-gray,line width=2pt,postaction={decorate,decoration={text along path,text align=right,raise=-1.9ex,text={|\color{light-gray}|True evolution of macroscopic variables}}}] (sm_start) .. controls (sm_contr) .. (sm_end);

\coordinate[above = 2 of tmicstart] (mac_start);
\coordinate[above = 1.7 of tmicmid1] (mac_mid1);
\coordinate[above = 1.73 of tmicmid2] (mac_mid2);
\coordinate[above = 2.05 of tmicend] (mac_end);
\coordinate[above = 3.5 of textend] (mac_ext);

\draw[very thick, dashed] (mac_mid2) -- node[pos=0.6,fill=white,sloped,draw=darkblue,rounded corners,inner sep=4pt,solid] {\normalfont\textcolor{darkblue}{(iii) Extrapolation}} (mac_ext);

\fill[orange] (mac_start) circle [radius = 3pt];
\fill[orange] (mac_mid1) circle [radius = 3pt];
\fill[orange] (mac_mid2) circle [radius = 3pt];
\fill[orange] (mac_end) circle [radius = 3pt];
\fill[orange] (mac_ext) circle [radius = 3pt];

\coordinate[above = 5 of tmicstart] (mic_start);
\coordinate[above = 5 of tmicend] (mic_end);
\coordinate[above = 5 of tmicmid1] (mic_mid1);
\coordinate[above = 5 of tmicmid2] (mic_mid2);
\coordinate[above = 5 of textend] (mic_ext);

\def\boxheight{0.8};
\def\boxwidth{1.2};

\begin{scope}
\coordinate[left = {\boxwidth/2} of mic_start] (mic_start_s);
\draw[thick] (mic_start_s) rectangle ($(mic_start_s)+(\boxwidth,\boxheight)$);
\clip (mic_start_s) rectangle ($(mic_start_s)+(\boxwidth,\boxheight)$);
\partdraw{(mic_start_s)}{($(mic_start_s)+(\boxwidth,\boxheight)$)}{15};
\end{scope}

\begin{scope}
\coordinate[left = {\boxwidth/2} of mic_end] (mic_end_s);
\draw[thick] (mic_end_s) rectangle ($(mic_end_s)+(\boxwidth,\boxheight)$);
\clip (mic_end_s) rectangle ($(mic_end_s)+(\boxwidth,\boxheight)$);
\partdraw{(mic_end_s)}{($(mic_end_s)+(\boxwidth,\boxheight)$)}{15};
\end{scope}

\begin{scope}
\coordinate[left = {\boxwidth/2} of mic_ext] (mic_ext_s);
\draw[thick] (mic_ext_s) rectangle ($(mic_ext_s)+(\boxwidth,\boxheight)$);
\clip (mic_ext_s) rectangle ($(mic_ext_s)+(\boxwidth,\boxheight)$);
\partdraw{(mic_ext_s)}{($(mic_ext_s)+(\boxwidth,\boxheight)$)}{15};
\end{scope}

\draw[very thick, dashed,shorten <=2em,shorten >=1.9em] ([yshift=1em]mic_start) -- ([yshift=1em]mic_end);
\node[very thick,draw=darkblue,rounded corners,inner sep=4pt,solid,yshift=3.5em] at ($(mic_start)!0.5!(mic_end)$) {\normalfont\textcolor{darkblue}{(i) Simulation}} ;

\draw[thinarrow,shorten >=1mm,shorten <=1mm] (mic_start) -- (mac_start);
\draw[thinarrow,shorten >=1mm,shorten <=1mm] (mic_end) -- (mac_end);
\draw[thinarrow,shorten >=1mm,shorten <=1mm] (mic_mid1) -- (mac_mid1);
\draw[thinarrow,shorten >=1mm,shorten <=1mm] (mic_mid2) -- (mac_mid2);

\node[fill=white,very thick,draw=darkblue,rounded corners,inner sep=4pt,solid,yshift=-3.5em] at ($(mic_start)!0.5!(mic_end)$) {\normalfont\textcolor{darkblue}{(ii) Restriction}};

\draw[thinarrow,shorten <=0.5em] (mac_ext) -- ([yshift=-0.4em]mic_ext); 
\draw[thinarrow,dashed,shorten <=2.2em] ([yshift=1.2em]mic_end) -- ([yshift=1.2em,xshift=-2em]mic_ext);
\node[fill=white,very thick,draw=darkblue,rounded corners,inner sep=4pt,solid] at ([xshift=-5.2em,yshift=-0.6em]mic_ext) {\normalfont\textcolor{darkblue}{(iv) Matching}};

\end{tikzpicture}

%% file: main_arxiv.bbl
\providecommand{\de}[2]{#2}
\begin{thebibliography}{10}

\bibitem{abdulle13msa}
Abdulle Abdulle, Gilles Vilmart, and Konstantinos~C. Zygalakis.
\newblock Mean-square {$A$}-stable diagonally drift-implicit integrators of
  weak second order for stiff {I}t\^o stochastic differential equations.
\newblock {\em BIT}, 53(4):827--840, 2013.

\bibitem{amiri15aco}
Sadegh Amiri and Seyed~Mohammad Hosseini.
\newblock A class of weak second order split-drift stochastic {R}unge-{K}utta
  schemes for stiff {SDE} systems.
\newblock {\em J. Comput. Appl. Math.}, 275:27--43, 2015.

\bibitem{AtkHan2009}
Kendall~E. Atkinson and Weimin Han.
\newblock {\em {Theoretical Numerical Analysis}}, volume~39.
\newblock Springer, 3rd edition, 2009.

\bibitem{BorLew1991b}
Jonathan~M. Borwein and Adrian~S. Lewis.
\newblock {Duality Relationships for Entropy-Like Minimization Problems}.
\newblock {\em SIAM Journal on Control and Optimization}, 29(2):325--338, 1991.

\bibitem{Bou-Rabee2015}
Nawaf Bou-Rabee and Eric Vanden-Eijnden.
\newblock {Continuous-time Random Walks for the Numerical Solution of
  Stochastic Differential Equations}.
\newblock 2015.

\bibitem{Buckwar2010}
Evelyn Buckwar and C{\'{o}}nall Kelly.
\newblock {Towards a systematic linear stability analysis of numerical methods
  for systems of stochastic differential equations}.
\newblock {\em SIAM J. Numer. Anal.}, 48(1):298--321, 2010.

\bibitem{BucRieKlo2011}
Evelyn Buckwar, Martin~G. Riedler, and Peter~E. Kloeden.
\newblock {The numerical stability of stochastic ordinary differential
  equations with additive noise}.
\newblock {\em Stochastic Dynamics}, 11(2 {\&} 3):265--281, 2011.

\bibitem{Caflisch1998}
Russel~E. Caflisch.
\newblock {Monte Carlo and quasi-Monte Carlo methods}.
\newblock {\em Acta Numerica}, 7:1--49, 1998.

\bibitem{Csiszar1967}
Imre Csisz{\'{a}}r.
\newblock {Information-type measures of difference of probability distributions
  and indirect observations}.
\newblock {\em Studia Scientiarum Mathematicarum Hungarica}, 2:299--318, 1967.

\bibitem{PraFra2001}
Giuseppe {Da Prato} and H{\'{e}}l{\`{e}}ne Frankowska.
\newblock {Stochastic viability for compact sets in terms of the distance
  function}.
\newblock {\em Dynamic Systems and Applications}, 10:177--184, 2001.

\bibitem{PraFra2004}
Giuseppe {Da Prato} and H{\'{e}}l{\`{e}}ne Frankowska.
\newblock {Invariance of stochastic control systems with deterministic
  arguments}.
\newblock {\em Journal of Differential Equations}, 200(1):18--52, 2004.

\bibitem{PraFra2007}
Giuseppe {Da Prato} and H{\'{e}}l{\`{e}}ne Frankowska.
\newblock {Stochastic viability of convex sets}.
\newblock {\em Journal of Mathematical Analysis and Applications},
  333(1):151--163, 2007.

\bibitem{debrabant10rkm}
Kristian Debrabant.
\newblock {R}unge-{K}utta methods for third order weak approximation of {SDE}s
  with multidimensional additive noise.
\newblock {\em BIT}, 50(3):541--558, 2010.

\bibitem{debrabant08bao}
Kristian Debrabant and Anne Kv{\ae}rn{\o}.
\newblock B-series analysis of stochastic {R}unge-{K}utta methods that use an
  iterative scheme to compute their internal stage values.
\newblock {\em SIAM J. Numer. Anal.}, 47(1):181--203, 2008/09.

\bibitem{debrabant11bao}
Kristian Debrabant and Anne Kv{\ae}rn{\o}.
\newblock B-series analysis of iterated {T}aylor methods.
\newblock {\em BIT}, 51(3):529--553, 2011.

\bibitem{debrabant09ddi}
Kristian Debrabant and Andreas R\"{o}{\ss}ler.
\newblock Diagonally drift-implicit {R}unge-{K}utta methods of weak order one
  and two for {I}t\^o {SDE}s and stability analysis.
\newblock {\em Appl. Numer. Math.}, 59(3-4):595--607, 2009.

\bibitem{debrabant09foe}
Kristian Debrabant and Andreas R\"{o}{\ss}ler.
\newblock Families of efficient second order {R}unge-{K}utta methods for the
  weak approximation of {I}t\^o stochastic differential equations.
\newblock {\em Appl. Numer. Math.}, 59(3-4):582--594, 2009.

\bibitem{debrabant15ota}
Kristian Debrabant and Andreas R{\"o}{\ss}ler.
\newblock On the acceleration of the multilevel {M}onte {C}arlo method.
\newblock {\em J. Appl. Probab.}, 52(2):307--322, 2015.

\bibitem{DouCapMou2005}
Randal Douc, Olivier Cappe, and Eric Moulines.
\newblock {Comparison of resampling schemes for particle filtering}.
\newblock In {\em ISPA 2005. Proceedings of the 4th International Symposium on
  Image and Signal Processing and Analysis, 2005.}, pages 64--69, 2005.

\bibitem{Durrett2010}
Rick Durrett.
\newblock {\em {Probability: Theory and Examples}}.
\newblock Cambridge Series in Statistical and Probabilistic Mathematics.
  Cambridge University Press, 4th edition, 2010.

\bibitem{EEng03}
W~E and B~Engquist.
\newblock The heterogeneous multi-scale methods.
\newblock {\em Communications in Mathematical Sciences}, 1(1):87--132, 2003.

\bibitem{EEnq2003}
Weinan E and Bjorn Engquist.
\newblock {The heterogeneous multi-scale methods}.
\newblock {\em Communications in Mathematical Sciences}, 1(1):87--132, 2003.

\bibitem{EEnqLiRenVan2007}
Weinan E, Bjorn Engquist, Xiantao Li, Weiqing Ren, and Eric Vanden-Eijnden.
\newblock {Heterogeneous multiscale methods: a review}.
\newblock {\em Communications in Computational Physics}, 2(3):367--450, 2007.

\bibitem{MakJouLel2007}
Mohamed {El Makrini}, Benjamin Jourdain, and Tony Leli{\`{e}}vre.
\newblock {Diffusion Monte Carlo method: numerical analysis in a simple case}.
\newblock {\em ESAIM: Mathematical Modelling and Numerical Analysis},
  41(2):189--213, 2007.

\bibitem{FerRamTic2011}
Augusto Ferrante, Federico Ramponi, and Francesco Ticozzi.
\newblock {On the Convergence of an Efficient Algorithm for Kullback–Leibler
  Approximation of Spectral Densities}.
\newblock {\em IEEE Transactions on Automatic Control}, 56(3):506--515, mar
  2011.

\bibitem{GeaKev2003}
C.~William Gear and Ioannis~G. Kevrekidis.
\newblock {Projective methods for stiff differential equations: problems with
  gaps in their eigenvalue spectrum}.
\newblock {\em SIAM Journal on Scientific Computing}, 24(4):1091--1106, 2003.

\bibitem{GeaKevThe2002}
C.~William Gear, Ioannis~G. Kevrekidis, and Constantinos Theodoropoulos.
\newblock {'Coarse' integration/bifurcation analysis via microscopic
  simulators: micro-Galerkin methods}.
\newblock {\em Computers and Chemical Engineering}, 26(7-8):941--963, 2002.

\bibitem{GeoLin2003}
Tryphon~T. Georgiou and Anders Lindquist.
\newblock {Kullback-Leibler Approximations of Spectral Density Functions}.
\newblock {\em IEEE Trans. on Information Theory}, 49(11):2910--2917, nov 2003.

\bibitem{giles08imm}
Michael~B. Giles.
\newblock Improved multilevel {M}onte {C}arlo convergence using the {M}ilstein
  scheme.
\newblock In {\em Monte {C}arlo and quasi-{M}onte {C}arlo methods 2006}, pages
  343--358. Springer, Berlin, 2008.

\bibitem{giles08mmc}
Michael~B. Giles.
\newblock Multilevel {M}onte {C}arlo path simulation.
\newblock {\em Oper. Res.}, 56(3):607--617, 2008.

\bibitem{HarKalKatPle2016}
Vagelis Harmandaris, Evangelia Kalligiannaki, Markos Katsoulakis, and Petr
  Plech{\'{a}}{\v{c}}.
\newblock {Path-space variational inference for non-equilibrium coarse-grained
  systems}.
\newblock {\em Journal of Computational Physics}, 314:355--383, 2016.

\bibitem{HauLevTit2008}
Cory~D. Hauck, C.~David Levermore, and Andr{\'{e}}~L. Tits.
\newblock {Convex Duality and Entropy-Based Moment Closures: Characterizing
  Degenerate Densities}.
\newblock {\em SIAM Journal on Control and Optimization}, 47(4):1977--2015,
  2008.

\bibitem{HolSchGus2006}
Jeroen~D. Hol, Thomas~B. Sch{\"{o}}n, and Fredrik Gustafsson.
\newblock {On resampling algorithms for particle filters}.
\newblock In {\em NSSPW - Nonlinear Statistical Signal Processing Workshop
  2006}, 2006.

\bibitem{Holmes1975}
Richard~B. Holmes.
\newblock {\em {Geometric Functional Analysis and tis Applications}}, volume~24
  of {\em Graduate Texts in Mathematics}.
\newblock Springer-Verlag, New York Hedelberg Berlin, 1975.

\bibitem{HongYang2013}
Jiaxing Hong and Ge~Yang.
\newblock {On the regularity of solutions to FENE models}.
\newblock {\em SIAM Journal on Mathematical Analysis}, 45(4):2228--2252, 2013.

\bibitem{IlgKarOtt2002}
Patrick Ilg, Iliya~V. Karlin, and Hans~Christian {\"{O}}ttinger.
\newblock {Canonical distribution functions in polymer dynamics . (I). Dilute
  solutions of flexible polymers}.
\newblock {\em Physica A}, 315:367--385, 2002.

\bibitem{JouLel2003}
Benjamin Jourdain and Tony Leli{\`{e}}vre.
\newblock {Mathematical analysis of a stochastic differential equation arising
  in the micro-macro modelling of polymeric fluids}.
\newblock In {\em Probabilistic Methods in Fluids}, pages 205--223, 2003.

\bibitem{JouLelBri2004}
Benjamin Jourdain, Tony Leli{\`{e}}vre, and Claude {Le Bris}.
\newblock {Existence of solution for a micro-macro model of polymeric fluid:
  the FENE model}.
\newblock {\em Journal of Functional Analysis}, 209(1):162--193, apr 2004.

\bibitem{KatPle2013}
Markos~A. Katsoulakis and Petr Plech{\'{a}}{\v{c}}.
\newblock {Information-theoretic tools for parametrized coarse-graining of
  non-equilibrium extended systems}.
\newblock {\em Journal of Chemical Physics}, 139(7):1--14, 2013.

\bibitem{Keunings1997}
Roland Keunings.
\newblock {On the Peterlin approximation for finitely extensible dumbbells}.
\newblock {\em Journal of Non-Newtonian Fluid Mechanics}, 68(1):85--100, jan
  1997.

\bibitem{KevGeaHymKevRunThe2003}
Ioannis~G. Kevrekidis, C.~William Gear, James~M. Hyman, Panagiotis~G.
  Kevrekidis, Olof Runborg, and Constantinos Theodoropoulos.
\newblock {Equation-free, coarse-grained multiscale computation: enabling
  microscopic simulators to perform system-level analysis}.
\newblock {\em Communications in Mathematical Sciences}, 1(4):715--762, 2003.

\bibitem{KevSam2009}
Ioannis~G. Kevrekidis and Giovanni Samaey.
\newblock {Equation-free multiscale computation: algorithms and applications}.
\newblock {\em Annual Review of Physical Chemistry}, 60:321--344, 2009.

\bibitem{KloPla1999}
Peter~E. Kloeden and Eckhard Platen.
\newblock {\em {Numerical Solution of Stochastic Differential Equations}},
  volume~23 of {\em Applications of Mathematics}.
\newblock Springer Berlin Heidelberg, 1999.

\bibitem{komori07wso}
Yoshio Komori.
\newblock Weak second-order stochastic {R}unge--{K}utta methods for
  non-commutative stochastic differential equations.
\newblock {\em J. Comput. Appl. Math.}, 206(1):158--173, 2007.

\bibitem{komori08wfo}
Yoshio Komori.
\newblock Weak first- or second-order implicit {R}unge-{K}utta methods for
  stochastic differential equations with a scalar {W}iener process.
\newblock {\em J. Comput. Appl. Math.}, 217(1):166--179, 2008.

\bibitem{Lafitte:EprintArxiv14046104:2014}
Pauline Lafitte, Annelies Lejon, and Giovanni Samaey.
\newblock eprint arxiv:1404.6104, 4 2014.

\bibitem{Lafitte:EprintArxiv14064305:2014}
Pauline Lafitte, Ward Melis, and Giovanni Samaey.
\newblock eprint arxiv:1406.4305, 6 2014.

\bibitem{LasOtt1993}
Manuel Laso and Hans~Christian {\"{O}}ttinger.
\newblock {Calculation of viscoelastic flow using molecular models: the
  CONNFFEESSIT approach}.
\newblock {\em Journal of Non-Newtonian Fluid Mechanics}, 47:1--20, jun 1993.

\bibitem{BriLel2009}
Claude {Le Bris} and Tony Leli{\`{e}}vre.
\newblock {Multiscale Modelling of Complex Fluids : A Mathematical Initiation}.
\newblock In {\em Multiscale Modeling and Simulation in Science}, pages
  49--137. Springer Berlin Heidelberg, 2009.

\bibitem{Lee:2007p2355}
Steven~L. Lee and C.~William Gear.
\newblock Second-order accurate projective integrators for multiscale problems.
\newblock {\em J. Comput. Appl. Math.}, 201:258--274, Jan 2007.

\bibitem{LelRouSto2010}
Tony Leli{\`{e}}vre, Mathias Rousset, and Gabriel Stolz.
\newblock {\em {Free Energy Computations. A Mathematical Perspective}}.
\newblock Imperial College Press, London, 2010.

\bibitem{LieHalJauKeuLeg1998}
Gregory Lielens, Pierre Halin, Ingrid Jaumain, Roland Keunings, and Vincent
  Legat.
\newblock {New closure approximations for the kinetic theory of finitely
  extensible dumbbells}.
\newblock {\em Journal of Non-Newtonian Fluid Mechanics}, 76(1-3):249--279, apr
  1998.

\bibitem{LiuShin2012}
Hailiang Liu and Jaemin Shin.
\newblock {Global well-posedness for the microscopic FENE model with a sharp
  boundary condition}.
\newblock {\em Journal of Differential Equations}, 252(1):641--662, 2012.

\bibitem{Milstein1995}
Grigori~N. Milstein.
\newblock {\em {Numerical Integration of Stochastic Differential Equations}},
  volume 313 of {\em Mathematics and Its Applications}.
\newblock Springer Netherlands, 1995.

\bibitem{newton94vrf}
Nigel~J. Newton.
\newblock Variance reduction for simulated diffusions.
\newblock {\em SIAM J. Appl. Math.}, 54(6):1780--1805, 1994.

\bibitem{Ott1996}
Hans~Christian {\"{O}}ttinger.
\newblock {\em {Stochastic processes in polymeric fluids}}.
\newblock Springer Berlin Heidelberg, 1996.

\bibitem{Pavliotis2014}
Grigorios~A. Pavliotis.
\newblock {\em {Stochastic Processes and Applications: Diffusion Processes, the
  Fokker-Planck and Langevin Equations}}, volume~60 of {\em Texts in Applied
  Mathematics}.
\newblock Springer Berlin Heidelberg, 2014.

\bibitem{RicoGearKevr04}
Ramiro Rico-Mart{\'\i}nez, C.~William Gear, and Ioannis~G. Kevrekidis.
\newblock Coarse projective kmc integration: forward/reverse initial and
  boundary value problems.
\newblock {\em Journal of Computational Physics}, 196(2):474--489, 2004.

\bibitem{Rockafellar1968}
Ralph~T. Rockafellar.
\newblock {Integrals which are convex functionals}.
\newblock {\em Pacific Journal of Mathematics}, 24(3):525--539, 1968.

\bibitem{roessler06rta}
Andreas R{\"o}{\ss}ler.
\newblock Rooted tree analysis for order conditions of stochastic
  {R}unge--{K}utta methods for the weak approximation of stochastic
  differential equations.
\newblock {\em Stoch. Anal. Appl.}, 24(1):97--134, 2006.

\bibitem{roessler07sor}
Andreas R{\"o}{\ss}ler.
\newblock Second order {R}unge--{K}utta methods for {S}tratonovich stochastic
  differential equations.
\newblock {\em BIT}, 47(3):657--680, 2007.

\bibitem{Rossler2010}
Andreas R{\"{o}}{\ss}ler.
\newblock {Stochastic Taylor Expansions for Functionals of Diffusion
  Processes}.
\newblock {\em Stochastic Analysis and Applications}, 28(3):415--429, 2010.

\bibitem{Saito2008}
Yoshihiro Saito.
\newblock {Stability analysis of numerical methods for stochastic systems with
  additive noise}.
\newblock {\em Review of Economics and Information Studies}, 8(3-4):119--123,
  mar 2008.

\bibitem{SamLelLeg2011}
Giovanni Samaey, Tony Leli{\`{e}}vre, and Vincent Legat.
\newblock {A numerical closure approach for kinetic models of polymeric fluids:
  exploring closure relations for FENE dumbbells}.
\newblock {\em Computers {\&} Fluids}, 43:119--133, 2011.

\bibitem{Scott2015}
Dawid~W. Scott.
\newblock {\em {Multivariate Density Estimation: Theory, Practice, and
  Visualization}}.
\newblock Wiley Series in Probability and Statistics. Wiley, 2nd edition, 2015.

\bibitem{SOMMEIJER:1990p2657}
Ben~P. Sommeijer.
\newblock Increasing the real stability boundary of explicit methods.
\newblock {\em Comput. Math. Appl.}, 19(6):37--49, Jan 1990.

\bibitem{Stroock2008a}
Daniel~W. Stroock.
\newblock {\em {Partial differential equations for probabilists}}, volume 112
  of {\em Cambridge Studies in Advanced Mathematics}.
\newblock Cambridge University Press, 2008.

\bibitem{TalTub1990}
Denis Talay and Luciano Tubaro.
\newblock {Expansion of the global error for numerical schemes solving
  stochastic differential equations}.
\newblock {\em Stochastic Analysis and Applications}, 8(4):483--509, 1990.

\bibitem{Talenti1987}
Giorgio Talenti.
\newblock {Recovering a function from a finite number of moments}.
\newblock {\em Inverse problems}, 3:501--517, 1987.

\bibitem{Vandekerckhove:2009p4623}
Christophe Vandekerckhove, Ioannis~G. Kevrekidis, and Dirk Roose.
\newblock An efficient {N}ewton--{K}rylov implementation of the constrained
  runs scheme for initializing on a slow manifold.
\newblock {\em Journal on Scientific Computing}, 39(2):167--188, 2009.

\bibitem{VanRoo2008}
Christophe Vandekerckhove and Dirk Roose.
\newblock {Accuracy analysis of acceleration schemes for stiff multiscale
  problems}.
\newblock {\em Journal of Computational and Applied Mathematics},
  211(2):181--200, 2008.

\bibitem{Wang:2008}
Han Wang, Kun Li, and Pingwen Zhang.
\newblock Crucial properties of the moment closure model {FENE-QE}.
\newblock {\em Journal of Non-Newtonian Fluid Mechanics}, 150(2-3):80--92,
  2008.

\end{thebibliography}
